\documentclass[11pt]{article}
\usepackage{amsgen,amsmath,amstext,amsbsy,amsopn,amssymb,mathabx,mathrsfs,amsthm,bm,bbm,romannum,enumitem,natbib,xr,multirow}

\usepackage{caption}
\usepackage[dvips]{graphicx}

\usepackage{subcaption}  % f
\usepackage[ruled,vlined,linesnumbered,resetcount]{algorithm2e}
\usepackage{hyperref}
\hypersetup{
	colorlinks=true,
	linkcolor=red,
	urlcolor=blue,
	citecolor=blue
}
\usepackage[depth=2]{bookmark}
\usepackage[margin=1in]{geometry}
\usepackage[utf8]{inputenc}

\usepackage{booktabs}
\usepackage{longtable}
\usepackage{array}
\usepackage{multirow}
\usepackage[table]{xcolor}
\usepackage{wrapfig}
\usepackage{float}
\usepackage{colortbl}
\usepackage{pdflscape}
\usepackage{tabu}
\usepackage{threeparttable}
\usepackage{threeparttablex}
\usepackage[normalem]{ulem}
\usepackage{makecell}
\usepackage{enumitem}

\DeclareMathOperator{\Var}{Var}

\newtheorem{thm}{Theorem}
\newtheorem{lem}[thm]{Lemma}
\newtheorem{prop}[thm]{Proposition}
\newtheorem{coro}[thm]{Corollary}

\theoremstyle{definition}

\newtheorem{asm}{Assumption}

\newtheorem{rmk}{Remark}

\theoremstyle{remark}

\newcommand{\bd}{\bm{d}}
\newcommand{\be}{\bm{e}}

\newcommand{\bt}{\bm{t}}

\newcommand{\bx}{\bm{x}}

\newcommand{\Ab}{\mathbf{A}}
\newcommand{\Bb}{\mathbf{B}}

\newcommand{\Eb}{\mathbf{E}}

\newcommand{\Gb}{\mathbf{G}}

\newcommand{\Ib}{\mathbf{I}}
\newcommand{\Jb}{\mathbf{J}}

\newcommand{\Xb}{\mathbf{X}}
\newcommand{\Yb}{\mathbf{Y}}
\newcommand{\Zb}{\mathbf{Z}}

\newcommand{\bI}{\bm{I}}

\newcommand{\bSigma}{\bm{\Sigma}}

\newcommand{\bbE}{\mathbb{E}}

\newcommand{\bbP}{\mathbb{P}}

\newcommand{\bbR}{\mathbb{R}}

\newcommand{\bbZ}{\mathbb{Z}}

\newcommand{\cE}{\mathcal{E}}

\newcommand{\cG}{\mathcal{G}}

\newcommand{\cN}{\mathcal{N}}

\newcommand{\bzero}{\mathbf{0}}	% zero vector
\newcommand{\bone}{\mathbf{1}}	% all-one vector
\newcommand{\convd}{\stackrel{d}{\longrightarrow}}

\usepackage{authblk}
\title{Consistent Identification of Top-$K$ Nodes in Noisy Networks}
\author{Hui Shen}
\author{Eric D. Kolaczyk}
\affil{McGill University}
\date{\today}

\begin{document}
	\pagenumbering{arabic}
	\maketitle
%	\newpage
%	\pdfbookmark[<level>]{<title>}{<dest>}
%	\pdfbookmark[section]{\contentsname}{toc}
%	\tableofcontents
\begin{abstract}

% outline 
% 1. importance of centrality measures in networks and real applications
% 2. network noise is common
% 3. Top-K consistency condition
% 4. investigation under a number of canonical network models
% 5. both positive and negative results
% 6. confidence set is provided based under noise setting 

Identifying the most influential nodes in a network, typically using centrality measures, is a central task in applied network analysis. However, real-world networks are often constructed from noisy or incomplete data, which can distort rankings and lead to errors in identifying the true top-$k$ nodes. In this paper, we study how network noise affects the recovery of the true top-$k$ node set based on degree centrality. Specifically, we consider a noisy network observation in which edges are randomly added or removed according to a probabilistic noise model, and analyze the resulting empirical top-$k$ set. 
We show that top-$k$ recovery under network noise is governed by the relationship between the degree gap and the noise magnitude, which separates recoverable and unrecoverable regimes. To quantify ranking stability, we derive upper and lower bounds on the expected discrepancy between the empirical and true top-$k$ sets in a general framework and for specific network models. We also extend the analysis to eigenvector centrality, showing that similar noise-gap tradeoffs arise in spectral rankings. Simulation studies support our theoretical findings and illustrate the practical impact of network noise across a range of settings.

\vspace{0.1in}
\noindent\textbf{Keywords:} degree centrality; noisy networks; set recovery; stability analysis; top-$k$ nodes
\end{abstract}

\section{Introduction}\label{sec:intro}
Networks have been a powerful tool for studying complex systems across biology, sociology, and engineering, enabling researchers to uncover interaction patterns and make informed decisions based on network connectivity. A central task in network analysis is the identification of important nodes, typically quantified through \emph{centrality} measures. Centrality plays a key role in applications such as influence maximization, infrastructure monitoring, and information diffusion, and is widely used in empirical studies. 

Because centrality is computed directly from the observed network, it is inherently sensitive to measurement noise, missing edges, and spurious connections. In many applications—ranging from protein–protein interaction networks to social and sensor-based systems—observed graphs are imperfect and represent noisy versions of the underlying interaction structure \citep{newman2018network}.

% P3: 1) empirical robustness work has a long history, 2) statistical/theoretical noisy-network work is more recent and focused on other tasks, 3) ML/GSP also study perturbation robustness, 4) node-level decision problems remain underdeveloped.
Empirical studies of noise and perturbations in network data have a long history across applied domains and the physics literature, including work on error and attack robustness of complex networks \citep{albert2000error, iyer2013attack}. In contrast, theoretical and statistical treatments of noisy network observation are more recent and have focused primarily on prediction and estimation tasks, such as link prediction, vertex classification, and estimation of global network summaries under observation error \citep{priebe2015statistical, zhao2017link, balach2017propagation, le2018estimating, arroyo2021maximum, chang2022estimation}. Related work in machine learning and graph signal processing has also emphasized robustness to graph perturbations, including stability of graph-based learning methods and signal-processing operators under structural errors \citep{xu2021robustness, wang2024uncertainty, dong2020graph, gama2020stability, miettinen2021modelling}. However, comparatively little attention has been paid to node-level decision problems, such as prioritizing, ranking, or selecting a small number of important vertices, where uncertainty directly alters downstream decisions \citep{shen2025minority}.

% P4: 1) Centrality robustness exists, 2) but it studies sensitivity/stability rather than top-k identifiability.
Beyond these task-specific studies, prior work has also examined the robustness of centrality measures under sampling and graph perturbations. These studies show that centrality values and rankings can change substantially under perturbation, with sensitivity varying across different centrality notions \citep{avella2020centrality, borgatti2006robustness, ufimtsev2016understanding, segarra2016stability, cavallaro2024sensitivity}. 
However, this literature focuses on perturbation sensitivity rather than statistical recoverability. In particular, it remains unclear when the empirical top-$k$ set can reliably recover the true top-$k$ set under noisy observations, and when such recovery is impossible.

In this paper, we study the empirical top-$k$ set computed from a noisy observed graph as an estimator of the true top-$k$ set in the underlying network. For degree centrality, we derive recovery conditions, characterize regimes where recovery fails, and obtain finite-sample bounds on the difference between the empirical and true top-$k$ sets. We also examine eigenvector centrality as a secondary global measure of node importance.

Degree centrality depends on local edge counts and is therefore the most tractable setting for developing general theory. Eigenvector centrality, by contrast, depends on the global spectral structure of the graph and is harder to analyze under perturbations. Accordingly, we focus primarily on degree centrality, while using eigenvector centrality as a contrasting global example.

Throughout, $\Ab$ denotes the adjacency matrix of the latent network and $\Yb$ its noisy observation.

\subsection*{A Motivating Example}

While centrality measures provide useful summaries of network structure, their induced rankings can be sensitive to observation noise. To illustrate the top-$k$ recovery problem studied in this paper, we consider a simple simulation based on three widely used network models: the Preferential Attachment (PA) model, the Erd\H{o}s--R\'enyi (ER) model, and the Small-World (SW) model. In each case, we generate a network with adjacency matrix $\Ab$ and then introduce edge noise by randomly adding and removing edges according to \eqref{eqn:noise_model} below.

In this example, we focus on degree centrality, which is the main object of our theoretical analysis. Our aim is to illustrate how noise alters the empirical degree ranking near the top, and how top-$k$ recoverability depends on the structural features of the underlying network.

% Figure 1: Separate Plots for ER, SW, and PA
\begin{figure}[htbp]
    \centering
    \includegraphics[width=16cm, height = 6cm]{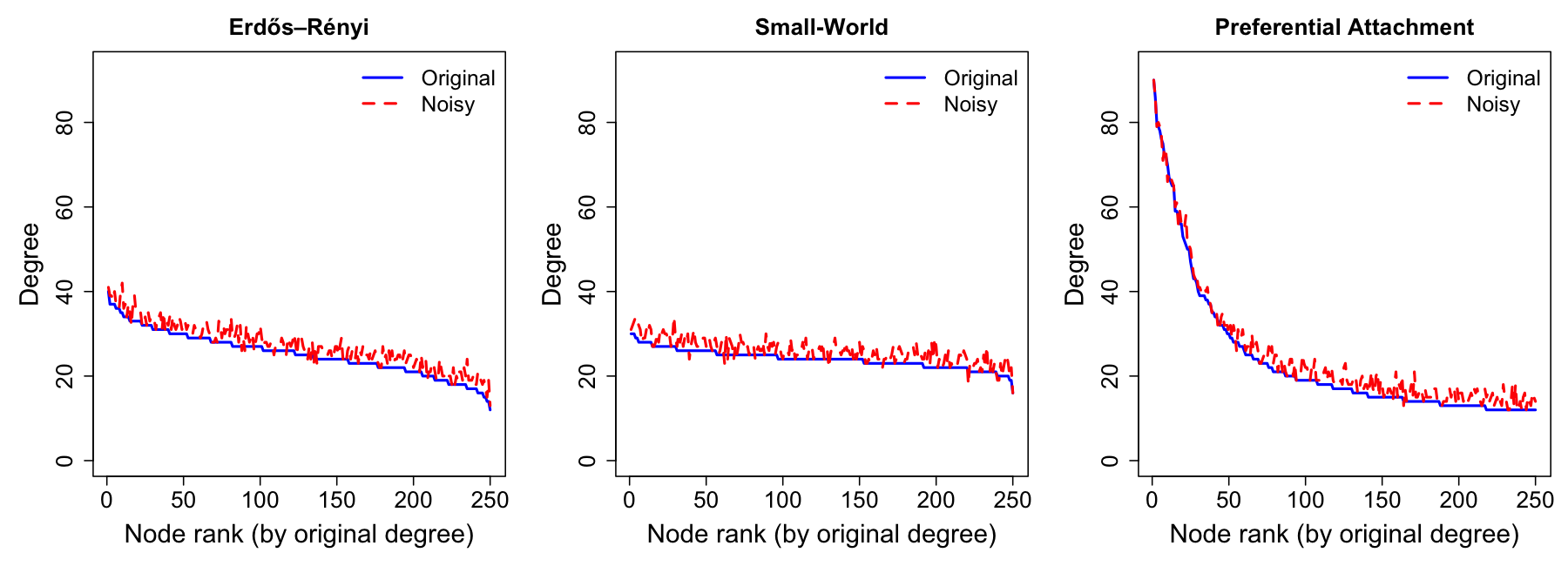}
    \caption{Ordered degree sequences for ER, SW, and PA networks before and after noise, where nodes are ranked by their original degrees ($n=250$, mean degree 25, $\alpha_n=0.01$ for edge addition, and $\beta_n=0.02$ for edge deletion).}
    \label{fig:toy_example}
\end{figure}

Figure~\ref{fig:toy_example} shows the ordered degree sequences before and after noise for the three models with $n=250$ nodes, mean degree $25$, and noise levels $\alpha_n = 0.01$ (edge addition) and $\beta_n = 0.02$ (edge deletion). The three panels exhibit markedly different behavior near the upper end of the degree sequence. In the PA model, the highest-degree nodes remain relatively well separated from the rest of the network after perturbation, suggesting that the top-$k$ set is comparatively stable. By contrast, the ER and SW models display much flatter degree profiles near the top ranks, so even modest perturbations can lead to visible changes in the ordering of the leading nodes. This suggests that the same level of observation noise can have very different consequences for top-$k$ recovery depending on the structure of the underlying graph.

This example highlights two points that motivate our theory. First, even modest noise can substantially distort centrality rankings near the top of the ordering. Second, recoverability depends not only on the amount of noise, but also on structural features of the underlying network, especially the separation among the leading centrality values. 

% Contributions
\paragraph{Contributions.}
We analyze the empirical top-$k$ set computed from a noisy graph as an estimator of the true top-$k$ set in the underlying network. Focusing on degree centrality, we derive recovery guarantees, impossibility results, and finite-sample bounds that show how top-$k$ recoverability depends on the interaction between observation noise and separation in the underlying degree sequence. 

\begin{itemize}
    \item \textbf{Recovery conditions for top-$k$ identification.} We derive conditions under which the empirical top-$k$ set under degree centrality reliably recovers the true top-$k$ set from a noisy graph. These conditions make explicit how successful recovery depends on the noise level and the separation between the relevant underlying degrees.
    \item \textbf{Impossibility of reliable recovery.} We identify regimes in which the true top-$k$ set cannot be consistently recovered from noisy observations. In particular, when the relevant degree separation is too small relative to the noise level, reliable recovery of the true top-$k$ set becomes impossible.
    \item \textbf{Finite-sample bounds on set recovery error.} We derive nonasymptotic upper and lower bounds on the symmetric set difference between the empirical and true top-$k$ sets. These bounds quantify how much ranking error remains in finite samples, including settings where exact top-$k$ recovery fails. 
    \item \textbf{Implications for standard network models.} We specialize the general theory to several commonly used random graph models and show how differences in degree structure lead to different top-$k$ recovery regimes. These results explain the contrasting behaviors observed in simulation and clarify how network topology shapes statistical recoverability. 
\end{itemize}

The remainder of the paper is organized as follows. Section~\ref{sec:theory} develops the theoretical results for degree centrality; Section~\ref{sec:EC} extends the analysis to eigenvector centrality; and Section~\ref{sec:num.res} presents numerical results. Proofs of all results are collected in the Appendix.

\section{Theoretical Results for Degree Centrality}\label{sec:theory}

We begin by defining the problem setting. Let $G=(V,E)$ be a simple undirected graph with vertex set
$V=\{1,\ldots,n\}$ and edge set $E \subseteq \{\{i,j\}: 1\le i<j\le n\}$.
Let $\Ab \in \{0,1\}^{n\times n}$ denote the symmetric adjacency matrix of $G$ with zero diagonal, so that
$A_{ij}=1$ if $\{i,j\}\in E$ and $A_{ij}=0$ otherwise.
In practice, the network is observed with noise, resulting in an observed graph $G^{\mathrm{obs}}$ with adjacency matrix
$\Yb \in \{0,1\}^{n\times n}$.
We adopt the edge-flip noise model of \citet{chang2022estimation}, under which
\begin{equation}\label{eqn:noise_model}
\mathbb{P}\!\left(Y_{ij}=1 \mid A_{ij}=0\right)=\alpha_n,
\qquad
\mathbb{P}\!\left(Y_{ij}=0 \mid A_{ij}=1\right)=\beta_n,
\end{equation}
for all $1\le i<j\le n$.

In this section we focus on degree centrality. For the underlying graph, the (true) degree of node $i$ is 
\[
d_i = \sum_{j \neq i} A_{ij},
\]
while the noisy degree, based on the noisy observation $\Yb$, is given by
\[
\tilde d_i = \sum_{j \neq i} Y_{ij}.
\]
We define the top-$k$ node set based on degree centrality as the $k$ nodes with the largest degrees. 
Without loss of generality, we relabel the vertices so that
\[
d_1 \ge d_2 \ge \cdots \ge d_n.
\]
Under this relabeling, the population top-$k$ set is given by 
\[
S_k := \{1,\ldots,k\},
\]
with ties at the cutoff broken uniformly at random. 
Similarly, the empirical top-$k$ set based on the noisy degrees is
\[
\widetilde S_k := \text{the set of $k$ nodes with the largest values of } \widetilde d_i. 
\]
We treat $\widetilde S_k$ as an estimator of the population-level target $S_k$ and study its recovery properties under noisy edge observations.

Exact top-$k$ recovery fails if there exist indices $i \le k$ and $j > k$ such that \(
\widetilde d_j > \widetilde d_i,
\) 
that is, if a node outside the true top-$k$ set is ranked above a node in the true top-$k$ set under noise. For any such pair $(i,j)$,
\[
\widetilde d_j - \widetilde d_i
=
(d_j - d_i)
+
\big[(\widetilde d_j - d_j) - (\widetilde d_i - d_i)\big].
\]
This suggests that recovery becomes difficult unless the deterministic degree gap $d_i-d_j$ dominates the stochastic fluctuation scale induced by edge noise.

This leads to two sources of instability: local reversals across the top-$k$ boundary and interference from lower-ranked nodes deeper in the degree sequence.
\begin{itemize}
\item[(i)] \textbf{Boundary separation.}
\[
\Delta_k^{\mathrm{bdry}} := d_k - d_{k+1}.
\]
This quantity measures separation at the top-$k$ boundary and controls rank reversals across the top-$k$ boundary near the cutoff.

\item[(ii)] \textbf{Bulk separation.}
For an index $i^*>k$ that separates near-boundary competitors from deeper bulk nodes, define
\[
\Delta_{k,i^*}^{\mathrm{bulk}} := d_k-d_{i^*}.
\]
\end{itemize}

The index $i^*$ partitions the lower-ranked nodes into two groups. Nodes with ranks $k<j<i^*$ are close competitors near the top-$k$ boundary and are the most likely to be ranked above nodes in $S_k$ under noise, while nodes with ranks $j\ge i^*$ lie in the bulk of the degree sequence and contribute only through aggregate interference.

The same distinction explains the logarithmic factors appearing below: if all lower-ranked nodes are treated uniformly as potential competitors, the fluctuation scale is typically of order $\sqrt{\log(n-k)}$, whereas separating local boundary reversals from bulk interference allows the boundary term to be controlled on the smaller $\sqrt{\log k}$ scale.

\subsection{Exact recovery}

We use the vanishing correction terms
$\epsilon_1(m)=\frac{\log\log m}{2\sqrt{2\log m}}$
and
$\epsilon_2(n)=\frac{C(n)}{\sqrt{\log n}}$
(with $C(n)\to\infty$ arbitrarily slowly),
which only affect lower-order terms.

Let
\begin{align*}
\mu_{n,i} &:= \bbE\left(\tilde{d}_i \right) = (n-1-d_{i}) \alpha_n  + d_{i} (1 - \beta_n), \\
\sigma_{n, i}^2 &:= \operatorname{Var} \left(\tilde{d}_i \right) = (n-1-d_{i}) \alpha_n (1 - \alpha_n) + d_{i} \beta_n (1 - \beta_n),   
\end{align*}
denote the mean and variance of the noisy degree $ \tilde{d}_i $, respectively.

\subsubsection{Exact Recovery}

Our first result establishes conditions under which the empirical estimator $\widetilde{S}_k$ coincides with the true top-$k$ set $S_k$. 
All probabilities are taken conditionally on the latent network $\Ab$, treating degree sequences as fixed but unknown.

Fix $\delta\in(0,1)$ and define $
L_{k}:=\log(k/\delta).$
For any $i^*>k$, define
\[
L_{\mathrm{bdry}}:=\log\!\big(k(i^*-k)/\delta\big)
\]
and the boundary variance envelope
\[
\bar{\sigma}_{\mathrm{bdry}}^2
:= \max_{i\le k,\; k<j<i^*}\big(\sigma_{n,i}^2+\sigma_{n,j}^2\big),
\qquad
\bar{\sigma}_{\mathrm{bdry}}:=\sqrt{\bar{\sigma}_{\mathrm{bdry}}^2}.
\]

\begin{thm}[Top-$k$ recovery under boundary and bulk separation]\label{thm:consistency_general}
Assume $\alpha_n + \beta_n \le 1 - c_0$ for some $c_0 \in (0,1)$ and $k = o(n)$. 
Suppose there exists an index $i^*>k$ with $\sigma_{n,i^*} \to \infty$ such that the following conditions hold:
\begin{enumerate}
\item[\textnormal{(i)}] \textbf{Bulk separation.}
\[
d_k - d_{i^*}
\;\ge\;
\frac{1}{1-\alpha_n-\beta_n}
\Big[
(\sqrt{2\log(n - i^* + 1)} - \epsilon_1(n-i^*+1) + \epsilon_2(n))\,\sigma_{n,i^*}
+ \sigma_{n,k}\sqrt{2L_{k}}
+ \tfrac{2}{3}L_{k}
\Big].
\]
\item[\textnormal{(ii)}] \textbf{Boundary separation.}
\[
d_k - d_{k+1}
\;\ge\;
\frac{1}{1-\alpha_n-\beta_n}
\Big[
\sqrt{2L_{\mathrm{bdry}}}\bar{\sigma}_{\mathrm{bdry}}
+ \tfrac{2}{3}L_{\mathrm{bdry}}
\Big].
\]
\end{enumerate}
Then, conditional on the latent network $\Ab$,
\[
\bbP(\widetilde S_k = S_k \mid \Ab) \ge 1 - 2\delta - o(1).
\]
\end{thm}

\begin{rmk}
In many heterogeneous network models, bulk nodes are substantially less competitive for inclusion in the top-$k$ set under moderate noise. As a result, the bulk separation condition can be weaker than the boundary requirement, allowing recovery even when the global degree profile is relatively flat. 
\end{rmk}

Theorem~\ref{thm:consistency_general} exploits both boundary and bulk separation and is therefore sharper than a one-gap analysis based solely on the boundary gap $d_k-d_{k+1}$.

We next state a simpler one-gap sufficient condition depending only on $\Delta_k^{\mathrm{bdry}} = d_k-d_{k+1}$, which is easier to interpret but more restrictive.

\begin{prop}[Boundary-gap sufficient condition for exact recovery]\label{prop:consistency}
Assume that $\alpha_n + \beta_n \leq 1 - c_0$ with $c_0 \in (0,1)$, $\sigma_{n,k+1}\to\infty$ as $n\to\infty$, and $k = o(n)$. Conditional on $\Ab$, the noisy estimator $\widetilde{S}_k$ recovers the true top-$k$ set $S_k$ with probability at least $1 - \delta - o(1)$, i.e.,
$\bbP(S_k = \widetilde{S}_k \mid \Ab) \geq 1 - \delta - o(1)$, provided that
\begin{equation}\label{eqn:cons_cond1}
d_k - d_{k+1}
\;\ge\;
\frac{1}{1-\alpha_n-\beta_n}
\left[
  (\sqrt{2\log(n-k)}-\epsilon_1(n-k)+\epsilon_2(n))\,\sigma_{n,k+1}
  +\sigma_{n,k}\sqrt{2L_k} + \tfrac{2}{3}L_k
\right],
\end{equation}
where $L_k = \log(k/\delta)$.
\end{prop}

\begin{rmk} \label{rmk:thm_consistency}
Here are a few remarks on Proposition~\ref{prop:consistency}. 
\begin{enumerate}
\item 
% interpretation of the condition
{\bf Signal-to-noise ratio}: 
The simplified condition in Proposition~\ref{prop:consistency} can be interpreted in terms of a signal-to-noise ratio. In \eqref{eqn:cons_cond1}, the boundary gap $\Delta_k^{\mathrm{bdry}}$ represents the signal strength, while the noise parameters $\alpha_n$ and $\beta_n$ determine the effective noise level. Specifically, $\sigma_{n,k+1}/(1-\alpha_n - \beta_n)$ can be viewed as the scale of random fluctuations in the noisy degrees. This suggests that, in the boundary-dominated regime, consistent recovery requires  
\[
\operatorname{SNR}
:= \frac{(1-\alpha_n-\beta_n)\Delta_k^{\mathrm{bdry}}}{\sigma_{n,k+1}}
\gtrsim \sqrt{2\log (n-k)}
\quad \text{for fixed } k,
\]
up to lower-order terms. Compared with Theorem~\ref{thm:consistency_general}, this proposition uses a single stronger boundary-gap condition, whereas the theorem exploits both boundary and bulk separation.
\item \textbf{Structural plausibility.}
The condition is most natural in heterogeneous networks (e.g., hub-dominated or heavy-tailed graphs), where the top-$k$ degrees are separated from the bulk.
\item
\textbf{Edge addition vs.\ deletion.}
The effects of edge addition ($\alpha_n$) and deletion ($\beta_n$) on top-$k$ stability follow directly from the variance
\[
\sigma_{n,k+1}^2
=
(n-1-d_{k+1})\alpha_n(1-\alpha_n)
+
d_{k+1}\beta_n(1-\beta_n),
\]
which depends on the underlying degree scale. In sparse networks, the contribution of edge deletion is limited, whereas edge addition can introduce larger fluctuations by creating spurious connections, making edge addition typically more destabilizing than deletion. This theoretical insight is consistent with empirical observations on network robustness reported in \cite{borgatti2006robustness}.
\item 
\textbf{Trade-off between probability and bound tightness}:
The residual probability term $o(1)$ in the success probability $1 - \delta - o(1)$ is of order $\exp(-C(n))$,  where $C(n)$ appears in the deviation parameter $\epsilon_2(n) = C(n) / \sqrt{\log n}$. Hence, increasing $C(n)$ and thus $\epsilon_2(n)$ strengthens the signal condition in \eqref{eqn:cons_cond1},  which reduces the residual probability $o(1)$ and yields a higher recovery probability. 
However, this improvement comes at the cost of a looser upper bound in the sufficient condition, reflecting a trade-off between statistical confidence and signal strength.
\item 
\textbf{Noise regime interpretation}:
The condition $\sigma_{n,k+1}\rightarrow\infty$ places us in a moderate-noise regime where rank reversals are possible. In the complementary sparse-noise regime $\sigma_{n,k+1} = O(1)$, rankings are intrinsically more stable and Poisson approximations could be used instead.
\end{enumerate}
\end{rmk}

\subsubsection{Infeasibility Regime}
To complement the sufficient conditions in Theorem~\ref{thm:consistency_general}, we establish lower bounds showing that top-$k$ recovery fails when the degree separation falls below the stochastic fluctuation scale induced by edge noise.

To describe such weak-separation regimes, we distinguish between two types: boundary separation near the top-$k$ cutoff and bulk separation between the top-$k$ set and lower-ranked nodes.

\paragraph{Critical separation thresholds.}
For $1 \le k < i^* \le n$, define
\[
\delta_{k,i^*}^{\mathrm{bulk}}
:=
\frac{1}{1-\alpha_n-\beta_n}
\Big(
\sqrt{2\log (n-i^*+1)}
- \epsilon_1(n-i^*+1)
- \epsilon_2(n)
\Big)\sigma_{n,i^*},
\]
and
\[
\delta_k^{\mathrm{bdry}}
:=
\frac{c_1}{1-\alpha_n-\beta_n}
2\sqrt{2\log k}\,\max\{ \sigma_{n,k}, \sigma_{n,k+1}\}, 
\]
for some fixed constant $c_1 \in (0,1)$.

For $\delta>0$, define
\[
\mathcal D_{k,i^*}^{\mathrm{bulk}}(\delta)
:=
\left\{
\bd = (d_1,\ldots,d_n):
d_1 \ge \cdots \ge d_n,\;
d_k - d_{i^*} \le \delta
\right\},
\]
and
\[
\mathcal D_k^{\mathrm{bdry}}(\delta)
:=
\left\{
\bd = (d_1,\ldots,d_n):
d_1 \ge \cdots \ge d_n,\;
d_k - d_{k+1} \le \delta
\right\}.
\]

\begin{thm}[Infeasibility of top-$k$ recovery]\label{thm:infeasibility}
Assume $\alpha_n+\beta_n \le 1 - c_0$ for some $c_0\in(0,1)$ and 
$\sigma_{n,i^*}\to\infty$ as $n\to\infty$.
Let $k$ and $i^*$ satisfy
\[
1 \le k < i^* \le n,
\qquad
\log^7(kn) = o\!\left(\min_{1\le i\le k} \sigma_{n,i}^2\right),
\]
where $k$ and $i^*$ may depend on $n$.

If
\[
\Delta_{k,i^*}^{\mathrm{bulk}} \le \delta_{k,i^*}^{\mathrm{bulk}}
\quad\text{or}\quad
\Delta_k^{\mathrm{bdry}} \le \delta_k^{\mathrm{bdry}},
\]
then there exists a constant $c_2>0$ such that
\[
\liminf_{n \to \infty}
\sup_{\Ab:\,\bd(\Ab)\in
\mathcal D^{\mathrm{bulk}}_{k,i^*}(\delta_{k,i^*}^{\mathrm{bulk}})
\cup
\mathcal D^{\mathrm{bdry}}_{k}(\delta_k^{\mathrm{bdry}})
}
\bbP\!\left( S_k \neq \widetilde{S}_k \mid \Ab \right)
\ge c_2.
\]
\end{thm}

\begin{coro}[Boundary-dominated infeasibility condition]\label{coro:infeasibility}
Assume $\alpha_n+\beta_n\le 1-c_0$ with $c_0\in(0,1)$ and $\sigma_{n,k+1}\to\infty$.
Let $k=o(n)$ and suppose
\[
\log^7(kn)=o\!\left(\min_{1\le i\le k}\sigma_{n,i}^2\right).
\]
Define
\[
\bar\delta_k^{\mathrm{bdry}}
:=
\frac{1}{1-\alpha_n-\beta_n}
\Big(
\sqrt{2\log (n-k)}
-\epsilon_1(n-k)-\epsilon_2(n)
\Big)\sigma_{n,k+1}.
\]
Then there exists a constant $c>0$ such that
\[
\liminf_{n\to\infty}
\sup_{\Ab:\,\bd(\Ab)\in \mathcal D_k^{\mathrm{bdry}}(\bar\delta_k^{\mathrm{bdry}})}
\bbP\!\left(S_k\neq \widetilde S_k \mid \Ab\right)
\ge c.
\]
\end{coro}

\begin{rmk}[Near-optimality]
The infeasibility results are near-optimal in both the bulk and boundary regimes when compared with the corresponding sufficient conditions. 
Specifically, the bulk threshold $\delta_{k,i^*}^{\mathrm{bulk}}$ matches the bulk separation condition in Theorem~\ref{thm:consistency_general} at the leading logarithmic scale, while the boundary threshold in Corollary~\ref{coro:infeasibility} matches the one-gap sufficient condition in Proposition~\ref{prop:consistency}, up to lower-order terms. 
Thus the recovery boundary is sharp to first order in both regimes under the present framework.
\end{rmk}

These results establish a sharp phase transition for top-$k$ recovery under noisy degree observations. Recovery occurs with high probability when the boundary and bulk separation exceed the stochastic fluctuation scale induced by edge noise, and fails with probability bounded away from zero when either separation falls below this scale. 

Because the sufficient and infeasible conditions coincide at the leading logarithmic order in both regimes, the recovery boundary is sharp to first order.

\subsection{Consequences for canonical network models}

\subsubsection{Implications for the Preferential Attachment Model}

The preferential attachment (PA) model is a classical random graph model
with pronounced hub structure. A small number of vertices accumulate degrees growing polynomially with the network size, resulting in large separation among the top-ranked nodes. These structural properties make PA networks fundamentally different from homogeneous random graph models and play a crucial role in the stability of degree-based rankings under noise.

In the simplest form of the linear PA model, given a graph $G_t$ at time $t$, a new node $v_{t+1}$ is introduced and connects to an existing node $v \in V_t$ with probability
\begin{equation}\label{eqn:PA_prob}
    \bbP(v_{t+1} \leadsto v \mid G_t)
    = \frac{\deg(v) + b}{\sum_{u \in V_t} (\deg(u) + b)},
\end{equation}
where $\deg(v)$ denotes the degree of node $v$, and $b>-1$ is an offset parameter in the linear attachment function $f(d)=d+b$.  This mechanism induces a power-law degree distribution.

As illustrated in Section~\ref{sec:intro}, the leading hubs in PA networks exhibit remarkable stability under edge perturbations. The following theorem formalizes this robustness for top-$k$ recovery under the independent edge-flip model defined in \eqref{eqn:noise_model}.

\begin{thm}[Hub stability in PA model]\label{thm:PA_joint}
Let $\Ab$ be generated by the linear PA model defined in \eqref{eqn:PA_prob} with parameter $b>-1$. Fix $k\ge 1$. Assume
\[
\alpha_n=o\!\Big(n^{-\frac{b}{2+b}}\tfrac{1}{\log n}\Big)
\qquad\text{and}\qquad
\alpha_n+\beta_n\le 1-c_0
\]
for some constant $c_0\in(0,1)$. Then
\[
\bbP(\widetilde S_k=S_k)=1-o(1).
\]
\end{thm}

\begin{rmk}[Interpretation]\label{rmk:PA_interpretation}
Theorem~\ref{thm:PA_joint} reflects the structural robustness of degree-based ranking in PA networks.
\begin{itemize}
\item \textbf{Scale separation.}
Power-law growth produces large gaps among the top degrees, so the degree separation among leading hubs can dominate the fluctuations induced by edge noise.

\item \textbf{Localization of errors.}
Once attention is restricted to leading hubs, interference from moderate-degree nodes is negligible; the dominant source of error arises near the top-$k$ boundary.

\item \textbf{Asymmetry and tolerance to noise.}
The theorem imposes a vanishing-rate condition on $\alpha_n$, but no comparable decay condition on $\beta_n$ beyond $\alpha_n+\beta_n\le 1-c_0$. Thus, in the present analysis, edge additions are the more critical source of instability. At the same time, the requirement
\[
\alpha_n=o\!\Big(n^{-\frac{b}{2+b}}\tfrac{1}{\log n}\Big)
\]
is relatively mild, showing that exact recovery persists even when the edge-addition noise decays rather slowly with $n$.
\end{itemize}
\end{rmk}

\subsubsection{Infeasibility in Dense Erd\H{o}s--R\'enyi Graphs}\label{sec:infeasibility_ER}
As a contrasting example, we consider the Erd\H{o}s--R\'enyi (ER) model. In $\mathrm{ER}(n,p_n)$, all edges appear independently with probability $p_n$, so node degrees concentrate sharply around $(n-1)p_n$ with fluctuations of order $\sqrt{np_n(1-p_n)}$.

\begin{thm}[Dense ER: impossibility of exact top-$k$ recovery]\label{thm:ER_threshold}
Let $\Ab \sim \mathrm{ER}(n,p_n)$ with $p_n\in[c_1,1-c_1]$ for some constant $c_1\in(0,1/2)$, and fix $k\ge 1$.
Consider the noise model~\eqref{eqn:noise_model} with $\alpha_n+\beta_n\le 1-c_0$ for some $c_0>0$. 
If
\[
\alpha_n \gg \frac{1}{(\log n)^2}
\qquad\text{or}\qquad
\beta_n \gg \frac{1}{(\log n)^2},
\]
then there exists a constant $c\in(0,1)$, independent of $n$, such that for all $n$ large enough,
\[
\bbP\bigl(S_k \neq \widetilde S_k\bigr)\ge 1-c.
\]
\end{thm}

\begin{rmk}[Interpretation and comparison with PA]
In dense ER graphs, degrees fluctuate on the scale $\sqrt{n p_n(1-p_n)}$, while the separation between the largest degrees is much smaller than $\sqrt n$.
As a result, noise levels that are negligible relative to the mean degree, yet large relative to the intrinsic degree gaps, are sufficient to destabilize the top-$k$ ranking.

This behavior contrasts sharply with PA models, where heavy-tailed degree distributions produce hubs with degrees of order $n^{1/(2+b)}$ and macroscopic gaps between top-ranked nodes. In such heterogeneous networks, the degree signal dominates noise, allowing consistent top-$k$ recovery under substantially larger perturbations.
\end{rmk}

\subsection{Deviation of empirical top-$k$ set from the truth}\label{sec:set_deviations}
Exact recovery gives only a binary notion of success. In many noisy regimes, this criterion is too stringent: the empirical top-$k$ set may fail to equal $S_k$ exactly, yet still retain substantial overlap with it. To quantify this approximate recovery behavior, we study the Hamming distance
\[
d_H(S_k,\widetilde S_k):=|S_k\triangle \widetilde S_k|,
\]
and derive upper and lower bounds on its expectation. These bounds provide a general characterization of the reliability of $\widetilde S_k$ in terms of the degree sequence and the noise parameters.

\subsubsection{Lower bound for set recovery error}
We establish a lower bound on the Hamming distance, which quantifies the intrinsic difficulty of recovering the top-$k$ set under noise. The following lemma gives a general model-free bound that depends only on the full degree sequence and the noise parameters $\alpha_n$ and $\beta_n$.

\begin{lem}\label{lem:lower_bound}
Let $t=t(\widetilde d_1,\ldots,\widetilde d_n)$ be any measurable random threshold such that
\[
\widetilde d_{(k+1)} \le t \le \widetilde d_{(k)}
\qquad \text{almost surely}.
\]
Then
\begin{align}
d_H(S_k,\widetilde S_k)
&\ge
2\max\left(
\sum_{i=1}^k \bone\{\widetilde d_i < t\},
\sum_{i=k+1}^n \bone\{\widetilde d_i > t\}
\right). \label{eq:lower-bound-max}
\end{align}
Consequently,
\begin{align}
\frac{1}{2}\,\bbE\!\left[d_H(S_k,\widetilde S_k)\right]
&\ge
\bbE\!\left[
\max\left(
\sum_{i\in S_k}\bone\{\widetilde d_i < t\},
\sum_{i\in S_k^c}\bone\{\widetilde d_i > t\}
\right)
\right] \label{eq:lower-bound-max-exp1}\\
&\ge
\max\left(
\sum_{i\in S_k}\bbP\!\left(\widetilde d_i < t\right),
\sum_{i\in S_k^c}\bbP\!\left(\widetilde d_i > t\right)
\right). \label{eq:lower-bound-max-exp2}
\end{align}
\end{lem}

\begin{rmk}
Here $t=t(\widetilde d_1,\ldots,\widetilde d_n)$ is a measurable function of the noisy degree vector, so the probabilities in \eqref{eq:lower-bound-max-exp2} are taken with respect to the joint distribution of $(\widetilde d_1,\ldots,\widetilde d_n)$.
\end{rmk}

We make two observations about this bound.
\begin{itemize}
    \item 
    \textbf{Dependence on degree distribution.}
    The bound depends on the entire degree sequence, not only on the boundary separation \(\Delta_k^{\mathrm{bdry}}\). Since different network models induce different degree distributions, deriving an explicit universal lower bound is generally infeasible. However, within a given model class or under structural assumptions, the bound becomes more interpretable and tractable. 
    \item
    \textbf{Choice of threshold $t$.}
    To obtain the strongest lower bound in practice, one may optimize over admissible choices of $t$.
\end{itemize}

To illustrate the behavior of the lower bound, we consider a representative homogeneous network model.

\textbf{Example: ER Networks.}
While Theorem~\ref{thm:ER_threshold} rules out \emph{exact} top-$k$ recovery in dense ER graphs, it does not quantify the magnitude of the resulting error. 
We therefore study the expected Hamming distance $d_H(S_k,\widetilde S_k)$ under noise. 
\begin{thm}\label{thm:ER_lower_bound}
Assume \(\Ab\sim \mathrm{ER}(n,p_n)\) with \(p_n\in[c_1,1-c_1]\) for some constant \(c_1\in(0,1/2)\), and let \(k\) be fixed. Assume
\[
\alpha_n \gtrsim \frac{\log n}{n}
\qquad\text{or}\qquad
\beta_n \gtrsim \frac{\log n}{n}.
\]
Let \(q_n:=1-p_n\), let \(C(n)\to\infty\) arbitrarily slowly, and define
\[
c_n(\alpha_n,\beta_n)
:=
\left[q_n-\left(2 p_n q_n \frac{\log n}{n}\right)^{1/2}\right]\alpha_n(1-\alpha_n)
+
\left[p_n-\left(2 p_n q_n \frac{\log n}{n}\right)^{1/2}\right]\beta_n(1-\beta_n).
\]
Then
\begin{align}
\frac{1}{2}\,\bbE\!\left[d_H(S_k,\widetilde S_k)\right]
&\ge
k\,
\Phi\!\left(
-\frac{2C(n)}{\sqrt{c_n(\alpha_n,\beta_n)}\,\sqrt{\log n}}
\right)
-o(1).
\label{eq:ER-lb-simple}
\end{align}
\end{thm}

These bounds show that under moderate noise, the expected number of misidentified top-$k$ nodes remains bounded away from zero. In particular, if $\alpha_n \gg 1/\log n$ or $\beta_n \gg 1/\log n$, then
\[
\frac{1}{2}\bbE\!\left[d_H(S_k,\widetilde S_k)\right]
\ge
\frac{k}{2}-o(1).
\]
Equivalently,
\[
\liminf_{n\to\infty}\frac{1}{2}\bbE\!\left[d_H(S_k,\widetilde S_k)\right]\ge \frac{k}{2}.
\]
This indicates that degree-based top-$k$ recovery loses meaningful discriminative power in homogeneous networks.

\subsubsection{Upper bound for set recovery error}
We now derive an upper bound on the expected Hamming distance, which quantifies the achievable accuracy of top-$k$ recovery under noise. 
\begin{lem}\label{lem:upper_bound}
Let $t=t(\widetilde d_1,\ldots,\widetilde d_n)$ be any measurable random threshold such that
\[
\widetilde d_{(k+1)} \le t \le \widetilde d_{(k)}
\qquad \text{almost surely}.
\]
Then
\begin{align}
d_H(S_k,\widetilde S_k)
&\le
2\min\left(
\sum_{i=1}^k \bone\{\widetilde d_i \le t\},
\sum_{i=k+1}^n \bone\{\widetilde d_i \ge t\}
\right). \label{eq:upper-bound-min}
\end{align}
Consequently,
\begin{align}
\frac{1}{2}\,\bbE\!\left[d_H(S_k,\widetilde S_k)\right]
\;\le\;
\min\left(
\sum_{i\in S_k}\bbP(\widetilde d_i \le t),
\sum_{i\in S_k^c}\bbP(\widetilde d_i \ge t)
\right). \label{eq:upper-bound-min-exp}
\end{align}
\end{lem}

We next illustrate this upper bound under a heterogeneous network model.

\textbf{Example: Preferential attachment networks.} We consider the PA model with a power-law degree distribution, where the expected set recovery error depends not only on boundary separation, but also on the full configuration of leading degrees.
\begin{thm} \label{thm:PA_upper_bound}
Suppose the network $\Ab$ is generated by a linear PA model with offset parameter $b>-1$, and let $k$ be fixed. Assume
\(
\alpha_n + \beta_n \le 1-c_0
\) 
for some constant $c_0>0$, and additionally
\[
\alpha_n \lesssim \frac{n^{-\frac{b}{2+b}}}{\log n}.
\]
Then
\[
\frac{1}{2} \bbE\!\left( d_H(S_k, \widetilde{S}_k) \right) = o(1),
\]
where the expectation is taken over both the random network and the noise model.
\end{thm}

This result complements Theorem~\ref{thm:PA_joint} by showing that, under the same structural regime, the expected Hamming distance also vanishes.

\section{Beyond Degree Centrality: Eigenvector Centrality}\label{sec:EC}

Eigenvector centrality is a popular importance measure in network analysis, as it reflects not only the number of neighbors a node has but also their influence. However, this global dependence makes eigenvector centrality more sensitive to edge noise than degree centrality. In this section, we quantify this sensitivity under the model~\eqref{eqn:noise_model} and provide recovery guarantees under suitable spectral and separation conditions.

For a network with adjacency matrix $\Ab \in \mathbb{R}^{n \times n}$, the eigenvector centrality $x_i$ of node $i$ is defined through the principal eigenvector $\bx$ of $\Ab$:
\begin{equation}
\Ab \bx = \lambda_1 \bx,
\qquad
x_i = \frac{1}{\lambda_1}\sum_{j=1}^n A_{ij} x_j.
\end{equation}
Here, $\lambda_1$ denotes the largest eigenvalue of $\Ab$.

For the noisy network $\Yb$, the corresponding centrality vector $\tilde{\bx}$ satisfies
\begin{equation}
\Yb \tilde{\bx} = \tilde{\lambda}_1 \tilde{\bx},
\end{equation}
where $\tilde{\lambda}_1$ is the largest eigenvalue of $\Yb$. We choose the signs of $\bx$ and $\tilde{\bx}$ so that their entries are nonnegative and normalize the eigenvectors so that $\|\bx\|_2 = \|\tilde{\bx}\|_2 = 1$.

Following the degree-based setup, we define the top-$k$ set under eigenvector centrality as the $k$ nodes with the largest entries of the principal eigenvector. Let
\[
x_{(1)} \ge x_{(2)} \ge \cdots \ge x_{(n)}
\qquad\text{and}\qquad
\tilde{x}_{(1)} \ge \tilde{x}_{(2)} \ge \cdots \ge \tilde{x}_{(n)}
\]
denote the ordered entries of $\bx$ and $\tilde{\bx}$, respectively. The true and observed top-$k$ sets under eigenvector centrality are denoted by
\[
S_k^\lambda
\qquad\text{and}\qquad
\widetilde S_k^\lambda,
\]
respectively, where ties at the cutoff are broken uniformly at random.

Our goal is to characterize when $\widetilde{S}_k^\lambda$ recovers $S_k^\lambda$. To quantify the separation at the decision boundary, we define
\[
\Delta_k^\lambda := x_{(k)} - x_{(k+1)}.
\]
To state the perturbation bound explicitly, we introduce three auxiliary quantities that summarize the dominant noise terms in the analysis: 
\begin{equation}\label{eq:def-B123}
\begin{aligned}
B_1 &:= (\alpha_n+\beta_n)\sqrt{n}
       + 5\sqrt{\big[(\alpha_n+\beta_n) - (\alpha_n-\beta_n)^2\big]\log n}, \\
B_2 &:= \sqrt{2n(\alpha_n+\beta_n)+\log n}, \\  
B_3 &:= 5\sqrt{n(\alpha_n + \beta_n)} + \alpha_n n + (\alpha_n + \beta_n)\,\|\Ab\|.
\end{aligned}
\end{equation}
Their precise roles are made explicit in the proof of Theorem~\ref{thm:evec-infty-perturb}. 

Let $\lambda_1 \ge \lambda_2 \ge \cdots \ge \lambda_n$ denote the eigenvalues of $\Ab$. Combining these quantities, we define the explicit entrywise error bound
\begin{equation}\label{eq:def-eps-evec}
\epsilon_n
:=
\frac{\lambda_1 - \lambda_2 - 2B_3 - 2B_2}
     {\lambda_1 - \lambda_2 - 2B_3 - 4B_2}
\Bigg(\frac{\lambda_2}{\lambda_1} + \|\bx\|_{\infty}\Bigg)
\Bigg(
    \frac{2B_3}{\lambda_1 - \lambda_2}
    + \frac{B_3}{\lambda_1 - B_3}
\Bigg)
+
\frac{
    2B_1 + 2B_2\|\bx\|_{\infty}
}{
    \lambda_1 - \lambda_2 - 2B_3 - 4B_2
}.
\end{equation}

\begin{thm}[Entrywise perturbation bound for eigenvector centrality]
\label{thm:evec-infty-perturb}
Assume that $\lambda_1 - \lambda_2 > 2B_3 + 4B_2$.
Then with probability $1-o(1)$,
\[
\|\bx - \tilde{\bx}\|_{\infty} \le \epsilon_n.
\]
\end{thm}

\begin{coro}[Top-$k$ recovery under eigenvector centrality]
\label{cor:topk-consistency}
Under the assumptions of Theorem~\ref{thm:evec-infty-perturb}, 
if $\Delta_k^\lambda > 2\epsilon_n$ for some $1 \le k < n$, 
then with probability $1-o(1)$,
\[
S_k^\lambda = \widetilde{S}_k^\lambda.
\]
\end{coro}

\begin{asm}[Hub-dominated localization]
\label{asm:localization}
Consider a network $G$ with adjacency matrix $\Ab$ and principal eigenvector $\bx$, normalized in $\ell_2$.
Let $h$ denote the unique maximum-degree vertex in $G$, and let $
N(h):=\{i\in [n]: i\sim h\}$ 
be its neighbor set.
We assume that the network lies in a \emph{hub-dominated regime} such that:
\begin{enumerate}
\item $x_h \to c_0 > 0$ as $n\to\infty$,
      i.e., the hub retains a non-vanishing centrality score;
\item $\sum_{i \notin \{h\}\cup N(h)} x_i^2 \rightarrow 0,$ i.e., the eigenvector mass concentrates on the hub and its direct neighbors.
\end{enumerate}
\end{asm}
\begin{asm}[Hub separation]
\label{asm:hub-separation}
Under Assumption~\ref{asm:localization}, assume further that there exists a constant $c_1>0$ such that
\(
x_h - \max_{i\in N(h)} x_i \ge c_1
\)
with probability \(1-o(1)\) under the PA model.
\end{asm}

\begin{rmk}[Interpretation and scope]
Assumptions~\ref{asm:localization}--\ref{asm:hub-separation} characterize a
\emph{hub-dominated} network regime in which a single vertex exhibits
overwhelming structural importance.
In such networks, the leading eigenvector is localized around the hub and its immediate neighborhood,
and the hub entry remains separated from those of its neighbors by a non-vanishing gap.
Accordingly, the consistency result below applies most naturally to top-$1$ recovery under eigenvector centrality.
\end{rmk}

\begin{thm}\label{thm:EC_consistency}
Consider $\Ab$ from a linear PA model with attachment function $f(d)=d+b$, where $b>-1$. Assume Assumptions~\ref{asm:localization} and~\ref{asm:hub-separation}. Let $\Yb$ be the noisy observation under model~\eqref{eqn:noise_model}. If
\[
\alpha_n \ll n^{-\frac{b+3/2}{b+2}}
\quad\text{and}\quad
\beta_n \ll n^{-\frac{b+1}{b+2}},
\]
then
\[
\bbP(S_1^\lambda = \widetilde{S}_1^\lambda) = 1 - o(1).
\]
\end{thm}

\begin{rmk}
Theorem~\ref{thm:EC_consistency} is restricted to $k=1$. For $k\ge2$, the gap condition in Corollary~\ref{cor:topk-consistency} is typically much harder to verify in PA networks, since the principal eigenvector is localized around the leading hub and the lower-ranked entries are much less separated.
\end{rmk}

\begin{rmk}
Comparing the stability of eigenvector and degree centrality under PA models, we make two observations.
\begin{itemize}
    \item 
    Edge missingness, controlled by $\beta_n$, appears to have a stronger effect on the stability of eigenvector centrality than on degree centrality.
    \item 
    The admissible range of $\alpha_n$ for stable eigenvector-based ranking is strictly smaller than that for degree-based ranking.
\end{itemize}
\end{rmk}

\section{Numerical Results}
\label{sec:num.res}

\subsection{Simulation study for top-$k$ recovery under edge noise}
In this section, we assess the informativeness of the upper and lower bounds derived in Section~\ref{sec:set_deviations} for the set difference $d_H(S_k, \widetilde{S}_k)$ under various network models, by evaluating how well these bounds capture the empirical behavior of top-$k$ instability under noise.

\subsubsection{Erd\H{o}s--R\'enyi Networks}

Assuming the unobserved network $\Ab$ is generated from an ER graph $\mathcal{G}(n,p_n)$, we conduct simulation studies under three different settings to investigate how network size, density, noise levels, and the top-$k$ parameter affect the recovery of the top-$k$ set under edge noise. 

\begin{itemize}

\item \textbf{Setting 1 (Varying $n$).}
Number of nodes $n$ increases from 200 to 2000 over a discrete grid;
fixed parameters: $p = 0.25$, $k = 5$, $\alpha_n = 0.05$, $\beta_n = 0.05$.

\item \textbf{Setting 2 (Varying $\alpha_n$).}
False positive rate $\alpha_n \in \{0.01, 0.02, 0.03, 0.04, 0.05\}$;
fixed parameters: $n = 1000$, $p = 0.25$, $k = 5$, $\beta_n = 0.05$.

\item \textbf{Setting 3 (Noise scaling with $n$).}
To examine the scaling behavior predicted by Theorem~\ref{thm:ER_threshold},
we consider a regime in which the noise level decays logarithmically with the network size.
Specifically, we let $
n \in \{200, 300, 400, 500, 600, 800, 1000\},$
$ p = 0.25, k = 5$,
and set $\alpha_n = \frac{1}{2 \log n}, \beta_n = \frac{1}{2 \log n}.$

This setting probes the transition into the impossibility regime for exact top-$k$ recovery described in Theorem~\ref{thm:ER_threshold}.

\end{itemize}

\begin{figure}[tbp]
    \centering
    \includegraphics[width=\textwidth]{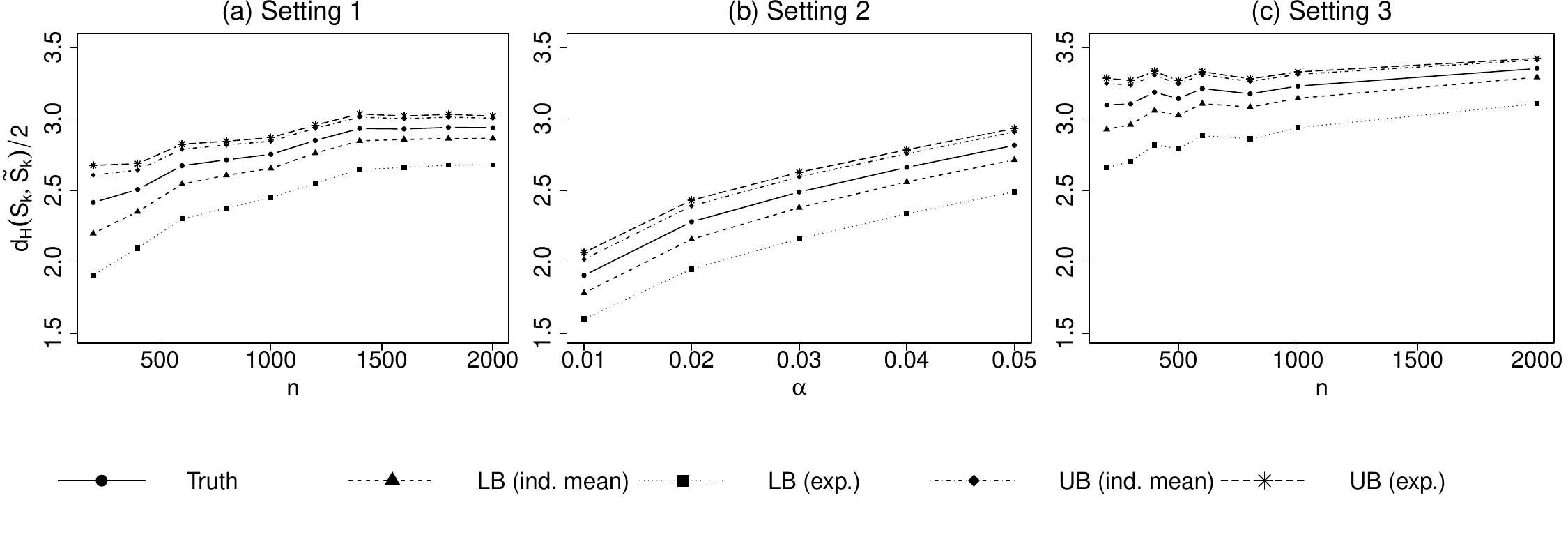}
    \caption{
    Empirical mean of $d_H(S_k,\widetilde{S}_k)/2$ and the corresponding lower and upper bounds under ER graphs. In each panel, the curves show the empirical mean of $d_H(S_k,\widetilde{S}_k)/2$ together with its lower and upper bounds.
    The three panels correspond to Settings 1--3.
    }
    \label{fig:er-results}
\end{figure}

Figure~\ref{fig:er-results} summarizes the results. In each panel, we plot the empirical mean of $d_H(S_k, \widetilde{S}_k)/2$, averaged over 100 realizations of the ER graph and, for each graph, 100 realizations of the noise process. The dashed lines correspond to the theoretical lower and upper bounds from Lemmas~\ref{lem:lower_bound} and~\ref{lem:upper_bound}. We consider two versions of these bounds: an indicator-based version and an expectation-based version. 

Across all three panels in Figure~\ref{fig:er-results}, the empirical set difference remains non-negligible, indicating that top-$k$ recovery under ER graphs is highly sensitive to edge noise. The upper and lower bounds provide informative approximations to the empirical behavior. Panel~(c) corresponds to a parameter regime that falls in the impossibility region for exact top-$k$ recovery under the ER model, as characterized by Theorem~\ref{thm:ER_threshold}.

\subsubsection{Preferential Attachment Networks}

We next investigate the stability of the top-$k$ set under PA networks.
Here, $m$ denotes the number of edges added at each step. We fix $b = 1$, corresponding to linear preferential attachment with attachment function $d(v) + b$.
\begin{itemize}
    \item \textbf{Setting 1 (Varying $n$):} 
    Number of nodes 
    $n \in \{200, 500, 1000, 2000, 5000\}$; 
    fixed parameters: $m = 5$, $k = 5$, $\alpha_n = 0.05$, $\beta_n = 0.05$.
    
    \item \textbf{Setting 2 (Varying $\alpha_n$):} 
    False positive rate 
    $\alpha_n \in \{0.01, 0.02, 0.03, 0.04, 0.05\}$; 
    fixed parameters: $n = 1000$, $m = 5$, $k = 5$, $\beta_n = 0.05$.

    \item \textbf{Setting 3 (Noise scaling with $n$):} 
       To mimic the consistency regime in Theorem~\ref{thm:PA_joint}, we let the false positive rate $\alpha_n$
    decay with $n$. Specifically, we take
    $n$ from 300 to 1500, fix $m = 5$ and $k = 5$, set $\beta_n = 0.05$, and consider
    \[
    \alpha_n =  n^{-\frac{b}{2+b}}(\log n)^{-2},
    \]
    with $b=1$, which places the PA model in the consistency regime characterized by Theorem~\ref{thm:PA_joint}.
\end{itemize}

\begin{figure}[tbp]
    \centering
    \includegraphics[width=\textwidth]{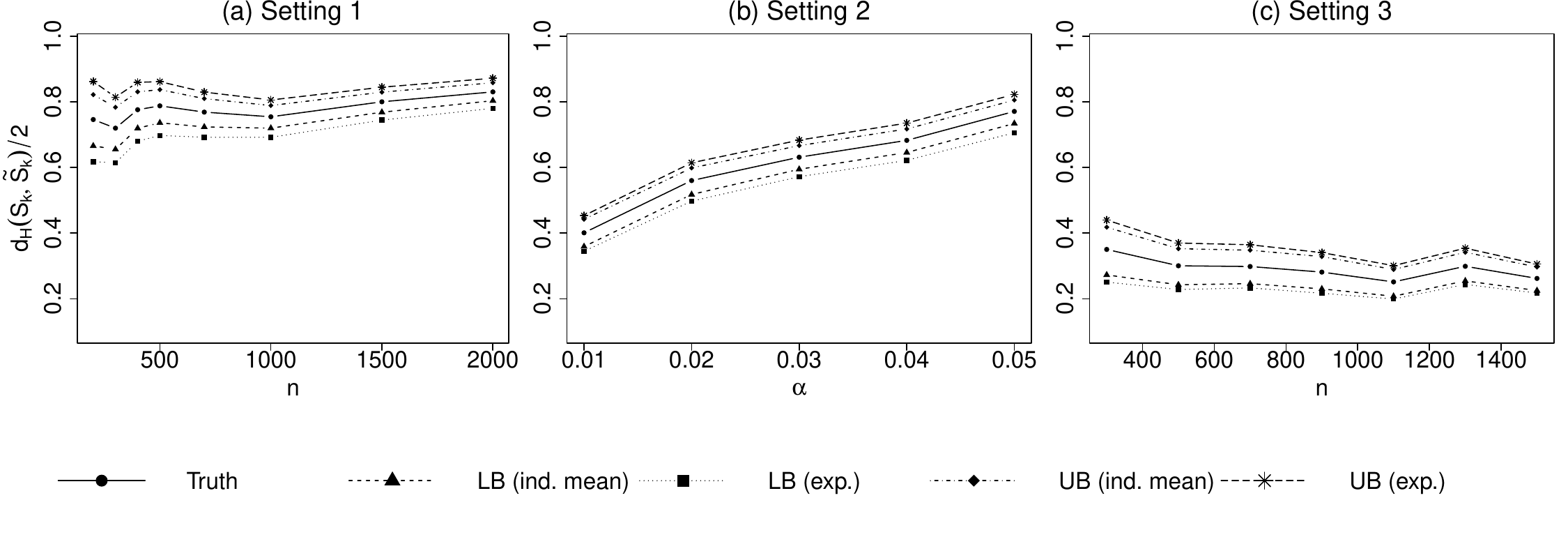}
    \caption{
    Empirical mean of $d_H(S_k,\widetilde S_k)/2$ and the corresponding lower and upper bounds under PA models.
    In each panel, the curves show the empirical mean of $d_H(S_k,\widetilde{S}_k)/2$ together with its lower and upper bounds.
    The three panels correspond to Settings 1--3.
    }
    \label{fig:PA_results}
\end{figure}

Figure~\ref{fig:PA_results} displays the corresponding results. As in the ER case, the solid line represents the empirical set difference, while the dashed lines denote the theoretical bounds. In contrast to ER graphs, PA networks exhibit substantially greater robustness to noise. Panel~(c) shows that when the parameters fall within the consistency regime, the set difference remains small and exhibits a mild decreasing trend, reflecting the stabilizing effect of the heavy-tailed degree distribution.

Overall, these results show that the proposed bounds capture the empirical behavior reasonably well and highlight a fundamental contrast between homogeneous and heavy-tailed network models in terms of top-$k$ stability.

\subsection{Localization of eigenvector centrality}

In this section, we empirically examine Assumptions~\ref{asm:localization} and
\ref{asm:hub-separation} in PA networks. These simulations are intended to assess whether the hub-dominated regime underlying
Theorem~\ref{thm:EC_consistency} is realized in practice.

For each network size \(n\), we generate \(B=200\) independent PA networks in which each new node attaches to \(m=1\) existing node, with the attachment probability proportional to \(\deg(i)+b\), where \(b=1\).
Let \(h\) denote the unique maximum-degree vertex, \(N(h)=\{i:i\sim h\}\) its neighbor set,
and \(\bx\) the \(\ell_2\)-normalized principal eigenvector of the adjacency matrix.

Figure~\ref{fig:gap-simulation} summarizes the results.
Panel~(a) shows that the hub entry \(x_h\) remains bounded away from zero as \(n\) increases.
Panel~(b) plots
\(
M_{\mathrm{out}}=\sum_{i\notin\{h\}\cup N(h)} x_i^2,
\) 
which decreases with \(n\), indicating increasing concentration of eigenvector mass on the hub and its immediate neighborhood.
Together, Panels~(a) and (b) provide empirical support for Assumption~\ref{asm:localization}.
Panel~(c) reports the separation statistic
\(
\Delta = x_h - \max_{i\in N(h)} x_i,
\)
which remains non-negligible across network sizes, supporting the hub separation condition in Assumption~\ref{asm:hub-separation}.

\begin{figure}[t]
  \centering
  \begin{subfigure}[t]{0.32\textwidth}
    \centering
    \includegraphics[width=\textwidth]{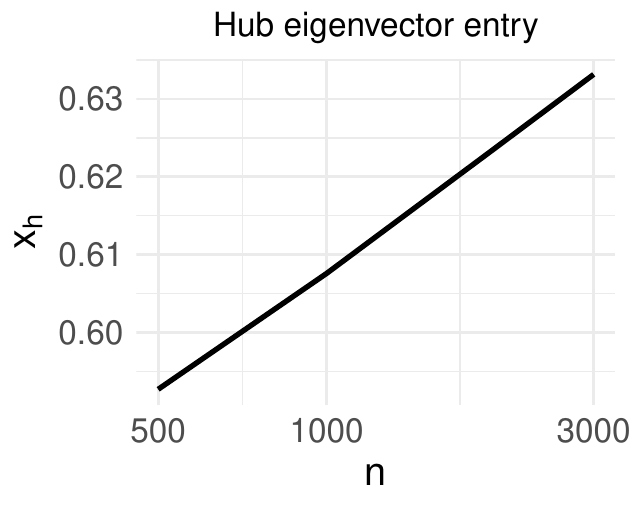}
    \caption{Hub eigenvector entry \(x_h\).}
    \label{fig:xh-m1}
  \end{subfigure}\hfill
  \begin{subfigure}[t]{0.32\textwidth}
    \centering
    \includegraphics[width=\textwidth]{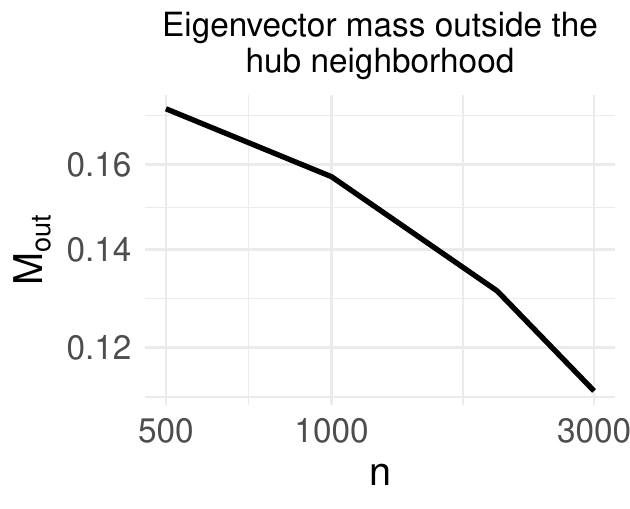}
    \caption{Mass outside hub neighborhood \(M_{\mathrm{out}}\).}
    \label{fig:mout-m1}
  \end{subfigure}\hfill
  \begin{subfigure}[t]{0.32\textwidth}
    \centering
    \includegraphics[width=\textwidth]{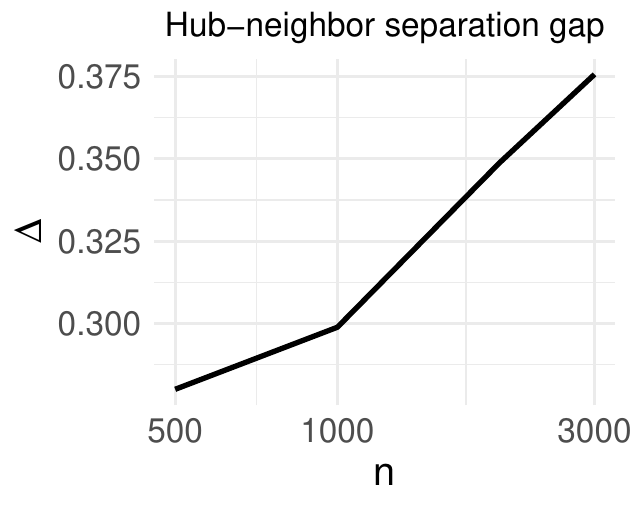}
    \caption{Hub--neighbor separation gap \(\Delta\).}
    \label{fig:delta-m1}
  \end{subfigure}

  \caption{Empirical diagnostics for eigenvector localization under preferential attachment with \(m=1\).}
  \label{fig:gap-simulation}
\end{figure}

\subsection{Comparing degree and eigenvector centrality under noise}
In this section, we empirically compare the robustness of degree centrality and eigenvector centrality under noisy observations of PA networks. Our focus is on how well each centrality measure preserves the top-$k$ node set under noise. To summarize performance, we use the Jaccard similarity between the true and noisy top-$k$ sets, which provides a normalized measure of overlap on the unit interval.

We generate PA networks with $n=1000$ nodes and attachment parameter $m=3$, and consider top-$k$ identification with $k=10$. For each noise level $(\alpha_n,\beta_n)$, results are averaged over 50 independently generated networks and 50 independent noise realizations per network. Figure~\ref{fig:topk_stability} shows that when both $\alpha_n$ and $\beta_n$ are small, the true and noisy top-$k$ sets exhibit similarly high overlap under both centrality measures. As the noise level increases, degree centrality becomes consistently more robust than eigenvector centrality, with the difference becoming more pronounced at moderate to large noise levels.

\begin{figure}[t]
    \centering
    \includegraphics[width=\textwidth]{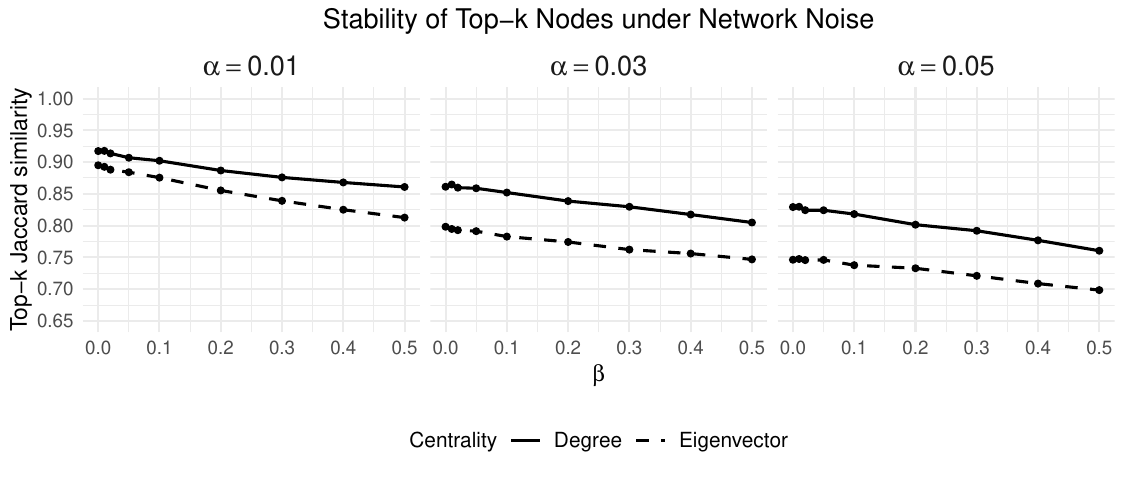}
    \caption{
    Stability of top-$k$ nodes under network noise.
    The y-axis shows the average Jaccard similarity between the true and noisy top-$k$ sets for degree and eigenvector centrality.
    }
    \label{fig:topk_stability}
\end{figure}

\section{Discussion and Conclusions}

In this paper, we develop a theoretical framework for understanding how network noise impacts centrality-based rankings. Focusing on degree-based rankings, we analyze how edge-level perturbations affect the recovery of the top-$k$ nodes and establish conditions under which reliable identification is possible or provably impossible. We also study eigenvector centrality, showing how edge noise propagates through spectral structure and affects ranking stability in a hub-dominated regime. 

For practical implementation of our bounds, estimation of the noise parameters $\alpha_n$ and $\beta_n$ is required; consistent estimators have been proposed and analyzed in \cite{chang2022estimation}. Our analysis focuses on edge-independent noise and degree-based summaries, and extending the framework to more complex dependence structures and alternative notions of centrality remains an important direction for future work. For example, distance-based centrality measures such as betweenness and closeness are of practical importance, but their analysis typically requires substantially stronger modeling assumptions. In particular, even a small number of missing edges can disconnect the graph or substantially alter shortest-path structure, making meaningful perturbation analysis fundamentally more delicate.

Several directions for future work remain. One natural extension is to develop inference procedures for quantifying uncertainty in top-$k$ rankings when multiple independent network observations are available, potentially building on latent-variable or likelihood-based approaches such as EM-type algorithms \citep{newman2018network}. It would also be interesting to explore connections with subsampling-based inference, differential privacy, and network mismeasurement models, where uncertainty arises from systematic rather than purely random noise.

\bibliography{bib/bibfile.bib} \bibliographystyle{asa}

\section*{Appendix}
\addcontentsline{toc}{section}{Appendix}

This appendix is organized as follows. 
Appendix~\ref{sec:appendix-main} contains the proofs of the main
theoretical results stated in Sections~\ref{sec:theory}–\ref{sec:EC}.  Appendix~\ref{sec:appendix-lemmas} collects technical lemmas and auxiliary concentration bounds used throughout the proofs. 

%%%%%%%%%%%%%%%%%%%%%%%%%%%%%%%%%%%%%%%%%%%%%%%%%%%%
\section{Proofs of Main Theorems}
\label{sec:appendix-main}
%%%%%%%%%%%%%%%%%%%%%%%%%%%%%%%%%%%%%%%%%%%%%%%%%%%%

\subsection{Degree–based results (Section~\ref{sec:theory})}

\subsubsection{Proof of Theorem \ref{thm:consistency_general}} 

\begin{proof}
Without loss of generality, assume $d_1 \ge d_2 \ge \cdots \ge d_n$. Let $i^*$ be the smallest index $>k$ satisfying the assumptions in the Theorem \ref{thm:consistency_general}. 
Write
\[
\cE_{\mathrm{bulk}} := \Big\{\max_{i\geq i^*}\tilde d_i \ge \min_{i\leq k} \tilde d_i\Big\},
\qquad
\cE_{\mathrm{bdry}} := \Big\{\max_{k< i < i^*}\tilde d_i \ge \min_{i\leq k} \tilde d_i\Big\}.
\]
Then
\[
\bbP(\widetilde S_k = S_k \mid \Ab) \;\ge\; 1 - \bbP(\cE_{\mathrm{bulk}}\mid \Ab) - \bbP(\cE_{\mathrm{bdry}}\mid \Ab).
\]
It suffices to show
\[
\bbP(\cE_{\mathrm{bulk}}\mid \Ab)\le \delta+o(1),
\qquad
\bbP(\cE_{\mathrm{bdry}}\mid \Ab)\le \delta.
\]
By Lemma~\ref{lem:quantile_bound}, with
\[
c_{n,i^*}^u
:= \mu_{n,i^*}
+\big(\sqrt{2\log(n-i^*+1)}-\epsilon_1(n-i^*+1)+\epsilon_2(n)\big)\sigma_{n,i^*},
\]
we have
\[
\bbP\!\Big(\max_{i\ge i^*}\tilde d_i \le c_{n,i^*}^u \,\Big|\,\Ab\Big)=1-o(1).
\]
Hence
\[
\bbP(\cE_{\mathrm{bulk}}\mid \Ab)
\le
\bbP\!\Big(\min_{i\le k}\tilde d_i\le c_{n,i^*}^u \,\Big|\,\Ab\Big)+o(1).
\]
By association and stochastic dominance,
\[
\bbP\!\Big(\min_{i\le k}\tilde d_i> c_{n,i^*}^u \,\Big|\,\Ab \Big)
\ge
\Big[\bbP\!\big(\tilde d_k>c_{n,i^*}^u \,\big|\,\Ab \big)\Big]^k.
\]
Bernstein’s inequality and assumption~(i) give
\[
\bbP\!\big(\tilde d_k\le c_{n,i^*}^u \,\big|\,\Ab\big)\le e^{-L_{\mathrm{bdry}}}=\delta/k,
\]
and therefore
\[
\bbP(\cE_{\mathrm{bulk}}\mid \Ab)\le \delta+o(1).
\]

\paragraph{Step 2: Control of indices $k < i < i^*$ under (ii).}
By the union bound and Bernstein’s inequality,
\[
\bbP(\cE_{\mathrm{bdry}}\mid \Ab)
\le
k(i^*-k)\exp\!\left(
-\,\frac{(\mu_{n,k}-\mu_{n,k+1})^2}{
2\bar{\sigma}_{\mathrm{bdry}}^2+\frac23(\mu_{n,k}-\mu_{n,k+1})
}
\right).
\]
Under assumption~(ii),
\[
\mu_{n,k}-\mu_{n,k+1}
\ge
\sqrt{2L_{\mathrm{bdry}}}\,\bar{\sigma}_{\mathrm{bdry}}
+\tfrac23 L_{\mathrm{bdry}},
\]
so the exponential term is at most \(e^{-L_{\mathrm{bdry}}}=\delta/[k(i^*-k)]\). Hence
\[
\bbP(\cE_{\mathrm{bdry}}\mid \Ab)\le \delta.
\]
\paragraph{Conclusion.}
Combining Steps~1–2,
\[
\bbP(\widetilde S_k = S_k \mid \Ab)
\;\ge\;
1 - \bbP(\cE_{\mathrm{bulk}}\mid \Ab) - \bbP(\cE_{\mathrm{bdry}}\mid \Ab)
\;\ge\;
1 - 2\delta - o(1),
\]
as claimed.
\end{proof}

\subsubsection{Proof of Theorem \ref{thm:infeasibility}}

\begin{proof}
Throughout the proof, probabilities are taken conditionally on the latent network $\Ab$.

Since the theorem establishes a lower bound on the supremum error probability over the degree classes
\[
\mathcal D^{\mathrm{bulk}}_{k,i^*}(\delta_k^{\mathrm{bulk}})
\cup
\mathcal D^{\mathrm{bdry}}_{k}(\delta_k^{\mathrm{bdry}}),
\]
it suffices to construct, in each regime, a particular degree vector in the corresponding class for which the top-$k$ recovery error is bounded away from zero. We therefore work with least-favorable configurations in each regime.

Accordingly, the proof proceeds separately for the bulk-separation and top-gap regimes.

\paragraph{bulk-separation regime.}  We start with the bulk-separation regime. We construct a homogeneous configuration such that
\[
d_1 = \cdots = d_k, 
\qquad 
d_{k+1} = \cdots = d_n,
\]
with
\[
d_k - d_{i^*} = d_k - d_{k+1} \le \delta_k^{\mathrm{bulk}}.
\]
Then $\bd \in \mathcal D^{\mathrm{bulk}}_{k,i^*}(\delta_k^{\mathrm{bulk}})$. Under this construction, the top-$k$ nodes share identical parameters $\mu_{n,k}$ and $\sigma_{n,k}$.

Let $c$ be a threshold to be specified below. The probability of incorrect top-$k$ recovery satisfies
\begin{align*}
    \bbP(S_k \neq \widetilde{S}_k)
    &= \bbP\!\left(\min_{i\le k} \tilde{d}_i < \max_{i\ge k+1} \tilde{d}_i \right) 
    \ge \bbP\!\left(\min_{i\le k} \tilde{d}_i < \max_{i\ge i^*} \tilde{d}_i \right) \\
    &\ge \bbP\!\left(\min_{i\le k} \tilde{d}_i < c,\; \max_{i\ge i^*} \tilde{d}_i \ge c \right) \\
    &\ge \bbP\!\left(\min_{i\le k} \tilde{d}_i < c \right) + \bbP\!\left(\max_{i\ge i^*} \tilde{d}_i \ge c \right) - 1.
\end{align*}
We first control the upper tail of $\max_{i\ge i^*}\tilde d_i$ by choosing
\[
c = c_{n,i^*}^{\,l} := 
\mu_{n,i^*} + \Big(\sqrt{2\log(n-i^*+1)} - \epsilon_1(n-i^*+1) - \epsilon_2(n)\Big)\sigma_{n,i^*},
\]
so that by Lemma~\ref{lem:quantile_bound}, under $\sigma_{n,i^*}^2 \to \infty$,
\[
\bbP\!\left(\max_{i\ge i^*} \tilde{d}_i \ge c\right) \ge 1 - o(1).
\]

Next, we bound the lower tail of $\min_{i\le k}\tilde d_i$. If
\(\log^7(kn) = o(\min_{1\le i\le k}\sigma_{n,i}^2)\),
then by Lemma~\ref{lem:min_CLT},
\[
\bbP\!\left(\min_{i\le k} \tilde{d}_i < c \right)
\ge \bbP\!\left(\min_{i\le k} z_i < \frac{c - \mu_{n,k}}{\sigma_{n,k}}\right)
- r_n,
\]
where $r_n \to 0$.

We now analyze the Gaussian minimum. Let
\[
t := \frac{c - \mu_{n,k}}{\sigma_{n,k}}.
\]
Then
\[
\bbP\!\left(\min_{i\le k} z_i < t \right)
= 1 - \bbP(z_1 \ge t)^k
= 1 - \big(1 - \Phi(t)\big)^k.
\]

Under the conditions of Theorem~\ref{thm:infeasibility},
\begin{align*}
t
&= \frac{1}{\sigma_{n,k}}\Big[\mu_{n,i^*}
+ \Big(\sqrt{2\log(n-i^*+1)} - \epsilon_1(n-i^*+1) - \epsilon_2(n)\Big)\sigma_{n,i^*}
- \mu_{n,k}\Big] \\
&= -\frac{1}{\sigma_{n,k}}\Big[
(1 - \alpha_n - \beta_n)\,\Delta_k^{\operatorname{bulk}}
- \big(\sqrt{2\log(n-i^*+1)} - \epsilon_1(n-i^*+1) - \epsilon_2(n)\big)\sigma_{n,i^*}
\Big] \ge 0.
\end{align*}
Hence $\Phi(t) \ge \tfrac{1}{2}$, so
\[
\bbP\!\left(\min_{i\le k} z_i < t \right)
= 1 - (1 - \Phi(t))^k
\ge 1 - 2^{-k}
\ge \frac{1}{2}, \qquad \text{for all } k \ge 1.
\]

Combining the bounds, we obtain
\[
\bbP(S_k \neq \widetilde{S}_k)
\ge \bbP\!\left(\min_{i\le k} \tilde{d}_i < c\right)
+ \bbP\!\left(\max_{i\ge i^*} \tilde{d}_i \ge c\right) - 1
\ge \frac{1}{2} - o(1).
\]

This establishes the desired lower bound in the bulk-separation regime.

\paragraph{Top-gap regime.}
We next consider the top-gap regime. Since the theorem is stated as a lower bound on the supremum error probability over
\(
\mathcal D_k^{\mathrm{bdry}}(\delta_k^{\mathrm{bdry}})
\),
it suffices to construct a particular degree vector in this class for which the top-$k$ recovery error is bounded away from zero.

We choose a least-favorable homogeneous configuration such that
\[
d_1=\cdots=d_k > d_{k+1}=\cdots=d_{i^*},
\]
with
\[
d_k-d_{k+1}\le \delta_k^{\mathrm{bdry}}.
\]
Then \(\bd\in \mathcal D_k^{\mathrm{bdry}}(\delta_k^{\mathrm{bdry}})\). Under this construction, the top-$k$ nodes share identical parameters \(\mu_{n,k}\) and \(\sigma_{n,k}\), while the competing block \(\{k+1,\ldots,i^*-1\}\) shares identical parameters \(\mu_{n,k+1}\) and \(\sigma_{n,k+1}\).

Let
\[
t:=\frac{\mu_{n,k}+\mu_{n,k+1}}{2}.
\]
Then
\[
\bbP(S_k\neq \widetilde S_k)
\ge
\bbP\!\left(\min_{i\le k}\tilde d_i \le \max_{k<j<i^*}\tilde d_j\right).
\]

We distinguish two cases.

\smallskip
\noindent\textit{Case 1: \(k\) and \(i^*-k\) are fixed.}

Since \(i^*-1\) is fixed, Lemma~\ref{lem:min_CLT} implies that
\[
(\tilde d_1^s,\ldots,\tilde d_{i^*-1}^s)\convd \cN(\bzero,\bI_{i^*-1}).
\]
Therefore,
\[
\bbP\!\left(\min_{i\le k}\tilde d_i \le \max_{k<j<i^*}\tilde d_j\right)
=
\bbP\!\left(
\min_{i\le k}(\sigma_{n,k}z_i+\mu_{n,k})
\le
\max_{k<j<i^*}(\sigma_{n,k+1}z_j+\mu_{n,k+1})
\right)+o(1),
\]
where \((z_1,\ldots,z_{i^*-1})\sim \cN(\bzero,\bI_{i^*-1})\).

Hence
\begin{align*}
\bbP(S_k\neq \widetilde S_k)
&\ge
\bbP\!\left(
\min_{i\le k}(\sigma_{n,k}z_i+\mu_{n,k})
\le
\max_{k<j<i^*}(\sigma_{n,k+1}z_j+\mu_{n,k+1})
\right)+o(1) \\
&\ge
\bbP\!\left(
\min_{i\le k}(\sigma_{n,k}z_i+\mu_{n,k}) \le t
\le
\max_{k<j<i^*}(\sigma_{n,k+1}z_j+\mu_{n,k+1})
\right)+o(1).
\end{align*}
By independence of the Gaussian coordinates, this equals
\begin{align*}
&\bbP\!\left(\min_{i\le k}(\sigma_{n,k}z_i+\mu_{n,k}) \le t\right)
\cdot
\bbP\!\left(\max_{k<j<i^*}(\sigma_{n,k+1}z_j+\mu_{n,k+1}) \ge t\right)
+o(1).
\end{align*}

For the first factor, let
\[
x_1:=\frac{t-\mu_{n,k}}{\sigma_{n,k}}
=
-\frac{\mu_{n,k}-\mu_{n,k+1}}{2\sigma_{n,k}}.
\]
Since
\[
\mu_{n,k}-\mu_{n,k+1}
=
(1-\alpha_n-\beta_n)\Delta_k^{\mathrm{bdry}},
\]
and
\[
\Delta_k^{\mathrm{bdry}}
\le
\frac{\eta}{1-\alpha_n-\beta_n}
\min\left\{
2\sqrt{2\log k}\,\sigma_{n,k},
\,
2\sqrt{2\log(i^*-k)}\,\sigma_{n,k+1}
\right\},
\]
for some \(\eta\in(0,1)\), we have
\[
x_1\ge -\eta\sqrt{2\log k}.
\]
Therefore,
\[
\bbP\!\left(\min_{i\le k}(\sigma_{n,k}z_i+\mu_{n,k}) \le t\right)
=
1-\bigl(1-\Phi(x_1)\bigr)^k.
\]
Since \(k\) is fixed, this is bounded below by a positive constant \(a_1>0\).

Similarly, let
\[
x_2:=\frac{t-\mu_{n,k+1}}{\sigma_{n,k+1}}
=
\frac{\mu_{n,k}-\mu_{n,k+1}}{2\sigma_{n,k+1}}
\le
\eta\sqrt{2\log(i^*-k)}.
\]
Then
\[
\bbP\!\left(\max_{k<j<i^*}(\sigma_{n,k+1}z_j+\mu_{n,k+1}) \ge t\right)
=
1-\Phi(x_2)^{\,i^*-k-1}.
\]
Since \(i^*-k\) is fixed, this is also bounded below by a positive constant \(a_2>0\).

Combining the two bounds yields
\[
\bbP(S_k\neq \widetilde S_k)\ge a_1a_2-o(1).
\]
Hence there exists a constant \(c_2>0\) such that
\[
\bbP(S_k\neq \widetilde S_k)\ge c_2
\]
for all sufficiently large \(n\). 

\smallskip
\noindent\textit{Case 2: \(i^*-k\to\infty\).}
In this case,
\begin{align*}
\bbP(S_k\neq \widetilde S_k)
&\ge
\bbP\!\left(\min_{i\le k}\tilde d_i \le \max_{k<j<i^*}\tilde d_j\right) \\
&\ge
\bbP\!\left(\min_{i\le k}\tilde d_i \le t,\ \max_{k<j<i^*}\tilde d_j \ge t\right) \\
&\ge
\bbP\!\left(\min_{i\le k}\tilde d_i \le t\right)
+
\bbP\!\left(\max_{k<j<i^*}\tilde d_j \ge t\right)-1.
\end{align*}

For the first term, by Lemma~\ref{lem:min_CLT},
\[
\bbP\!\left(\min_{i\le k}\tilde d_i \le t\right)
\ge
\bbP\!\left(\min_{i\le k} z_i \le x_1\right)-r_{n,1},
\]
where \(r_{n,1}\to 0\),
\[
x_1:=\frac{t-\mu_{n,k}}{\sigma_{n,k}}
=
-\frac{\mu_{n,k}-\mu_{n,k+1}}{2\sigma_{n,k}},
\]
and \((z_1,\ldots,z_k)\sim \cN(\bzero,\bI_k)\).

Recall
\[
\mu_{n,k}-\mu_{n,k+1}
=
(1-\alpha_n-\beta_n)\Delta_k^{\mathrm{bdry}},
\]
and
\[
\Delta_k^{\mathrm{bdry}}
\le
\frac{\eta}{1-\alpha_n-\beta_n}
\min\left\{
2\sqrt{2\log k}\,\sigma_{n,k},\;
2\sqrt{2\log(i^*-k)}\,\sigma_{n,k+1}
\right\},
\]

Since \(x_1 \ge -\eta\sqrt{2\log k}\) with \(\eta\in(0,1)\), we have
\[
\Phi(x_1)\ge \Phi\!\left(-\eta\sqrt{2\log k}\right).
\]
By the standard Gaussian lower-tail bound, for some constant \(C>0\),
\[
\Phi\!\left(-\eta\sqrt{2\log k}\right)\ge \frac{C}{k^{\eta^2}\sqrt{\log k}}
\]
for all sufficiently large \(k\). Hence
\[
k\,\Phi(x_1)\to\infty.
\]
Therefore,
\[
\bbP\!\left(\min_{i\le k} z_i \le x_1\right)
=1-\bigl(1-\Phi(x_1)\bigr)^k
\ge 1-\exp\{-k\Phi(x_1)\}\to 1.
\]
In particular, there exists a constant \(a_1>0\) such that
\[
\bbP\!\left(\min_{i\le k} z_i \le x_1\right)\ge a_1
\]
for all sufficiently large \(n\).
For the second term, let
\[
x_2:=\frac{t-\mu_{n,k+1}}{\sigma_{n,k+1}}
=
\frac{\mu_{n,k}-\mu_{n,k+1}}{2\sigma_{n,k+1}}
\le \eta\sqrt{2\log(i^*-k)}.
\]
Writing \(m:=i^*-k-1\), we have
\[
\bbP\!\left(\max_{k<j<i^*}\tilde d_j \ge t\right)
\ge
\bbP\!\left(\max_{1\le j\le m} z_j \ge x_2\right)-r_{n,2},
\]
where \(r_{n,2}\to 0\). Since
\[
\bbP\!\left(\max_{1\le j\le m} z_j \ge x_2\right)
=
1-\Phi(x_2)^m,
\]
and \(x_2\le \eta\sqrt{2\log(i^*-k)}\) with \(\eta\in(0,1)\), it follows from the Gaussian tail bound that
\[
\Phi(x_2)^m\to 0.
\]
Hence
\[
\bbP\!\left(\max_{k<j<i^*}\tilde d_j \ge t\right)\ge 1-o(1).
\]
Combining the two bounds gives
\[
\bbP(S_k\neq \widetilde S_k)
\ge a_1-o(1).
\]
Hence there exists a constant \(c_3>0\) such that
\[
\bbP(S_k\neq \widetilde S_k)\ge c_3
\]
for all sufficiently large \(n\). Therefore, the desired lower bound holds in the top-gap regime in both cases.

\end{proof}

\subsubsection{Proof of Theorem \ref{thm:PA_joint}}
\begin{proof}
Without loss of generality, write \(d_1 \ge d_2 \ge \cdots \ge d_n\).
Fix \(b>-1\) and \(k\). By Lemma~\ref{lem:PA_limits}, for every fixed \(J\ge1\),
\[
n^{-1/(2+b)}(d_1,\ldots,d_{k+J})\convd (\nu_1,\ldots,\nu_{k+J}),
\]
where \(\nu_1\ge\cdots\ge\nu_{k+J}>0\) a.s. Moreover, Lemma~\ref{lem:no_ties_nu} gives
\(\mathbb P(\nu_j=\nu_{j+1})=0\) for every fixed \(j\ge1\). Hence
\[
\nu_k-\nu_{k+1}>0 \quad\text{a.s.}, \qquad \nu_k-\nu_{k+J}>0 \quad\text{a.s.}
\]
for every fixed \(J\ge1\).

Thus, it remains to verify the conditions of Theorem~\ref{thm:consistency_general}. Specifically, we show that the relevant degree gaps are asymptotically of order \(n^{1/(2+b)}\), whereas the boundary and bulk fluctuation thresholds are \(o_p(n^{1/(2+b)})\).

Under independent edge flips,
\[
\sigma_{n,i}^2\le n\alpha_n(1-\alpha_n)+d_i\beta_n(1-\beta_n).
\]
Since \(d_i=\Theta_p(n^{1/(2+b)})\) for each fixed \(i\), it follows that
\(\sigma_{n,k}^2+\sigma_{n,k+1}^2 \le 2n\alpha_n + O_p(n^{1/(2+b)}\beta_n)\).

Now fix \(J\ge1\), let \(i^*:=k+J\), and set
\(L_{\mathrm{bdry}}=\log(kJ/\delta_n)\) and \(L_{\mathrm{bdry}}=\log(k/\delta_n)\).
Take \(\delta_n=1/\log n\). Then \(L_{\mathrm{bdry}}=O(\log\log n)\) and \(L_{\mathrm{bdry}}=O(\log\log n)\). 
Define
\[
y_{\mathrm{bdry}}(n):=
\frac{1}{(1-\alpha_n-\beta_n)\,n^{1/(2+b)}}
\Big[
\sqrt{2L_{\mathrm{bdry}}}\,\sqrt{\sigma_{n,k}^2+\sigma_{n,k+1}^2}
+\frac{2}{3}L_{\mathrm{bdry}}
\Big],
\]
and
\[
y_{\mathrm{bulk}}(n):=
\frac{1}{(1-\alpha_n-\beta_n)\,n^{1/(2+b)}}
\Big[
\big(\sqrt{2\log(n-i^*+1)}-\epsilon_1(n)-\epsilon_2(n)\big)\sigma_{n,i^*}
+\sigma_{n,k}\sqrt{2L_{\mathrm{bdry}}}
+\frac{2}{3}L_{\mathrm{bdry}}
\Big].
\]

Assume $
\alpha_n=r_n\,\frac{n^{-b/(2+b)}}{\log n}, r_n\to0,$ 
and $1-\alpha_n-\beta_n\ge c_0>0$. 

Using
\[
\sigma_{n,k}^2+\sigma_{n,k+1}^2
\le 2n\alpha_n+(d_k+d_{k+1})\beta_n
\]
and $d_k=\Theta_p(n^{1/(2+b)})$,
\[
\frac{\sqrt{2L_{\mathrm{bdry}}}\,\sqrt{\sigma_{n,k}^2+\sigma_{n,k+1}^2}}{n^{1/(2+b)}}
=
O_p\!\Big(
\sqrt{L_{\mathrm{bdry}}}\,\frac{\sqrt{n\alpha_n}}{n^{1/(2+b)}}
\Big)
+
O_p\!\Big(
\sqrt{L_{\mathrm{bdry}}}\,\sqrt{\beta_n}\,n^{-1/(4+2b)}
\Big).
\]
Since
\[
\frac{\sqrt{n\alpha_n}}{n^{1/(2+b)}}=\frac{\sqrt{r_n}}{\sqrt{\log n}}\to0,
\qquad
\sqrt{\beta_n}\,n^{-1/(4+2b)}\to0,
\]
and \(L_{\mathrm{bdry}}=O(\log\log n)\), the first term in \(y_{\mathrm{bdry}}(n)\) is \(o_p(1)\); moreover,
\(L_{\mathrm{bdry}}/n^{1/(2+b)}=o(1)\). Hence \(y_{\mathrm{bdry}}(n)=o_p(1)\).

Also,
\[
\sigma_{n,i^*}
=
O_p\left(\sqrt{n \alpha_n}\right)
+
O_p\!\big(n^{1/(4+2b)}\sqrt{\beta_n}\big).
\]
Therefore,
\[
\frac{\big(\sqrt{2\log(n-i^*+1)}-\epsilon_1(n)-\epsilon_2(n)\big)\sigma_{n,i^*}}{n^{1/(2+b)}}
=
O_p\!\Big(
\sqrt{\log n}\,\frac{\sqrt{n\alpha_n}}{n^{1/(2+b)}}
\Big)
+
O_p\!\Big(
\sqrt{\log n}\,n^{-1/(4+2b)}\sqrt{\beta_n}
\Big)
=o_p(1).
\]
Similarly,
\[
\frac{\sigma_{n,k}\sqrt{2L_{\mathrm{bdry}}}}{n^{1/(2+b)}}
=
O_p\!\Big(
\sqrt{\log n}\,\frac{\sqrt{n\alpha_n}}{n^{1/(2+b)}}
\Big)
+o_p(1)
=o_p(1),
\]
and
\[
\frac{L_{\mathrm{bdry}}}{n^{1/(2+b)}}=o(1).
\]
Thus
\[
y_{\mathrm{bulk}}(n)=o_p(1).
\]

We have shown
\[
y_{\mathrm{bdry}}(n)=o_p(1),\qquad y_{\mathrm{bulk}}(n)=o_p(1).
\]
Let \(\mathcal E_n\) denote the event that
\[
d_k-d_{k+1}>(1-\alpha_n-\beta_n)^{-1}n^{1/(2+b)}y_{\mathrm{bdry}}(n)
\quad\text{and}\quad
d_k-d_{k+J}>(1-\alpha_n-\beta_n)^{-1}n^{1/(2+b)}y_{\mathrm{bulk}}(n).
\]
By the preceding convergence results, \(\mathbb P(\mathcal E_n)\to 1\).
On \(\mathcal E_n\), the assumptions of Theorem~\ref{thm:consistency_general} are satisfied. Therefore, by Theorem~\ref{thm:consistency_general},
\[
\bbP(\widetilde S_k=S_k\mid \Ab)\ge 1-\delta_n-o(1)=1-o(1).
\]
Taking expectations over \(\Ab\), we obtain
\[
\bbP(\widetilde S_k=S_k)=1-o(1).
\]
\end{proof}

\begin{lem}\label{lem:PA_limits}[From \citet{mori2005maximum, bhamidi2012spectra}]
Consider the linear preferential attachment model with linear attachment function given in \eqref{eqn:PA_prob}. 
Let $\lambda_1 \ge \lambda_2 \ge \cdots \ge \lambda_k$ denote the $k$ largest eigenvalues of the adjacency matrix, 
and let $d_1 \ge d_2 \ge \cdots \ge d_k$ denote the $k$ largest degrees. 
For any fixed $k$, there exist non-degenerate random variables 
$\nu_1 \ge \nu_2 \ge \cdots \ge \nu_k > 0$ such that
\[
\left( \frac{d_{1}}{n^{1/(b+2)}}, \frac{d_{2}}{n^{1/(b+2)}}, \ldots, \frac{d_{k}}{n^{1/(b+2)}} \right)
\convd
\left(\nu_1, \nu_2, \ldots, \nu_k \right),
\]
and
\[
\left( \frac{\lambda_1}{n^{1/(2(b+2))}}, \frac{\lambda_2}{n^{1/(2(b+2))}}, \ldots,
\frac{\lambda_k}{n^{1/(2(b+2))}} \right)
\convd
\left(\sqrt{\nu_1}, \sqrt{\nu_2}, \ldots, \sqrt{\nu_k}\right),
\]
as $n \to \infty$, where \(b\) is the linear attachment parameter in \eqref{eqn:PA_prob}.  

\end{lem}

Recall that \(\mathcal F(t)\) denotes the family tree at time \(t\) in the continuous-time branching-process construction of the linear preferential attachment tree. In particular, for \(S>0\), the vertices of \(\mathcal F(S)\) are exactly those born before time \(S\).

\begin{lem}\label{lem:no_ties_nu}
Let $\nu_j$ be as in Lemma~\ref{lem:PA_limits}. For every fixed $j\ge 1$,
\[
\bbP(\nu_j=\nu_{j+1})=0.
\]
Hence
\[
\bbP(\nu_1>\nu_2>\cdots>\nu_{j+1})=1.
\]
\end{lem}

\begin{proof}
Fix \(j\ge 1\). By Corollary~\ref{cor:early_marks_no_ties}, for every \(S>0\),
\[
Y_1^S>Y_2^S>\cdots>Y_{j+1}^S>0
\qquad\text{a.s. on }\{\#\mathcal F(S)\ge j+1\}.
\]

Moreover, by Lemma 7.4 together with Corollary 7.6 and the completion argument following Corollary 7.6 in \citet{bhamidi2012spectra}, for every \(\varepsilon>0\) there exists \(S<\infty\) such that
\[
\mathbb P\bigl((\nu_1,\dots,\nu_{j+1})=(Y_1^S,\dots,Y_{j+1}^S)\bigr)\ge 1-\varepsilon.
\]
On the event
\[
(\nu_1,\dots,\nu_{j+1})=(Y_1^S,\dots,Y_{j+1}^S),
\]
we necessarily have \(\#\mathcal F(S)\ge j+1\), and hence
\[
\nu_j = Y_j^S > Y_{j+1}^S = \nu_{j+1}.
\]
Therefore,
\[
\{\nu_j=\nu_{j+1}\}
\subseteq
\{(\nu_1,\dots,\nu_{j+1})\neq (Y_1^S,\dots,Y_{j+1}^S)\}.
\]
It follows that
\[
\mathbb P(\nu_j=\nu_{j+1})
\le
\mathbb P\bigl((\nu_1,\dots,\nu_{j+1})\neq (Y_1^S,\dots,Y_{j+1}^S)\bigr)
\le \varepsilon.
\]
Since \(\varepsilon>0\) was arbitrary, we conclude that
\[
\mathbb P(\nu_j=\nu_{j+1})=0.
\]

Finally, since \(\nu_1\ge \nu_2\ge \cdots\) almost surely and
\(\mathbb P(\nu_i=\nu_{i+1})=0\) for each \(1\le i\le j\), we obtain
\[
\mathbb P(\nu_1>\nu_2>\cdots>\nu_{j+1})=1.
\]
\end{proof}

\begin{lem}\label{lem:Wa_continuous}
For the pure birth process $N_a$ in Theorem~7.2 of \citep{bhamidi2012spectra}(d), with jump rate
\[
q_{k,k+1}=k+1+a,\qquad k\ge 0,\quad a>-1,
\]
the martingale limit
\[
W_a:=\lim_{t\to\infty} e^{-t}N_a(t)
\]
has the $\mathrm{Gamma}(1+a,1)$ distribution. In particular, $W_a$ has a continuous density on $(0,\infty)$, and hence
\[
\bbP(W_a=x)=0
\qquad\text{for every }x\ge 0.
\]
\end{lem}

\begin{proof}
Let
\[
G_a(t,s):=\bbE[s^{N_a(t)}],\qquad |s|\le 1.
\]
Since the generator is
\[
Lf(k)=(k+1+a)\bigl(f(k+1)-f(k)\bigr),
\]
we have
\[
\partial_t G_a(t,s)
=
\bbE[L(s^{N_a(t)})]
=
(s-1)\Bigl(s\,\partial_s G_a(t,s)+(1+a)G_a(t,s)\Bigr),
\]
with initial condition $G_a(0,s)=1$.

A direct check shows that
\[
G_a(t,s)
=
\left(
\frac{e^{-t}}{1-(1-e^{-t})s}
\right)^{1+a}
\]
solves this equation. Thus, for every $\theta\ge 0$,
\[
\bbE\!\left[e^{-\theta e^{-t}N_a(t)}\right]
=
G_a\!\left(t,e^{-\theta e^{-t}}\right)
=
\left(
\frac{e^{-t}}{1-(1-e^{-t})e^{-\theta e^{-t}}}
\right)^{1+a}.
\]
Since
\[
1-(1-e^{-t})e^{-\theta e^{-t}}
=
(1+\theta)e^{-t}+o(e^{-t}),
\qquad t\to\infty,
\]
it follows that
\[
\lim_{t\to\infty}\bbE\!\left[e^{-\theta e^{-t}N_a(t)}\right]
=
(1+\theta)^{-(1+a)}.
\]
By Theorem 7.2(d) of \cite{bhamidi2012spectra}, $e^{-t}N_a(t)\to W_a$ almost surely, so bounded convergence gives
\[
\bbE[e^{-\theta W_a}]
=
(1+\theta)^{-(1+a)},
\qquad \theta\ge 0.
\]
This is the Laplace transform of the \(\Gamma(1+a,1)\) distribution. Hence
\[
W_a\sim \Gamma(1+a,1),
\]
with density
\[
f_{W_a}(x)=\frac{1}{\Gamma(1+a)}x^{a}e^{-x}\mathbf 1_{\{x>0\}}.
\]
In particular, \(W_a\) is absolutely continuous on \([0,\infty)\), and therefore
\[
\mathbb P(W_a=x)=0\qquad\text{for every }x\ge 0.
\]
\end{proof}

\begin{coro}\label{cor:early_marks_no_ties}
Fix $S>0$. Conditional on $\mathcal{F}(S)$, the limiting marks in Proposition~7.5 of \cite{bhamidi2012spectra} are independent and atomless. Consequently, on the event $\{\#\mathcal{F}(S)\ge k+1\}$,
\[
Y_1^{S}>Y_2^{S}>\cdots>Y_{k+1}^{S}>0
\qquad\text{a.s.}
\]
\end{coro}

\begin{proof}
Conditional on $\mathcal F(S)$, let $\{M_v : v\in \mathcal F(S)\}$ denote the limiting marks attached to the vertices in $\mathcal F(S)$. By Proposition~7.5 of \cite{bhamidi2012spectra}, conditional on $\mathcal F(S)$, the random variables $\{M_v : v\in \mathcal F(S)\}$ are independent, and for each $v\in\mathcal F(S)$, the law of $M_v$ is that of a strictly positive constant multiple of $W_{a+\zeta_v(S)}$.

By Lemma~\ref{lem:Wa_continuous}, the random variable $W_{a+\zeta_v(S)}$ is atomless, and hence so is any strictly positive constant multiple of it. Therefore, for any distinct vertices $u,v\in\mathcal F(S)$,
\[
\mathbb P(M_u=M_v\mid \mathcal F(S))=0.
\]
Since $\mathcal F(S)$ is finite, taking a union bound over all unordered pairs of distinct vertices yields
\[
\mathbb P\Bigl(\exists\,u\neq v\in\mathcal F(S): M_u=M_v \,\Bigm|\, \mathcal F(S)\Bigr)=0.
\]
Thus, conditional on $\mathcal F(S)$, all limiting marks $\{M_v : v\in\mathcal F(S)\}$ are distinct almost surely.

On the event $\{\#\mathcal F(S)\ge k+1\}$, the random variables $Y_1^S,\dots,Y_{k+1}^S$ are the first $k+1$ order statistics, in decreasing order, of the finite collection $\{M_v : v\in\mathcal F(S)\}$. Since these marks are almost surely all distinct conditional on $\mathcal F(S)$, their order statistics are almost surely strictly decreasing. Therefore,
\[
Y_1^{S}>Y_2^{S}>\cdots>Y_{k+1}^{S}>0
\qquad\text{a.s. on }\{\#\mathcal F(S)\ge k+1\}.
\]
\end{proof}

\subsubsection{Proof of Theorem \ref{thm:ER_threshold}}

\begin{lem}\label{lem:ER_maximum}
Assume $\Ab \sim \mathcal{G}(n, p_n)$ with $n p_n \to \infty$ as $n \to \infty$.
Let $k$ be fixed and let $C(n)\to\infty$ arbitrarily slowly.
Then, with probability $1-o(1)$,
\[
H^* + C(n)\Bigl(\frac{n p_n}{\log n}\Bigr)^{1/2}
\;\ge\;
d_1 \ge d_2 \ge \cdots \ge d_{k+1}
\;\ge\;
H^* - C(n)\Bigl(\frac{n p_n}{\log n}\Bigr)^{1/2}
\;>\; 0,
\]
where
\[
H^*
= n p_n
  + \bigl(2 p_n q_n n \log n\bigr)^{1/2}
  - \Bigl(\frac{p_n q_n n}{8 \log n}\Bigr)^{1/2}
    \log \log n,
\qquad q_n := 1 - p_n.
\]
\end{lem}

The proof of Lemma~\ref{lem:ER_maximum} is deferred to the appendix.

\begin{proof}[Proof of Theorem~\ref{thm:ER_threshold}]

As an immediate consequence of Lemma~\ref{lem:ER_maximum}, we have
\[
d_k - d_{k+1}
\;\le\;
2C(n)\Bigl(\frac{n p_n}{\log n}\Bigr)^{1/2}
\]
with probability $1-o(1)$.

Next, we lower bound the variance terms appearing in the noisy degree observations. Recall that
\[
\sigma_{n,i}^2
=
(n-1-d_i)\alpha_n(1-\alpha_n)
+
d_i \beta_n(1-\beta_n).
\]
Using the degree concentration in Lemma~\ref{lem:ER_maximum}, we obtain that, for all \(1 \le i \le k+1\),
\[
\sigma_{n,i}^2 \;\ge\; n\, c_n(\alpha_n,\beta_n)
\]
with probability \(1-o(1)\), where
\[
c_n(\alpha_n,\beta_n)
:=
\Bigl[q_n - \bigl(2p_n q_n \log n / n\bigr)^{1/2}\Bigr]\alpha_n(1-\alpha_n)
+
\Bigl[p_n - \bigl(2p_n q_n \log n / n\bigr)^{1/2}\Bigr]\beta_n(1-\beta_n).
\]
We now verify the sufficient condition of Theorem~\ref{thm:infeasibility}. Since
\[
\sigma_{n,i} \ge \sqrt{n\,c_n(\alpha_n,\beta_n)},
\]
it suffices to check that
\begin{equation}\label{eqn:ER_condition}
2C(n)\Bigl(\frac{n p_n}{\log n}\Bigr)^{1/2}
\;\le\;
\frac{1}{1-\alpha_n-\beta_n}
\Bigl(\sqrt{2\log(n-k)} - \epsilon_1(n) - \epsilon_2(n)\Bigr)
\sqrt{n\,c_n(\alpha_n,\beta_n)}.
\end{equation}

Since $k$ is fixed and $\alpha_n+\beta_n\le 1-c_0$, the right-hand side of
\eqref{eqn:ER_condition} is of order
\[
\sqrt{\log n}\,\sqrt{n\,c_n(\alpha_n,\beta_n)}.
\]
Therefore, \eqref{eqn:ER_condition} holds provided that
\[
C(n)\Bigl(\frac{n p_n}{\log n}\Bigr)^{1/2}
\;\lesssim\;
\sqrt{\log n}\,\sqrt{n q_n \alpha_n(1-\alpha_n)},
\]
or
\[
C(n)\Bigl(\frac{n p_n}{\log n}\Bigr)^{1/2}
\;\lesssim\;
\sqrt{\log n}\,\sqrt{n p_n \beta_n(1-\beta_n)}.
\]

\medskip
\noindent
\textbf{Dense regime.}
When $p_n = q_n = \Theta(1)$, the above inequalities simplify to
\[
C(n)\,\frac{\sqrt{n}}{\sqrt{\log n}}
\;\lesssim\;
\sqrt{n\log n}\,\sqrt{\alpha_n}
\qquad\text{or}\qquad
C(n)\,\frac{\sqrt{n}}{\sqrt{\log n}}
\;\lesssim\;
\sqrt{n\log n}\,\sqrt{\beta_n}.
\]
Equivalently,
\[
C(n)\;\lesssim\; (\log n)\sqrt{\alpha_n}
\qquad\text{or}\qquad
C(n)\;\lesssim\; (\log n)\sqrt{\beta_n}.
\]
Since \(C(n)\to\infty\) arbitrarily slowly, the above inequalities hold whenever
\[
(\log n)\sqrt{\alpha_n}\to\infty
\qquad\text{or}\qquad
(\log n)\sqrt{\beta_n}\to\infty,
\]
that is,
\[
\alpha_n \gg \frac{1}{(\log n)^2}
\qquad\text{or}\qquad
\beta_n \gg \frac{1}{(\log n)^2}.
\]
The conclusion follows.
\end{proof}

\subsubsection{Proof of Lemma \ref{lem:lower_bound}}

\begin{proof}
Without loss of generality, relabel the vertices so that
$d_1 \ge d_2 \ge \cdots \ge d_n$, and hence $S_k=\{1,\ldots,k\}$.
Since both $S_k$ and $\widetilde S_k$ have cardinality $k$,
\[
\frac12\,d_H(S_k,\widetilde S_k)
=
\sum_{i=1}^k \bone\{i\notin \widetilde S_k\}
=
\sum_{i=k+1}^n \bone\{i\in \widetilde S_k\}.
\]

Now let $t$ satisfy $\widetilde d_{(k+1)} \le t \le \widetilde d_{(k)}$ almost surely.
If $i\le k$ and $\widetilde d_i<t$, then necessarily $i\notin \widetilde S_k$, since every node in $\widetilde S_k$ has noisy degree at least $\widetilde d_{(k)}\ge t$. Likewise, if $i>k$ and $\widetilde d_i>t$, then necessarily $i\in \widetilde S_k$, since $\widetilde d_i>\widetilde d_{(k+1)}$ implies that $i$ must belong to the set of the $k$ largest noisy degrees. Therefore,
\[
d_H(S_k,\widetilde S_k)
\ge
2\max\left(
\sum_{i=1}^k \bone\{\widetilde d_i<t\},
\sum_{i=k+1}^n \bone\{\widetilde d_i>t\}
\right),
\]
which proves \eqref{eq:lower-bound-max}.

Taking expectations yields \eqref{eq:lower-bound-max-exp1}. Then \eqref{eq:lower-bound-max-exp2} follows from
\[
\bbE[\max(X,Y)]\ge \max\{\bbE X,\bbE Y\}.
\]
\end{proof}

\subsubsection{Proof of Theorem \ref{thm:ER_lower_bound}}

\begin{proof}
Without loss of generality, assume that $
d_1 \ge d_2 \ge \cdots \ge d_n,$
so that $
S_k=\{1,\ldots,k\}.$
Let
\[
t:=\widetilde d_{(k+1)}.
\]
Then \(t\) is measurable and satisfies
\[
\widetilde d_{(k+1)} \le t \le \widetilde d_{(k)}
\qquad \text{almost surely}.
\]
Hence, by Lemma~\ref{lem:lower_bound},
\begin{align}
\frac{1}{2}\,\bbE\!\left[d_H(S_k,\widetilde S_k)\right]
\ge
\sum_{i=1}^k \bbP(\widetilde d_i<t).
\label{eq:ER-proof-1}
\end{align}

Under the edge-flip noise model, for each \(1\le i<j\le n\),
\[
Y \sim \mathrm{Bernoulli}(\widetilde p_n),
\qquad
\widetilde p_n=p_n(1-\alpha_n-\beta_n)+\alpha_n,
\]
independently. Therefore,
\[
Y \sim \mathrm{ER}(n,\widetilde p_n).
\]
Let
\[
\widetilde q_n:=1-\widetilde p_n,
\]
and define
\[
\widetilde H^*
=
n \widetilde p_n
+\bigl(2 \widetilde p_n \widetilde q_n n \log n\bigr)^{1/2}
-\Bigl(\frac{\widetilde p_n \widetilde q_n n}{8 \log n}\Bigr)^{1/2}\log\log n,
\]
and
\[
\widetilde\Delta_n
:=
C(n)\Bigl(\frac{n\widetilde p_n}{\log n}\Bigr)^{1/2}.
\]
Applying Lemma~\ref{lem:ER_maximum} to the noisy graph \(Y\), we obtain
\[
\bbP\!\left(
\widetilde d_{(k+1)}\ge \widetilde H^*-\widetilde\Delta_n
\right)
=1-o(1).
\]
It follows that, for each \(1\le i\le k\),
\[
\bbP(\widetilde d_i<t)
\ge
\bbP\!\left(\widetilde d_i<\widetilde H^*-\widetilde\Delta_n\right)-o(1).
\]
Substituting this into \eqref{eq:ER-proof-1}, and using that \(k\) is fixed, yields
\begin{align}
\frac{1}{2}\,\bbE\!\left[d_H(S_k,\widetilde S_k)\right]
\ge
\sum_{i=1}^k
\bbP\!\left(\widetilde d_i<\widetilde H^*-\widetilde\Delta_n\right)
-o(1).
\label{eq:ER-proof-2}
\end{align}

Now define
\[
H^*
=
n p_n
+\bigl(2 p_n q_n n \log n\bigr)^{1/2}
-\Bigl(\frac{p_n q_n n}{8 \log n}\Bigr)^{1/2}\log\log n,
\qquad q_n:=1-p_n,
\]
and
\[
\Delta_n
:=
C(n)\Bigl(\frac{n p_n}{\log n}\Bigr)^{1/2}.
\]
Let
\[
\mathcal E_n
:=
\left\{
H^*+\Delta_n\ge d_1\ge \cdots \ge d_k\ge H^*-\Delta_n
\right\}.
\]
By Lemma~\ref{lem:ER_maximum},
\[
\bbP(\mathcal E_n)=1-o(1).
\]

Fix \(1\le i\le k\). Conditional on \(d_i\), we have
\[
\widetilde d_i
=
X_i+Y_i,
\]
where
\[
X_i\sim \mathrm{Bin}(d_i,1-\beta_n),
\qquad
Y_i\sim \mathrm{Bin}(n-1-d_i,\alpha_n),
\]
independently. Hence
\[
\mu(d_i):=\bbE(\widetilde d_i\mid d_i)
=
(1-\alpha_n-\beta_n)d_i+\alpha_n(n-1),
\]
and
\[
v(d_i):=\Var(\widetilde d_i\mid d_i)
=
d_i\beta_n(1-\beta_n)+(n-1-d_i)\alpha_n(1-\alpha_n).
\]

On the event \(\mathcal E_n\), we have
\[
d_i\ge H^*-\Delta_n
\qquad\text{and}\qquad
d_i\le H^*+\Delta_n.
\]
Since
\[
H^*
=
n p_n
+\bigl(2 p_n q_n n \log n\bigr)^{1/2}
+o(\sqrt{n\log n}),
\]
it follows that, on \(\mathcal E_n\),
\[
\frac{d_i}{n}
\ge
p_n-\Bigl(2p_nq_n\frac{\log n}{n}\Bigr)^{1/2}
\]
and
\[
\frac{n-1-d_i}{n}
\ge
q_n-\Bigl(2p_nq_n\frac{\log n}{n}\Bigr)^{1/2}
\]
for all sufficiently large \(n\). Therefore,
\[
v(d_i)
\ge
n\,c_n(\alpha_n,\beta_n)
\]
for all sufficiently large \(n\), where
\[
c_n(\alpha_n,\beta_n)
:=
\left[q_n-\left(2 p_n q_n \frac{\log n}{n}\right)^{1/2}\right]\alpha_n(1-\alpha_n)
+
\left[p_n-\left(2 p_n q_n \frac{\log n}{n}\right)^{1/2}\right]\beta_n(1-\beta_n).
\]

Also, on \(\mathcal E_n\),
\[
\mu(d_i)\le \mu(H^*+\Delta_n). 
\]
By definition,
\[
\mu(H^*+\Delta_n)
=
(1-\alpha_n-\beta_n)(H^*+\Delta_n)+\alpha_n(n-1).
\]
Hence
\[
\widetilde H^*-\mu(H^*+\Delta_n)
=
\alpha_n
-(1-\alpha_n-\beta_n)\Delta_n
+
\Bigl(\sqrt{2n\log n}-\sqrt{\tfrac{n}{8\log n}}\log\log n\Bigr)
\Bigl[
\sqrt{\widetilde p_n\widetilde q_n}
-(1-\alpha_n-\beta_n)\sqrt{p_nq_n}
\Bigr].
\]
Since
\[
\widetilde p_n=(1-\alpha_n-\beta_n)p_n+\alpha_n,
\qquad
\widetilde q_n=(1-\alpha_n-\beta_n)q_n+\beta_n,
\]
we have
\begin{align*}
\widetilde p_n\widetilde q_n
&=
\bigl((1-\alpha_n-\beta_n)p_n+\alpha_n\bigr)
\bigl((1-\alpha_n-\beta_n)q_n+\beta_n\bigr) \\
&=
(1-\alpha_n-\beta_n)^2p_nq_n
+(1-\alpha_n-\beta_n)p_n\beta_n \\
&\quad
+(1-\alpha_n-\beta_n)q_n\alpha_n
+\alpha_n\beta_n \\
&\ge
(1-\alpha_n-\beta_n)^2p_nq_n.
\end{align*}
Taking square roots yields
\[
\sqrt{\widetilde p_n\widetilde q_n}
\ge
(1-\alpha_n-\beta_n)\sqrt{p_nq_n}.
\]
It follows that
\[
\widetilde H^*-\mu(H^*+\Delta_n)
\ge
-(1-\alpha_n-\beta_n)\Delta_n.
\]
Therefore, on \(\mathcal E_n\),
\[
\widetilde H^*-\widetilde\Delta_n-\mu(d_i)
\ge
-(1-\alpha_n-\beta_n)\Delta_n-\widetilde\Delta_n.
\]
Since \(p_n\le 1\) and \(\widetilde p_n\le 1\),
\[
\Delta_n\le C(n)\sqrt{\frac{n}{\log n}},
\qquad
\widetilde\Delta_n\le C(n)\sqrt{\frac{n}{\log n}},
\]
and hence
\[
\widetilde H^*-\widetilde\Delta_n-\mu(d_i)
\ge
-2C(n)\sqrt{\frac{n}{\log n}}
\]
on \(\mathcal E_n\), for all sufficiently large \(n\).

Since \(\alpha_n \gtrsim \frac{\log n}{n}\) or \(\beta_n \gtrsim \frac{\log n}{n}\), the conditional variance tends to infinity. Thus, by Berry--Esseen for sums of independent Bernoulli random variables,
\[
\bbP\!\left(
\widetilde d_i<\widetilde H^*-\widetilde\Delta_n
\,\middle|\, d_i
\right)
\ge
\Phi\!\left(
\frac{\widetilde H^*-\widetilde\Delta_n-\mu(d_i)}{\sqrt{v(d_i)}}
\right)-o(1),
\]
uniformly for \(1\le i\le k\). On \(\mathcal E_n\), the bounds above imply
\[
\frac{\widetilde H^*-\widetilde\Delta_n-\mu(d_i)}{\sqrt{v(d_i)}}
\ge
-\frac{2C(n)}{\sqrt{c_n(\alpha_n,\beta_n)}\,\sqrt{\log n}}.
\]
Hence
\[
\bbP\!\left(
\widetilde d_i<\widetilde H^*-\widetilde\Delta_n
\,\middle|\, \mathcal E_n
\right)
\ge
\Phi\!\left(
-\frac{2C(n)}{\sqrt{c_n(\alpha_n,\beta_n)}\,\sqrt{\log n}}
\right)-o(1).
\]
Since \(\bbP(\mathcal E_n)=1-o(1)\), it follows that
\[
\bbP\!\left(
\widetilde d_i<\widetilde H^*-\widetilde\Delta_n
\right)
\ge
\Phi\!\left(
-\frac{2C(n)}{\sqrt{c_n(\alpha_n,\beta_n)}\,\sqrt{\log n}}
\right)-o(1),
\]
uniformly for \(1\le i\le k\). Substituting this into \eqref{eq:ER-proof-2}, we obtain
\[
\frac{1}{2}\,\bbE\!\left[d_H(S_k,\widetilde S_k)\right]
\ge
k\,
\Phi\!\left(
-\frac{2C(n)}{\sqrt{c_n(\alpha_n,\beta_n)}\,\sqrt{\log n}}
\right)
-o(1),
\]
which completes the proof.
\end{proof}

\subsubsection{Proof of Lemma \ref{lem:upper_bound}}

\begin{proof}
The argument is analogous to that of Lemma~\ref{lem:lower_bound}. Since
\[
\frac12\,d_H(S_k,\widetilde S_k)
=
\sum_{i=1}^k \bone\{i\notin \widetilde S_k\}
=
\sum_{i=k+1}^n \bone\{i\in \widetilde S_k\},
\]
it suffices to note that if $i\le k$ and $i\notin \widetilde S_k$, then necessarily $\widetilde d_i\le t$, while if $i>k$ and $i\in \widetilde S_k$, then necessarily $\widetilde d_i\ge t$. This gives \eqref{eq:upper-bound-min}, and taking expectations yields \eqref{eq:upper-bound-min-exp}.
\end{proof}

\subsubsection{Proof of Theorem \ref{thm:PA_upper_bound}}
\begin{proof}
    The proof follows the same argument as that of Theorem~\ref{thm:PA_joint}, and is therefore omitted.
\end{proof}

\subsection{Eigenvector–based results (Section~\ref{sec:EC})}

\subsubsection{Proof of Theorem~\ref{thm:evec-infty-perturb}}

\begin{proof}
\textbf{Proof sketch.} 
Fix $l\in[n]$ and let $\tilde{\bx}^{(l)}$ be the leading eigenvector of the leave-one-out matrix $\Yb^{(l)}$.
Using
\[
|x_l-\tilde x_l|\le |x_l-\tilde x_l^{(l)}|+\|\tilde{\bx}^{(l)}-\tilde{\bx}\|_2,
\]
and then maximizing over $l$, it suffices to control the two terms on the right-hand side.

\paragraph{Step 1: Leave-one-out decomposition.}
We begin by introducing the notation used throughout the proof. Let the error matrix be
\[
\Eb = \Yb - \Ab,
\]
which represents the perturbation of $\Ab$ due to edge noise. 

For any $l \in [n]$, define $\Eb^{(l)}$ as the modified error matrix obtained by
setting the $l$-th row and $l$-th column of $\Eb$ to zero, that is,
\[
E_{ij}^{(l)} =
\begin{cases}
E_{ij}, & \text{if } i \neq l \text{ and } j \neq l, \\
0,      & \text{otherwise}.
\end{cases}
\]
Let
\[
\Yb^{(l)} := \Ab + \Eb^{(l)}
\]
be the corresponding leave-one-out version of $\Yb$, and denote by
$\tilde{\bx}^{(l)}$ the leading eigenvector of $\Yb^{(l)}$ (with
eigenvalue $\tilde{\lambda}_1^{(l)}$), normalized in the same way as
$\tilde{\bx}$. 

By definition of the $\ell_\infty$-norm,
\[
\|\bx - \tilde{\bx}\|_\infty
= \max_{l \in [n]} |x_l - \tilde{x}_l|.
\]
For each $l \in [n]$, we insert the leave-one-out eigenvector
$\tilde{\bx}^{(l)}$ and apply the triangle inequality:
\begin{align*}
|x_l - \tilde{x}_l|
&\le \big|x_l - \tilde{x}_l^{(l)}\big|
     + \big|\tilde{x}_l^{(l)} - \tilde{x}_l\big| \\
&\le \big|x_l - \tilde{x}_l^{(l)}\big|
     + \|\tilde{\bx}^{(l)} - \tilde{\bx}\|_2.
\end{align*}

Taking the maximum over $l \in [n]$ yields
\[
\|\bx - \tilde{\bx}\|_\infty
\le \max_{l \in [n]} \big|x_l - \tilde{x}_l^{(l)}\big|
  + \max_{l \in [n]} \|\tilde{\bx}^{(l)} - \tilde{\bx}\|_2.
\]
Thus, it suffices to bound the two quantities
\[
\max_{l \in [n]} \big|x_l - \tilde{x}_l^{(l)}\big|
\quad \text{and} \quad
\max_{l \in [n]} \|\tilde{\bx}^{(l)} - \tilde{\bx}\|_2,
\]
which will be handled in Steps~2 and~3 below.

\paragraph{Step 2: Bounding $\|\tilde{\bx}^{(l)} - \tilde{\bx}\|_2$.}
We apply an advanced leave-one-out technique. We view $\Yb$ as a perturbation of $\Yb^{(l)}$:
\[
\Yb = \Yb^{(l)} + \Eb^{(-l)},
\]
where $\Eb^{(-l)}$ retains only the $l$-th row and column of $\Eb$ and
zeros out all others. By construction, $\Eb^{(-l)}$ and
$\tilde{\bx}^{(l)}$ are independent. 

Let $\tilde{\lambda}_1^{(l)}$ and $\tilde{\lambda}_2^{(l)}$ denote the two
largest eigenvalues of $\Yb^{(l)}$.
Applying the Davis-Kahan sin-$\Theta$ theorem (e.g., \citet[Corollary~2.8]{chen2021spectral}),
\begin{equation}
\|\tilde{\bx}^{(l)} - \tilde{\bx}\|_2
\;\le\;
\frac{2\|\Eb^{(-l)}\tilde{\bx}^{(l)}\|_2}
     {\tilde{\lambda}_1^{(l)} - \tilde{\lambda}_2^{(l)}}.
\label{2.1}    
\end{equation}
Therefore we need to bound:
(i) $\|\Eb^{(-l)}\tilde{\bx}^{(l)}\|_2$ and
(ii) $\tilde{\lambda}_1^{(l)} - \tilde{\lambda}_2^{(l)}$ from below.
We control these two terms separately.

\paragraph{Step 2.1: Bounding $\|\Eb^{(-l)}\tilde{\bx}^{(l)}\|_2$.}
The tightness of the bound on $\|\Eb^{(-l)}\tilde{\bx}^{(l)}\|_2$ comes from
the leave-one-out construction: by design, $\Eb^{(-l)}$ depends only on the
$l$-th row and column of $\Eb$, while $\tilde{\bx}^{(l)}$ is the leading
eigenvector of $\Yb^{(l)}$ and therefore does not depend on that row and
column. Hence $\Eb^{(-l)}$ and $\tilde{\bx}^{(l)}$ are independent.

By definition,
\begin{align*}
    \Eb^{(-l)}\tilde{\bx}^{(l)}
    &= \big(\Yb - \Yb^{(l)}\big)\tilde{\bx}^{(l)} \\
    &= (\Eb_{l,\cdot} \tilde{\bx}^{(l)})\,\be_l
       + \tilde x^{(l)}_l\big(\Eb_{\cdot,l} - E_{ll}\be_l\big),
\end{align*}
where $\be_l\in\mathbb{R}^n$ is the $l$-th standard basis vector (its $l$-th
component is $1$ and all others are $0$), $\Eb_{l,\cdot}$ denotes the $l$-th
row of $\Eb$, $\Eb_{\cdot,l}$ denotes the $l$-th column of $\Eb$, and $E_{ll}$
is the $(l,l)$-entry of $\Eb$.

Under our noisy network model, we can write the entries of $\Eb$ as
\[
E_{ij} = A_{ij}\big(\mathbbm{1}\{\varepsilon_{ij} = 0\} - 1\big)
        + \mathbbm{1}\{\varepsilon_{ij} = 1\},
\]
where $\{\varepsilon_{ij}\}_{1\le i<j\le n}$ are independent random variables
with
\[
\mathbb{P}(\varepsilon_{ij}=1)=\alpha_n,\quad
\mathbb{P}(\varepsilon_{ij}=0)=1-\alpha_n-\beta_n,\quad
\mathbb{P}(\varepsilon_{ij}=-1)=\beta_n,
\]
and $\alpha_n,\beta_n\ge 0$, $\alpha_n+\beta_n<1$, as in \citet[Assumption~1]{chang2022estimation}.

\paragraph{(a) Concentration for $\big|\Eb_{l,\cdot}\tilde{\bx}^{(l)}\big|$.}
Since $\{E_{lj}\}_{j\neq l}$ are independent, and $\Eb_{l,\cdot}$ and $\tilde{\bx}^{(l)}$ are independent
(by the leave-one-out construction), we may view $\{\tilde{x}^{(l)}_j\}$ as deterministic when conditioning on $\tilde{\bx}^{(l)}$. Thus we apply Bernstein’s inequality to the sum
\[
\sum_{j\neq l} E_{lj}\tilde{x}^{(l)}_j.
\]
Using $|E_{lj}| \le 1$ and
\[
\Var\!\left(\sum_{j\neq l}E_{lj}\tilde{x}^{(l)}_j\right)
\le
\sum_{j\neq l}\Var(E_{lj})(\tilde{x}^{(l)}_j)^2
\le \big[(\alpha_n+\beta_n) - (\alpha_n-\beta_n)^2\big]
\sum_{j\neq l}(\tilde{x}^{(l)}_j)^2
\le (\alpha_n+\beta_n) - (\alpha_n-\beta_n)^2,
\]
Bernstein’s inequality gives, for any $\varepsilon>0$,
\[
\mathbb{P}\!\left(
\left|\sum_{j\neq l}E_{lj}\tilde{x}^{(l)}_j
- \mathbb{E}\!\left[\sum_{j\neq l}E_{lj}\tilde{x}^{(l)}_j\right]
\right| > \varepsilon
\right)
\le
2\exp\!\left(
-\frac{\varepsilon^2/2}{
[(\alpha_n+\beta_n)-(\alpha_n-\beta_n)^2]+\varepsilon/3}
\right).
\]
Choosing
\[
\varepsilon = 5\sqrt{\big[(\alpha_n+\beta_n)-(\alpha_n-\beta_n)^2\big]\log n}
\]
and using
\[
\left|\mathbb{E}\!\left[\sum_{j\neq l}E_{lj}\tilde{x}^{(l)}_j\right]\right|
\le (\alpha_n+\beta_n)\sqrt{n},
\]
we obtain, for each fixed $l$,
\[
\big|\Eb_{l,\cdot}\tilde{\bx}^{(l)}\big|
\le
(\alpha_n+\beta_n)\sqrt{n}
+
5\sqrt{\big[(\alpha_n+\beta_n)-(\alpha_n-\beta_n)^2\big]\log n}
\]
with probability at least $1-n^{-10}$. A union bound over $l\in[n]$ then yields
\[
\max_{l\in[n]}\big|\Eb_{l,\cdot}\tilde{\bx}^{(l)}\big|
\le
(\alpha_n+\beta_n)\sqrt{n}
+
5\sqrt{\big[(\alpha_n+\beta_n)-(\alpha_n-\beta_n)^2\big]\log n}
\qquad\text{with probability } \ge 1-n^{-9}.
\]

\paragraph{(b) Concentration for $\|\Eb_{\cdot,l} - E_{ll}\be_l\|_2$.}
Since removing a coordinate can only decrease the $\ell_2$-norm,
\[
\|\Eb_{\cdot,l} - E_{ll}\be_l\|_2
\le \|\Eb_{\cdot,l}\|_2
= \Big(\sum_{j\neq l} E_{jl}^2\Big)^{1/2}.
\]
Under the noisy-edge model, $E_{jl}^2 \in \{0,1\}$ for $j\neq l$, and
\[
\mathbb{P}(E_{jl}^2 = 1) = \mathbb{P}(E_{jl} \neq 0) \le \alpha_n+\beta_n.
\]
Hence $\{E_{jl}^2\}_{j\neq l}$ are independent Bernoulli random variables with
\[
\mathbb{E}\!\left[\sum_{j\neq l} E_{jl}^2\right]
= \sum_{j\neq l}\mathbb{E}(E_{jl}^2)
\le n(\alpha_n+\beta_n).
\]
Applying Bernstein’s inequality to the sum
$\sum_{j\neq l}E_{jl}^2$ yields
\[
\mathbb{P}\!\left(
\sum_{j\neq l}E_{jl}^2
> 2n(\alpha_n+\beta_n) + \log n
\right)
\le n^{-10},
\]
for all sufficiently large $n$. Therefore, with probability at least $1-n^{-10}$,
\[
\|\Eb_{\cdot,l} - E_{ll}\be_l\|_2
\le
\sqrt{2n(\alpha_n+\beta_n) + \log n}.
\]
A union bound over $l \in [n]$ then gives
\[
\max_{l\in[n]}\|\Eb_{\cdot,l} - E_{ll}\be_l\|_2
\le
\sqrt{2n(\alpha_n+\beta_n) + \log n}
\qquad\text{with probability }\ge 1-n^{-9}.
\]

Combining (a)–(b), we conclude that with probability at least \(1-2n^{-9}\), for every \(l\in[n]\),
\[
\big\|\Eb^{(-l)}\tilde{\bx}^{(l)}\big\|
\le
(\alpha_n+\beta_n)\sqrt{n}
+
5\sqrt{\big[(\alpha_n+\beta_n)-(\alpha_n-\beta_n)^2\big]\log n}
+
\sqrt{2n(\alpha_n+\beta_n)+\log n}
\left(
\|\tilde{\bx}\|_{\infty}
+
\|\tilde{\bx}-\tilde{\bx}^{(l)}\|_2
\right).
\]
Consequently,
\begin{align*}
    & \max_{l\in[n]}\big\|\Eb^{(-l)}\tilde{\bx}^{(l)}\big\| \\
\le &
(\alpha_n+\beta_n)\sqrt{n}
+
5\sqrt{\big[(\alpha_n+\beta_n)-(\alpha_n-\beta_n)^2\big]\log n}
+
\sqrt{2n(\alpha_n+\beta_n)+\log n}
\left(
\|\tilde{\bx}\|_{\infty}
+
\max_{l\in[n]}\|\tilde{\bx}-\tilde{\bx}^{(l)}\|_2
\right).
\end{align*}

\paragraph{Step 2.2: Spectral gap of $\Yb^{(l)}$.}
To lower bound the spectral gap of $\Yb^{(l)}$, we relate its eigenvalues 
to those of $\Ab$ using Weyl’s inequality:
\[
|\lambda_i - \tilde{\lambda}_i^{(l)}|
\le \|\Eb^{(l)}\| \le \|\Eb\|,
\qquad i=1,2,
\]
where $\lambda_i$ and $\tilde{\lambda}_i^{(l)}$ are the $i$-th largest eigenvalues 
of $\Ab$ and $\Yb^{(l)}$, respectively.  
By Lemma~\ref{lem:E_concentration}, with probability at least $1-n^{-10}$,
\[
\|\Eb\| \le B_3,
\]
where $B_3$ is defined in~\eqref{eq:def-B123}.  
Therefore,
\[
\tilde{\lambda}_1^{(l)} - \tilde{\lambda}_2^{(l)}
\ge (\lambda_1 - B_3) - (\lambda_2 + B_3)
= \lambda_1 - \lambda_2 - 2B_3.
\]

Combining the bound from Step~2.1 with \eqref{2.1} and the estimate
\[
\tilde{\lambda}_1^{(l)} - \tilde{\lambda}_2^{(l)}
\ge \lambda_1 - \lambda_2 - 2B_3,
\]
we obtain
\[
\|\tilde{\bx}^{(l)}-\tilde{\bx}\|_2
\le
\frac{
2B_1+2B_2\|\tilde{\bx}\|_\infty
+2B_2\|\tilde{\bx}^{(l)}-\tilde{\bx}\|_2
}{
\lambda_1-\lambda_2-2B_3
}.
\]
Rearranging yields
\begin{equation}\label{eq:evec_bound}
\|\tilde{\bx}^{(l)} - \tilde{\bx}\|_2
\le
\frac{
    2B_1 + 2B_2 \|\tilde{\bx}\|_{\infty}
}{
    \lambda_1 - \lambda_2 - 2B_3 - 2B_2
},
\end{equation}
with probability at least \(1-Cn^{-9}\) for some absolute constant \(C>0\), provided that
\[
\lambda_1 - \lambda_2 > 2B_3 + 2B_2.
\]

This completes the control of $\|\tilde{\bx}^{(l)} - \tilde{\bx}\|_2$.

\paragraph{Step 3: Bounding $\max_{l\in[n]} |x_l - \tilde{x}_l^{(l)}|$.}
We now control the first term in the decomposition from Step~1.  
Recall that $\tilde{\bx}^{(l)}$ satisfies the eigen-equation
\[
\Yb^{(l)}\tilde{\bx}^{(l)} = \tilde{\lambda}_1^{(l)} \tilde{\bx}^{(l)}.
\]
In particular, its $l$-th coordinate obeys
\[
\tilde{\lambda}_1^{(l)} \tilde{x}_l^{(l)}
= \Yb^{(l)}_{l,\cdot}\tilde{\bx}^{(l)}
= \Ab_{l,\cdot}\tilde{\bx}^{(l)},
\qquad\Rightarrow\qquad
\tilde{x}_l^{(l)}
= \frac{1}{\tilde{\lambda}_1^{(l)}}\,\Ab_{l,\cdot}\tilde{\bx}^{(l)}.
\]
Similarly, since $\Ab\bx = \lambda_1\bx$,
\[
x_l = \frac{1}{\lambda_1}\,\Ab_{l,\cdot}\bx.
\]

Subtracting the above displays yields
\[
x_l - \tilde{x}_l^{(l)}
= \frac{1}{\lambda_1}\Ab_{l,\cdot}(\bx - \tilde{\bx}^{(l)})
\;-\;
\frac{\lambda_1 - \tilde{\lambda}_1^{(l)}}{\lambda_1\tilde{\lambda}_1^{(l)}}\,
\Ab_{l,\cdot}\tilde{\bx}^{(l)}
=: T_{1,l} + T_{2,l}.
\]

\smallskip
\noindent
\textbf{Bound on $T_{1,l}$.}
We insert the identity $\Ab_{l,\cdot} = \lambda_1 x_l \bx^\top + (\Ab_{l,\cdot} - \lambda_1 x_l \bx^\top)$:
\begin{align*}
T_{1,l}
&=
\frac{1}{\lambda_1}
\big(\Ab_{l,\cdot}-\lambda_1 x_l\bx^\top\big)(\bx - \tilde{\bx}^{(l)})
+ x_l \bx^\top(\bx - \tilde{\bx}^{(l)}).
\end{align*}
Using $\|\Ab_{l,\cdot}-\lambda_1 x_l\bx^\top\|\le \lambda_2$ and $\|\bx\|_2=1$, we obtain
\[
|T_{1,l|}
\le
\Big(\frac{\lambda_2}{\lambda_1}+\|x\|_\infty\Big)
\|\bx - \tilde{\bx}^{(l)}\|.
\]
Then Davis–Kahan and $\|\Eb^{(l)}\|\le B_3$ give
\[
\|\bx - \tilde{\bx}^{(l)}\|
\le \frac{2\|\Eb^{(l)}\|}{\lambda_1 - \lambda_2}
\le \frac{2B_3}{\lambda_1 - \lambda_2}.
\]
Thus
\[
|T_{1,l|}
\le
\Big(\frac{\lambda_2}{\lambda_1}+\|x\|_\infty\Big)
\frac{2B_3}{\lambda_1 - \lambda_2}.
\]

\smallskip
\noindent
\textbf{Bound on $T_{2,l}$.}
We use
\[
\left\|\frac{1}{\lambda_1}\Ab_{l,\cdot}\right\|
\le \frac{\lambda_2}{\lambda_1} + \|x\|_\infty,
\qquad
|\lambda_1 - \tilde{\lambda}_1^{(l)}|
\le \|\Eb^{(l)}\| \le B_3,
\]
which implies $\tilde{\lambda}_1^{(l)} \ge \lambda_1 - B_3$ whenever $\lambda_1>B_3$.  
Hence
\[
|T_{2,l|}
\le \frac{B_3}{\lambda_1 - B_3}
\Big(\frac{\lambda_2}{\lambda_1} + \|x\|_\infty\Big).
\]

\smallskip
\noindent
\textbf{Collecting terms.}
Combining the bounds on $T_{1,l}$ and $T_{2,l}$ yields
\[
|x_l - \tilde{x}_l^{(l)}|
\le
\Big(\frac{\lambda_2}{\lambda_1}+\|x\|_\infty\Big)
\left(
\frac{2B_3}{\lambda_1 - \lambda_2}
+
\frac{B_3}{\lambda_1 - B_3}
\right),
\qquad l\in[n].
\]
Taking the maximum over $l$ proves the desired bound.

\paragraph{Step 4: Final combination.}
Substituting the bounds from Step~2 and Step~3 into the leave-one-out
decomposition, we obtain
\begin{align*}
    \|\bx - \tilde{\bx}\|_{\infty} 
    &\le \max_{l\in[n]} \left|x_{l} - \tilde{x}_{l}^{(l)} \right|
       + \max_{l\in[n]}\|\tilde{\bx}^{(l)} - \tilde{\bx}\|_2 \\
    &\le \left( \frac{\lambda_2 }{\lambda_1} + \|\bx\|_{\infty} \right)
           \left( \frac{2B_3}{\lambda_1 - \lambda_2}
                + \frac{B_3}{\lambda_1 - B_3} \right)
       + \frac{2B_1 + 2B_2 \|\tilde{\bx}\|_{\infty}}
              {\lambda_1 - \lambda_2 - 2B_3 - 2B_2}.
\end{align*}
Using $\|\tilde{\bx}\|_{\infty} \le \|\bx\|_{\infty}
     + \|\bx - \tilde{\bx}\|_{\infty}$, we have
\[
\|\bx - \tilde{\bx}\|_{\infty} 
\le \left( \frac{\lambda_2 }{\lambda_1} + \|\bx\|_{\infty} \right)
     \left( \frac{2B_3}{\lambda_1 - \lambda_2}
          + \frac{B_3}{\lambda_1 - B_3} \right)
     + \frac{2B_1 + 2B_2 \big(\|\bx\|_{\infty}
                           + \|\bx - \tilde{\bx}\|_{\infty}\big)}
            {\lambda_1 - \lambda_2 - 2B_3 - 2B_2}.
\]
Rearranging terms, and assuming
$\lambda_1 - \lambda_2 > 2B_3 + 4B_2$
so that all denominators are positive, gives
\begin{equation}\label{eq:entrywise_final_bound}
\begin{aligned}
\|\bx - \tilde{\bx}\|_{\infty}
&\le
\frac{\lambda_1 - \lambda_2 - 2B_3 - 2B_2}
     {\lambda_1 - \lambda_2 - 2B_3 - 4B_2}
\left(\frac{\lambda_2}{\lambda_1}
      + \|\bx\|_{\infty}\right)
\left(
    \frac{2B_3}{\lambda_1 - \lambda_2}
    + \frac{B_3}{\lambda_1 - B_3}
\right)
\\[4pt]
&\quad+
\frac{
    2B_1 + 2B_2\|\bx\|_{\infty}
}{
    \lambda_1 - \lambda_2 - 2B_3 - 4B_2
}.
\end{aligned}
\end{equation}
Thus the entrywise perturbation is controlled whenever the eigen-gap $\lambda_1 - \lambda_2$ dominates the noise level encoded by $B_1,B_2,B_3$. This completes the proof.

\end{proof}

Define the noise matrix $\Eb = \Yb - \Ab$. Under the noisy edge model
\eqref{eqn:noise_model}, the following bound holds.

\begin{lem}\label{lem:E_concentration}
Let $\|\cdot\|$ denote the spectral norm. Then, conditional on $\Ab$,
for all sufficiently large $n$,
\[
  \|\Eb\|
  \;\le\;
  5\sqrt{n(\alpha_n + \beta_n)}
  \;+\;
  \alpha_n n
  \;+\;
  (\alpha_n + \beta_n)\,\|\Ab\|
\]
with probability at least $1 - n^{-10}$.
\end{lem}

\begin{proof}
We work conditional on $\Ab$. Recall $\Eb=\Yb-\Ab$. Decompose
\[
\Eb
=\underbrace{\Yb-\bbE[\Yb\mid \Ab]}_{=: \Zb}
+\underbrace{\bbE[\Yb\mid \Ab]-\Ab}_{=: \Bb}.
\]
Hence, by the triangle inequality,
\begin{equation}\label{eq:triangle_E}
\|\Eb\|\le \|\Zb\|+\|\Bb\|.
\end{equation}

\medskip\noindent
\textbf{Step 1: Bound $\|\Bb\|$.}
Under~\eqref{eqn:noise_model}, for $i\neq j$,
\[
\bbE[Y_{ij}\mid A_{ij}=0]=\alpha_n,
\qquad
\bbE[Y_{ij}\mid A_{ij}=1]=1-\beta_n.
\]
Let $\Jb$ be the all-ones matrix and $\Ib$ the identity. Then
\begin{align*}
\bbE[\Yb\mid \Ab]
&=\alpha_n(\Jb-\Ib-\Ab)+(1-\beta_n)\Ab \\
&=\alpha_n(\Jb-\Ib)+(1-\alpha_n-\beta_n)\Ab,
\end{align*}
so
\[
\Bb:=\bbE[\Yb\mid \Ab]-\Ab=\alpha_n(\Jb-\Ib)-(\alpha_n+\beta_n)\Ab.
\]
Therefore,
\[
\|\Bb\|
\le \alpha_n\|\Jb-\Ib\|+(\alpha_n+\beta_n)\|\Ab\|.
\]
Since $\|\Jb-\Ib\|=n-1\le n$, we obtain
\begin{equation}\label{eq:bias_bound_simple}
\|\Bb\|\le \alpha_n n+(\alpha_n+\beta_n)\|\Ab\|.
\end{equation}

\medskip\noindent
\textbf{Step 2: Bound $\|\Zb\|$.}
Conditional on $\Ab$, the entries $\{Z_{ij}\}_{1\le i<j\le n}$ are independent,
mean-zero, and satisfy $|Z_{ij}|\le 1$. Moreover,
\[
\Var(Z_{ij}\mid \Ab)=\Var(Y_{ij}\mid \Ab)\le \alpha_n+\beta_n,
\]
since $\Var(Y_{ij}\mid A_{ij}=0)=\alpha_n(1-\alpha_n)\le \alpha_n$ and
$\Var(Y_{ij}\mid A_{ij}=1)=(1-\beta_n)\beta_n\le \beta_n$.
Hence for each $i$,
\[
\sum_{j=1}^n \bbE[Z_{ij}^2\mid \Ab]\le n(\alpha_n+\beta_n).
\]
A matrix Bernstein inequality for symmetric matrices with independent centered bounded
entries (e.g., Lemma~2.1 of \cite{chen2021spectral}) yields that, for all sufficiently large $n$,
\begin{equation}\label{eq:Z_bound_5}
\|\Zb\|\le 5\sqrt{n(\alpha_n+\beta_n)}
\qquad\text{with probability at least } 1-n^{-10}.
\end{equation}

\medskip\noindent
\textbf{Step 3: Combine.}
Combining \eqref{eq:triangle_E}, \eqref{eq:bias_bound_simple}, and \eqref{eq:Z_bound_5} gives that,
with probability at least $1-n^{-10}$,
\[
\|\Eb\|
\le 5\sqrt{n(\alpha_n+\beta_n)}
+\alpha_n n
+(\alpha_n+\beta_n)\|\Ab\|,
\]
as claimed.
\end{proof}

\subsubsection{Proof of Theorem~\ref{thm:EC_consistency}}
\begin{proof}
We apply Corollary~\ref{cor:topk-consistency} to the linear PA model.
It suffices to (i) control the perturbation level $\epsilon_n$, and
(ii) verify the eigengap condition $\Delta_1^\lambda > 2\epsilon_n$.

\paragraph{Step 1: Spectral scale under preferential attachment.}
From Lemma~\ref{lem:PA_limits} with \(k=2\), we have
\[
\left(
\frac{\lambda_1}{n^{1/(2(\delta+2))}},
\frac{\lambda_2}{n^{1/(2(\delta+2))}}
\right)
\convd
\left(\sqrt{\nu_1},\sqrt{\nu_2}\right),
\]
for some non-degenerate random variables \(\nu_1 \ge \nu_2 >0\).
In particular,
\[
\lambda_1=\Theta_p\!\left(n^{1/(2(\delta+2))}\right),
\qquad
\lambda_2=\Theta_p\!\left(n^{1/(2(\delta+2))}\right),
\]
and therefore
\[
\frac{\lambda_2}{\lambda_1}=\Theta_p(1).
\]
Moreover, if \(\bbP(\nu_1>\nu_2)=1\), then by the continuous mapping theorem,
\[
\frac{\lambda_1-\lambda_2}{n^{1/(2(\delta+2))}}
\convd
\sqrt{\nu_1}-\sqrt{\nu_2},
\]
and hence
\[
\lambda_1-\lambda_2=\Theta_p\!\left(n^{1/(2(\delta+2))}\right).
\]
\paragraph{Step 2: Control of the eigenvector norm.}
Under Assumption~\ref{asm:localization}, the principal eigenvector $\bx$ satisfies
$x_h \to c_0 > 0$.
Since $\|\bx\|_2 = 1$, this implies
\[
\|\bx\|_\infty \ge x_h = \Theta_p(1),
\qquad
\|\bx\|_\infty \le 1,
\]
and hence
\[
\|\bx\|_\infty = \Theta_p(1).
\]

\paragraph{Step 3: Perturbation bound.}
Recall the quantities $B_1,B_2,B_3$ defined in \eqref{eq:def-B123}.
Define the event
\[
\mathcal E_n
:=
\bigl\{
\lambda_1 - \lambda_2 - 2B_3 - 4B_2 > 0,
\;
\lambda_1 - B_3 > 0
\bigr\}.
\]
Under the rate conditions stated below, $\bbP(\mathcal E_n) \to 1$.

On the event $\mathcal E_n$, the perturbation bound in
Corollary~\ref{cor:topk-consistency} yields
\begin{align*}
\epsilon_n
&=
\frac{\lambda_1 - \lambda_2 - 2B_3 - 2B_2}
     {\lambda_1 - \lambda_2 - 2B_3 - 4B_2}
\Bigg(\frac{\lambda_2}{\lambda_1} + \|\bx\|_{\infty}\Bigg)
\Bigg(
    \frac{2B_3}{\lambda_1 - \lambda_2}
    + \frac{B_3}{\lambda_1 - B_3}
\Bigg)
+
\frac{
    2B_1 + 2B_2\|\bx\|_{\infty}
}{
    \lambda_1 - \lambda_2 - 2B_3 - 4B_2
}.
\end{align*}
Using $\lambda_2/\lambda_1 = \Theta_p(1)$, $\|\bx\|_\infty = \Theta_p(1)$, and
$\lambda_1 \asymp \lambda_1-\lambda_2 \asymp n^{1/(2(\delta+2))}$ in probability,
we obtain the simplified bound
\[
\epsilon_n
=
O_p\!\left(
\frac{B_1 + B_2 + B_3}{\lambda_1 - \lambda_2}
\right)
=
O_p\!\left(
\frac{B_1 + B_2 + B_3}{n^{1/(2(\delta+2))}}
\right).
\]

\paragraph{Step 4: Verification of the eigengap condition.}
A sufficient condition for exact recovery of the top-ranked node is
\[
\Delta_1^\lambda > 2\epsilon_n.
\]
Under Assumption~\ref{asm:hub-separation}, the centrality gap
\[
\Delta_1^\lambda := x_{(1)} - x_{(2)}
\]
satisfies \(\Delta_1^\lambda \ge c_1\) with probability \(1-o(1)\). It suffices that
\[
B_1 + B_2 + B_3 = o\!\left(n^{1/(2(b+2))}\right).
\]
From the definitions in \eqref{eq:def-B123} and the fact that
$\|\Ab\| = \Theta_p(n^{1/(2(b+2))})$, this holds provided
\[
\alpha_n \ll n^{-(b+3/2)/(b+2)},
\qquad
\beta_n \ll n^{-(b+1)/(b+2)}.
\]
Under these conditions, $\epsilon_n = o_p(1)$ and therefore
\[
\bbP(S_1 = \tilde S_1) = 1 - o(1).
\]

\paragraph{Failure for $k\ge2$.}
For $k\ge2$, the condition $\Delta_k^\lambda>2\epsilon_n$ typically fails in the PA model: the sub-leading eigenvector centralities exhibit vanishing (or highly fluctuating) gaps, even when $\epsilon_n=o_p(1)$. Hence Corollary~\ref{cor:topk-consistency} does not yield consistency for $\tilde S_k$, and in general one should not expect exact recovery of $S_k$ for fixed $k\ge2$.

\end{proof}

\section{Technical Lemmas and Concentration Bounds}
\label{sec:appendix-lemmas}
%%%%%%%%%%%%%%%%%%%%%%%%%%%%%%%%%%%%%%%%%%%%%%%%%%%%

In this section, we present several technical lemmas that are used in the proofs of the main results. Throughout the section, everything is conditional on $\Ab$. Recall for a fixed node $i \in [n]$, the noisy degree is $\tilde{d}_i := \sum_{j \ne i} Y_{ij}$, where each $Y_{ij}$ is generated from the noise model in~\eqref{eqn:noise_model}:
$$
Y_{ij} \mid A_{ij} \sim 
\begin{cases}
\operatorname{Bern}(1 - \beta_n), & \text{if } A_{ij} = 1, \\
\operatorname{Bern}(\alpha_n), & \text{if } A_{ij} = 0.
\end{cases}
$$
Define
$$\tilde{d}_i^s = \frac{1}{\sigma_{n,i}}(\tilde{d}_i - \mu_{n,i}), \quad i = 1,\ldots,k,$$
where $\mu_{n,i}$ and $\sigma_{n,i}$ are the mean and standard deviation of $\tilde{d}_i$, respectively with
\begin{align*}
\mu_{n,i} &= \bbE\left(\tilde{d}_i \right) = (n-1-d_{i}) \alpha_n  + d_{i} (1 - \beta_n), \\
\sigma_{n, i}^2 &= \operatorname{Var} \left(\tilde{d}_i \right) = (n-1-d_{i}) \alpha_n (1 - \alpha_n) + d_{i} \beta_n (1 - \beta_n).   
\end{align*}

\begin{lem}[Joint CLT for normalized noisy degrees]\label{lem:min_CLT}
Assume $m_n:=\min_{1\le i\le k}\sigma_{n,i}\to\infty$. 
The following joint and high-dimensional central limit results hold:  
\begin{enumerate}
    \item 
    If $k$ is fixed, 
    $$
    (\tilde{d}_1^s, \ldots, \tilde{d}_k^s) \convd \cN(\bzero, \bI_k),
    $$
    as $n\rightarrow\infty$ where $\bI_k$ is the $k \times k$ identity matrix. 
    \item 
    If $k = k_n$ satisfies $k_n = o(m_n^2)$, then for every deterministic 
    $\bt^{(n)} \in \mathbb{R}^{k_n}$ with $\|\bt^{(n)}\|_2 \to \tau$,
    \[
    \left\langle \bt^{(n)}, (\tilde d_1^s, \ldots, \tilde d_{k_n}^s) \right\rangle
    \ \convd \ \mathcal N(0, \tau^2).
    \]
    \item 
    Let $k = k_n$ satisfy $\log^7(k_n n) = o(m_n^2)$.
    Let $a_{n,i} = (c - \mu_{n,i})/\sigma_{n,i}$.
    Then
    \[
    \bbP\!\left( \min_{i \le k_n} \tilde d_i^s < c \right)
    =
    \bbP\!\left( \exists i \in [k_n] : Z_i < a_{n,i} \right)
    + r_n,
    \]
    where $\Zb = (Z_1,\ldots,Z_{k_n}) \sim \mathcal N(\bzero, I_{k_n})$ and
    \[
    r_n =
    \begin{cases}
    O(m_n^{-1}), & k_n \equiv k \text{ fixed},\\[4pt]
    O\!\left( \left( \dfrac{\log^7(k_n n)}{m_n^2} \right)^{1/6} \right),
    & k_n \to \infty.
    \end{cases}
    \]
\end{enumerate}

\end{lem}

\begin{proof}
{\bf Proof of Conclusion (1)}

The proof proceeds by rewriting the normalized noisy degree as sum of independent vectors and applying Lindeberg-Feller central limit theorem. 

{\bf Step 1: Reduction to one-dimensional convergence}

By the Cramér-Wold theorem, it is sufficient to show that for any fixed $\bt = (t_1, \ldots, t_k) \in \mathbb{R}^k$,
\begin{equation}\label{eqn:1940}
    \bt^\top \tilde{\bd}^s \convd \mathcal{N}(0, \|\bt\|_2^2).
\end{equation}

{\bf Step 2: Construction of independent vectors}

We construct auxiliary vectors $\Xb_{ij} \in \bbR^k$ so that $\bt^\top \tilde{\bd}^s$ can be expressed as the the sum of independent terms. 
Denote $W_{ij} = Y_{ij} - \bbE(Y_{ij})$. 

For $r\in [k]$ and $i,j\in [n]$, we define 
$$
(X_{ij})_r = 
\begin{cases}
W_{rj}/\sigma_{n,r}, & i = r \\
W_{ir}/\sigma_{n,r}, & j= r \\
0,  & \text{otherwise}
\end{cases}
$$
Therefore, the normalized noisy degree can be written as
$$
\tilde \bd^s = (\tilde d_1^s,\ldots, \tilde d_k^s) = \sum_{i <j} \Xb_{ij}.
$$
Consequently, 
$$
 \bt^\top \tilde{\bd}^s = \sum_{i<j} \bt^\top \Xb_{ij} =:  \sum_{i<j} Z_{ij}. 
$$
By construction, the collection ${\Xb_{ij} : i<j}$ is mutually independent.

{\bf Step 3: Verification of the Lindeberg condition}

Each $Z_{ij}$ is centered. Since 
$$
|Z_{ij}|
= \left|\frac{t_i}{\sigma_{n,i}}W_{ij}\bone\{i\le k\}+\frac{t_j}{\sigma_{n,j}}W_{ij}\bone\{j\le k\}\right|
\le \frac{|t_i|}{\sigma_{n,i}}\bone\{i\le k\}+\frac{|t_j|}{\sigma_{n,j}}\bone\{j\le k\}
\le \frac{2\|\bt\|_2}{m_n}.
$$
and $m_n \rightarrow \infty$, it follows that
$$
\sup_{i<j} |Z_{ij}| \rightarrow 0,
$$
and hence the Lindeberg condition holds. 

{\bf Step 4: Identification of the asymptotic variance}

Since $\{Z_{ij}\}_{i<j}$ are independent,
\[
\Var\!\left(\sum_{i<j} Z_{ij}\right)
= \sum_{i<j} \Var(Z_{ij}).
\]
For a given pair $(i,j)$,
\[
Z_{ij}
= \frac{t_i}{\sigma_{n,i}} W_{ij}\,\bone\{i\le k\}
  + \frac{t_j}{\sigma_{n,j}} W_{ij}\,\bone\{j\le k\}.
\]
Thus, 
\[
\Var(Z_{ij})
= \Var(W_{ij})
\left(\frac{t_i}{\sigma_{n,i}}+\frac{t_j}{\sigma_{n,j}}\right)^2.
\]
Summing over all $(i,j)$ yields
\[
\sum_{i<j} \Var(Z_{ij})
= \sum_{r=1}^k \frac{t_r^2}{\sigma_{n,r}^2}\sum_{j\neq r}\Var(W_{rj})
+ 2\sum_{1\le r<s\le k}\frac{t_r t_s}{\sigma_{n,r}\sigma_{n,s}}\Var(W_{rs}).
\]
Since $\sigma_{n,r}^2=\sum_{j\neq r}\Var(W_{rj})$, so the first term equals
$\sum_{r=1}^k t_r^2$. Since $\Var(W_{rs})\le C$ and $m_n=\min_{r\le k}\sigma_{n,r}\to\infty$,
the second term is $O(m_n^{-2})$ and therefore vanishes. Hence,
\[
\Var\!\left(\sum_{i<j} Z_{ij}\right)\to \|\bt\|_2^2.
\]

{\bf Step 5: Application of Lindeberg--Feller CLT}

The triangular array $\{Z_{ij}\}$ is independent, centered, satisfies the Lindeberg condition, and has total variance converging to $\|\bt\|_2^2$. Hence, by the Lindeberg-Feller CLT, \eqref{eqn:1940} holds for every fixed $\bt$.

\noindent{\bf Proof of Conclusion (2).}
The argument follows the proof of Conclusion~(1) with $\bt$ replaced by
$\bt^{(n)}\in\mathbb R^{k_n}$. It suffices to show that for every deterministic
$\bt^{(n)}\in\mathbb R^{k_n}$ with $\|\bt^{(n)}\|_2\to\tau$,
\[
\langle \bt^{(n)}, \tilde\bd^s\rangle \ \convd \ \mathcal N(0,\tau^2),
\]
where $\tilde\bd^s = (\tilde d_1^s,\ldots,\tilde d_{k_n}^s)$.

As in Conclusion~(1), we have the decomposition
\[
(\bt^{(n)})^\top \tilde{\bd}^s
= \sum_{i<j} (\bt^{(n)})^\top \Xb_{ij}
=: \sum_{i<j} Z_{ij}^{(n)},
\]
which forms an independent, centered triangular array. Moreover,
\[
\max_{i<j} |Z_{ij}^{(n)}|
\le \frac{2\|\bt^{(n)}\|_2}{m_n} \to 0,
\]
so the Lindeberg condition holds.

For the variance, the argument of Step~4 in Conclusion~(1) yields
\[
\Var\!\left(\sum_{i<j} Z_{ij}^{(n)}\right)
= \|\bt^{(n)}\|_2^2 + O\!\left(\frac{k_n\|\bt^{(n)}\|_2^2}{m_n^2}\right).
\]
Since $\|\bt^{(n)}\|_2\to\tau$ and $k_n=o(m_n^2)$, the second term is $o(1)$, and
hence
\[
\Var\!\left(\sum_{i<j} Z_{ij}^{(n)}\right)\to \tau^2.
\]
The conclusion follows from the Lindeberg--Feller CLT.

{\bf Proof of Conclusion (3)}

We prove conclusion (3) using high-dimensional Gaussian approximation and Gaussian comparison. 

When $k$ is fixed, the same conclusion follows from a multivariate Berry--Esseen bound
(applied to $(\tilde d_1^s,\ldots,\tilde d_k^s)$), yielding an error of order $O(m_n^{-1})$.
This establishes Conclusion (3) when $k$ is fixed. In the remainder we assume $k\to\infty$.

We have defined in the previous proof that 
$$
\tilde \bd^s = (\tilde d_1^s,\ldots, \tilde d_k^s) = \sum_{i <j} \Xb_{ij},
$$
which is a sum of independent random vectors. Let
\[
\Sigma_n:=\operatorname{Cov}(\tilde\bd^s )=\sum_{1\le i<j\le n}\operatorname{Cov}(\Xb_{ij}). 
\]
Let $N:=\binom{n}{2}$ and define $\widetilde X_{uv}:=\sqrt N\,X_{uv}$, so that
\(
\tilde{\bd}^s=\sum_{u<v}X_{uv}=\frac{1}{\sqrt N}\sum_{u<v}\widetilde X_{uv}.
\)
Then $\{\widetilde X_{uv}\}_{u<v}$ are independent. 

To apply Proposition 2.1 in \cite{chernozhukov2017central}, we check the following conditions. 
\medskip\noindent
\textbf{Verification of (M.1).}
Fix $j\in[k]$. Since $\tilde d_j^s=\sum_{u<v}(X_{uv})_j$ and the summands are independent,
\begin{align*}
N^{-1}\sum_{u<v}\bbE\!\left[(\widetilde X_{uv,j})^2\mid A\right]
&=N^{-1}\sum_{u<v}\bbE\!\left[N(X_{uv,j})^2 \right]
=\sum_{u<v}\bbE\!\left[(X_{uv,j})^2 \right]\\
&=\Var\!\left(\sum_{u<v}(X_{uv})_j \right)
=\Var(\tilde d_j^s )=1.
\end{align*}
Hence (M.1) holds with $b=1$.

\medskip\noindent
\textbf{Verification of (M.2).}
Fix $r \in[k]$. Note that $(\widetilde X_{uv})_r \neq 0$ only if the edge $(u,v)$ is incident to
node $r$, i.e. $u=r$ or $v=r$. There are at most $n-1$ such edges. Moreover, since
$|W_{uv}|\le 1$ almost surely,
\[
|\widetilde X_{uv,r}|
=\sqrt N\,|X_{uv,r}|
=\sqrt N\,\frac{|W_{uv}|}{\sigma_{n,r}}
\le \frac{\sqrt N}{\sigma_{n,r}}
\le \frac{\sqrt N}{m_n}.
\]
Therefore, for $\delta=1$,
\begin{align*}
N^{-1}\sum_{u<v}\bbE\!\left[|\widetilde X_{uv,r}|^{2+\delta}\right]
&=N^{-1}\sum_{u<v:\, (u,v)\ni r}\bbE\!\left[|\widetilde X_{uv,r}|^{3}\right]
= \sqrt{N} \sum_{u<v:\, (u,v)\ni r}\bbE\!\left[| X_{uv,r}|^{3}\right] \\
& = \frac{\sqrt{N}}{\sigma_{n,r}^3} \sum_{j\neq r}
\bbE\left(  |Y_{rj} - \bbE Y_{rj}|^3 \right) \leq  \frac{\sqrt{N}}{\sigma_{n,r}^3} \sigma_{n,r}^2 \le \frac{\sqrt{N}}{\sigma_{n,r} } , 
\end{align*}
where the second last inequality uses that $W_{rj}:=Y_{rj}-\bbE Y_{rj}$ satisfies
$|W_{rj}|\le 1$ a.s., hence $|W_{rj}|^3\le |W_{rj}|^2$ a.s., and therefore
\[
\sum_{j\neq r}\bbE\!\left[|Y_{rj}-\bbE Y_{rj}|^{3}\right]
=\sum_{j\neq r}\bbE\!\left[|W_{rj}|^{3}\right]
\le \sum_{j\neq r}\bbE\!\left[W_{rj}^{2}\right]
=\sum_{j\neq r}\Var(Y_{rj})
=\sigma_{n,r}^{2}.
\]
For $\delta = 2$, 
\begin{align*}
N^{-1}\sum_{u<v}\bbE\!\left[|\widetilde X_{uv,r}|^{2+\delta}\right]
&=N^{-1}\sum_{u<v:\, (u,v)\ni r}\bbE\!\left[|\widetilde X_{uv,r}|^{4}\right]
= N \sum_{u<v:\, (u,v)\ni r}\bbE\!\left[| X_{uv,r}|^{4}\right] \\
& = \frac{N}{\sigma_{n,r}^4 } \sum_{j \neq r}
\bbE\left(  |Y_{rj} - \bbE Y_{rj}|^4 \right) \leq  \frac{N}{\sigma_{n,r}^4} \sigma_{n,r}^2 \le \frac{N}{\sigma_{n,r}^2 } , 
\end{align*}
where the second last inequality uses that
\[
\sum_{j\neq r}\bbE\!\left[|Y_{rj}-\bbE Y_{rj}|^{4}\right]
=\sum_{j\neq r}\bbE\!\left[|W_{rj}|^{4}\right]
\le \sum_{j\neq r}\bbE\!\left[W_{rj}^{2}\right]
= \sum_{j\neq r}\Var(Y_{rj})
=\sigma_{n,r}^{2}.
\]

Consequently, (M.2) holds with the choice
\[
B_N:=\frac{\sqrt N}{m_n}. 
\]

\medskip\noindent
\textbf{Verification of (E.1).}
With the above choice of $B_N$, we have $|\widetilde X_{uv,j}|/B_N\le 1$ for all $u<v$ and
$j\in[k]$. Hence
\[
\bbE\!\left[\exp\!\left(\frac{|\widetilde X_{uv,j}|}{B_N}\right)\right]
\le e
\le 2e^{0}
\quad\text{for all $u<v$ and $j\in[k]$}.
\]
By the Gaussian approximation result in Proposition~2.1 of \citet{chernozhukov2017central},
applied to the normalized sum
$\tilde\bd^s=\frac{1}{\sqrt N}\sum_{u<v}\widetilde X_{uv}$ of independent vectors, we have
\[
\sup_{B\in\mathcal R}
\Big|
\bbP(\tilde\bd^s\in B)-\bbP(\Gb\in B)
\Big|
\;\le\;
C\left(\frac{\log^7(kn)}{m_n^2}\right)^{1/6},
\quad
m_n:=\min_{1\le i\le k}\sigma_{n,i}, \quad \Gb \sim \cN(\bzero, \bSigma_n), 
\]
where $\mathcal R$ denotes the class of rectangles in $\bbR^k$ and $C>0$ is a universal constant.

The covariance matrix $\Sigma_n=\operatorname{Cov}(\tilde\bd^s)$ has the form
\[
(\Sigma_n)_{ii} = 1, \qquad
(\Sigma_n)_{ij}
= \operatorname{Cov}(\tilde d_i^s,\tilde d_j^s)
= \frac{\Var(W_{ij})}{\sigma_{n,i}\sigma_{n,j}}
\le \frac{C}{m_n^2}, \quad i\neq j,
\]
where $m_n:=\min_{1\le i\le k}\sigma_{n,i}\to\infty$.
Thus $\Sigma_n = I_k + R_n$ with $\|R_n\|_{\max}\le C/m_n^2$.

Let $\Gb\sim\mathcal N(0,\Sigma_n)$ and $\Zb\sim\mathcal N(0,I_k)$.
By the Gaussian comparison inequality for rectangles
(Lemma~A.1 in \citet{chernozhukov2017central}),
\[
\sup_{B\in\mathcal R}
\big|
\bbP(\Gb\in B)-\bbP(\Zb\in B)
\big|
\le
C\|\Sigma_n-I_k\|_{\max}^{1/3}\log^{2/3}k
\le
C\left(\frac{\log^2 k}{m_n^2}\right)^{1/3}.
\]

Combining this bound with the Gaussian approximation error from
Proposition~2.1 of \citet{chernozhukov2017central}, we obtain
\[
\sup_{B\in\mathcal R}
\big|
\bbP(\tilde\bd^s\in B)-\bbP(\Zb\in B)
\big|
\le
C\left(\frac{\log^7(kn)}{m_n^2}\right)^{1/6}.
\]

Taking $B=\{x\in\mathbb R^k:\min_{i\le k}x_i\ge a_i\}$ yields
\[
\bbP\!\left(\min_{i\le k}\tilde d_i < c\right)
=
\bbP\!\left(\exists i\in[k]: Z_i<a_i\right)
+
O\!\left(\left(\frac{\log^7(kn)}{m_n^2}\right)^{1/6}\right),
\]
where $a_i=(c-\mu_{n,i})/\sigma_{n,i}$.

This completes the proof.

\end{proof}

\begin{lem}[Refined Berry--Esseen bound for noisy degrees]\label{lem:refined_berry_esseen}
Let $\cG = (V, E)$ be the graph with adjacency matrix $A$, and let $Y$ be a noisy version of $A$ generated according to the noise model in~\eqref{eqn:noise_model}. 
Fix $i \in [n]$, and define the (true) degree
$d_i := \sum_{j \ne i} A_{ij}$ and the noisy degree $\tilde d_i := \sum_{j \ne i} Y_{ij}$.

Then we have 
$$
\sup_{x \in \mathbb{R}} \left| \bbP\left( \frac{\tilde{d}_i - \mu_{n,i}}{\sigma_{n,i}} \leq x \right) - \Phi(x) \right|
\leq \frac{C}{\sigma_{n,i}}, 
$$
where $C > 0$ is a universal constant. 
\end{lem}

\begin{proof}
Conditional on $A$, the variables $\{Y_{ij}\}_{j\ne i}$ are independent.

Applying the Berry--Esseen theorem for independent bounded random variables
(Lemma~\ref{lem:berry_esseen}), we obtain
\[
\sup_{x\in\mathbb R}
\left|
\bbP\!\left( \frac{\tilde d_i - \mu_{n,i}}{\sigma_{n,i}} \le x \right)
- \Phi(x)
\right|
\le \frac{C}{\sigma_{n,i}^3}
\sum_{j\ne i} \bbE\left|Y_{ij}-\bbE(Y_{ij})\right|^3 .
\]

We evaluate the third absolute moments case by case.

\textbf{Case 1:} If $A_{ij}=1$, then $Y_{ij}\sim\mathrm{Bern}(1-\beta_n)$ and
\[
\bbE|Y_{ij}-\bbE(Y_{ij})|^3
= \beta_n(1-\beta_n)(2\beta_n^2-2\beta_n+1).
\]

\textbf{Case 2:} If $A_{ij}=0$, then $Y_{ij}\sim\mathrm{Bern}(\alpha_n)$ and
\[
\bbE|Y_{ij}-\bbE(Y_{ij})|^3
= \alpha_n(1-\alpha_n)(2\alpha_n^2-2\alpha_n+1).
\]

Therefore,
\[
\sum_{j\ne i}\bbE|Y_{ij}-\bbE(Y_{ij})|^3
=
(n-1-d_i)\alpha_n(1-\alpha_n)(2\alpha_n^2-2\alpha_n+1)
+ d_i\beta_n(1-\beta_n)(2\beta_n^2-2\beta_n+1).
\]
Using the bound $2p^2-2p+1 \le 1$ for all $p\in[0,1]$, we obtain
\[
\sum_{j\ne i}\bbE|Y_{ij}-\bbE(Y_{ij})|^3
\le (n-1-d_i)\alpha_n(1-\alpha_n)+d_i\beta_n(1-\beta_n)
= \sigma_{n,i}^2.
\]
Substituting this into the Berry--Esseen bound yields
\[
\sup_{x\in\mathbb R}
\left|
\bbP\!\left( \frac{\tilde d_i - \mu_{n,i}}{\sigma_{n,i}} \le x \right)
- \Phi(x)
\right|
\le \frac{C}{\sigma_{n,i}}.
\]
\end{proof}

\begin{lem}[Berry--Esseen theorem via third moments]\label{lem:berry_esseen}
Let $X_1,\dots,X_n$ be independent real-valued random variables with finite third moments, and let
\[
S_n:=\sum_{i=1}^n X_i,\qquad \mu_S:=\bbE[S_n],\qquad \sigma_S^2:=\Var(S_n)>0.
\]
Then
\[
\sup_{x\in\bbR}\left|\bbP\!\left(\frac{S_n-\mu_S}{\sigma_S}\le x\right)-\Phi(x)\right|
\le \frac{C}{\sigma_S^3}\sum_{i=1}^n \bbE\!\left|X_i-\bbE X_i\right|^3,
\]
where $\Phi$ is the standard normal cdf and $C>0$ is a universal constant.
\end{lem}

\noindent
This is the classical Berry--Esseen theorem under the third-moment condition, so we omit the proof. We now proceed to establish two auxiliary lemmas that will be used in the proofs of the main theorems.

\begin{lem}[Extreme-value envelope for the tail]\label{lem:quantile_bound}
Let $k=o(n)$ and set $m:=n-k$.
Assume the latent degrees are ordered as $d_1\ge d_2\ge\cdots\ge d_n$
(as assumed throughout the paper), and that $\sigma_{n,k+1}\to\infty$.
Define
\[
\epsilon_1(m):=\frac{\log\log m}{2\sqrt{2\log m}},
\qquad
\epsilon_2(n):=\frac{C(n)}{\sqrt{\log n}},
\]
where $C(n)\to\infty$ arbitrarily slowly and $C(n)=o\big((\log n)^{1/2}\big)$.

\smallskip
\noindent
\textbf{(Upper envelope).}
Conditional on the latent network $\Ab$, with probability $1-o(1)$,
\begin{equation}\label{eqn:quantile_upper}
\max_{i\ge k+1}\tilde d_i
\le
\mu_{n,k+1}+\Big(\sqrt{2\log m}-\epsilon_1(m)+\epsilon_2(n)\Big)\sigma_{n,k+1}
=:c_{n,k+1}^{u}.
\end{equation}

\smallskip
\noindent
\textbf{(Worst-case lower envelope).}
Moreover, under the homogeneous-tail configuration
$d_{k+1}=d_{k+2}=\cdots=d_n$, the lower bound
\begin{equation}\label{eqn:quantile_lower}
\max_{i\ge k+1}\tilde d_i
\ge
\mu_{n,k+1}+\Big(\sqrt{2\log m}-\epsilon_1(m)-\epsilon_2(n)\Big)\sigma_{n,k+1}
=:c_{n,k+1}^{\ell}
\end{equation}
holds with probability $1-o(1)$ and is asymptotically sharp in that case. 
\end{lem}

\begin{proof}
The argument follows the classical order-statistic approach of
\cite{bollobas1981degree}, but must be substantially modified to handle
(i) noisy degree observations,
(ii) sparse and heterogeneous parameters, and
(iii) the dependence induced by symmetric edge observations.
The proof proceeds in three steps:
(1) we control the mean and variance of the exceedance count $X_h$,
(2) we apply concentration inequalities to invert $X_h$ into quantile bounds,
and (3) we extend the argument from the homogeneous-tail case to general
degree sequences using stochastic dominance.

We index the nodes in nonincreasing order of their latent degrees and first
consider the homogeneous-tail case
\[
d_1 = d_2 = \cdots = d_k > d_{k+1} = d_{k+2} = \cdots = d_n.
\]
The general case will be addressed at the end of the proof.
Throughout, we work conditional on the latent network $\Ab$.

For $h\in\{0,\ldots,n-k-1\}$, define
\[
X_h := \sum_{i=k+1}^n \mathbf 1\{\tilde d_i \ge h\},
\]
the number of tail vertices whose noisy degree exceeds $h$.
Let
\[
\tilde{d}_{(1)}^{[>k]} \ge \tilde{d}_{(2)}^{[>k]} \ge \cdots \ge \tilde{d}_{(n-k)}^{[>k]}
\]
denote the ordered noisy degrees of vertices $k+1,\ldots,n$. By definition,
\[
\tilde{d}_{(m)}^{[>k]} \ge h \quad \text{if and only if} \quad X_h \ge m.
\]

To derive bounds on degree quantiles, we relate order statistics to tail
probabilities of the exceedance counts $X_h$. The following lemma, adapted
from Lemma~1 of \cite{bollobas1981degree}, provides this link.
\begin{lem}\label{lem:quantile_tail}
Let $1 \le h' < h'' \le n-k-1$. Put
$\lambda'=\bbE(X_{h'})$, $\sigma'^2=\Var(X_{h'})$, and define
$\lambda''$ and $\sigma''^2$ analogously.
If $m$ is an integer satisfying $\lambda'' < m \le \lambda'$, then
\[
\bbP\!\left( \tilde d_{(m)}^{[>k]} \ge h'' \right)
\le
\min\!\left\{\frac{\lambda''}{m},
\frac{\sigma''^2}{(m-\lambda'')^2}\right\},
\]
and
\[
\bbP\!\left( \tilde d_{(m)}^{[>k]} < h' \right)
\le
\frac{\sigma'^2}{(\lambda'-m+1)^2}.
\]
\end{lem}
The proof follows directly from applications of Markov’s and Chebyshev’s inequalities. We next provide explicit upper bounds on $\bbE\left( X_h\right)$ and $ \operatorname{Var}\left(X_h\right)$ for the values of $h$ relevant to our analysis.

From now on, we restrict attention to thresholds of the form
\[
h = \mu_{n,k+1} + x_n \sigma_{n,k+1},
\qquad
x_n \to \infty,
\quad
x_n = o(\sigma_{n,k+1}^{1/3}).
\]
Therefore, by Lemma~\ref{lem:binom_approx} in this moderate-deviation regime,
\begin{equation}\label{eqn:mu_asymp}
\lambda_h := \bbE(X_h)
= (n-k)\,\bbP(\tilde d_{k+1}\ge h)
\sim
\frac{1}{\sqrt{2\pi}}\frac{n-k}{x_n}
\exp\!\left(-\frac{x_n^2}{2}\right),
\end{equation}
where $\sim$ denotes asymptotic equivalence as $n\to\infty$.

Let $T:=\{k+1,\ldots,n\}$. To bound the variance
$\sigma_h^2:=\Var(X_h)$, let $W_h$ denote the number of unordered pairs
$\{i,j\}\subseteq T$ such that both noisy degrees exceed $h$, i.e.
\[
W_h := \sum_{\substack{i,j\in T\\ i<j}} \mathbf 1\{\tilde d_i\ge h,\ \tilde d_j\ge h\}.
\]
Using $m^2=m+2\binom{m}{2}$, we have
\[
\bbE(X_h^2)=\bbE(X_h)+2\bbE\!\left(\binom{X_h}{2}\right)
=\bbE(X_h)+2\bbE(W_h).
\]
To compute $\bbE(W_h)$, define
\begin{align*}
V   &\sim \mathrm{Bin}(n - 1 - d_{k+1}, \alpha_n) + \mathrm{Bin}(d_{k+1}, 1 - \beta_n), \\
V'  &\sim \mathrm{Bin}(n - 1 - d_{k+1}, \alpha_n) + \mathrm{Bin}(d_{k+1} - 1, 1 - \beta_n), \\
V'' &\sim \mathrm{Bin}(n - 2 - d_{k+1}, \alpha_n) + \mathrm{Bin}(d_{k+1}, 1 - \beta_n),
\end{align*}
and set $r_h:=\bbP(V\ge h)$, $p_t:=\bbP(V'\ge t)$, $q_t:=\bbP(V''\ge t)$.
Here $V$ is the noisy degree of a typical tail vertex; $V'$ (resp.\ $V''$) is
the noisy degree excluding the shared edge $(i,j)$, conditional on $A_{ij}=1$
(resp.\ $A_{ij}=0$).

Summing over unordered pairs yields
\[
\bbE(W_h)
=\sum_{\substack{i,j\in T\\ i<j: A_{ij}=1}}\big[(1-\beta_n)p_{h-1}^2+\beta_n p_h^2\big]
+\sum_{\substack{i,j\in T\\ i<j: A_{ij}=0}}\big[\alpha_n q_{h-1}^2+(1-\alpha_n) q_h^2\big].
\]
The key idea is to condition on $A_{ij}$: conditional on $\Ab$ and $A_{ij}$,
the only shared random contribution to $\tilde d_i$ and $\tilde d_j$ is the
symmetric observation on $(i,j)$, and the remaining contributions are independent.
Denote $\rho:=\sum_{\substack{i,j\in T\\ i<j}} A_{ij}$, the number of latent edges inside $T$.

Since the shared edge contributes $1$ with probability $1-\beta_n$ when $A_{ij}=1$
(and with probability $\alpha_n$ when $A_{ij}=0$), we have
\[
r_h=(1-\beta_n)p_{h-1}+\beta_n p_h
=\alpha_n q_{h-1}+(1-\alpha_n)q_h.
\]
Since $\lambda_h=(n-k)r_h$, we obtain
\begin{align*}
\sigma_h^2
&=\Var(X_h)
=\bbE(X_h^2)-\lambda_h^2
=\lambda_h+2\bbE(W_h)-\lambda_h^2\\
&=\lambda_h
+2\rho\big[(1-\beta_n)p_{h-1}^2+\beta_n p_h^2\big]
+2\Big(\tbinom{n-k}{2}-\rho\Big)\big[\alpha_n q_{h-1}^2+(1-\alpha_n)q_h^2\big]
-\big[(n-k)r_h\big]^2\\
&\le \lambda_h
+2\rho\big[(1-\beta_n)p_{h-1}^2+\beta_n p_h^2\big]
+\big[(n-k)^2-2\rho\big]\big[\alpha_n q_{h-1}^2+(1-\alpha_n)q_h^2\big]
-\big[(n-k)r_h\big]^2\\
&=\lambda_h
+2\rho\big((1-\beta_n)p_{h-1}^2+\beta_n p_h^2-r_h^2\big)
+\big[(n-k)^2-2\rho\big]\big(\alpha_n q_{h-1}^2+(1-\alpha_n)q_h^2-r_h^2\big)\\
&=\lambda_h
+2\rho\,\beta_n(1-\beta_n)(p_{h-1}-p_h)^2
+\big[(n-k)^2-2\rho\big]\alpha_n(1-\alpha_n)(q_{h-1}-q_h)^2\\
&=\lambda_h
+2\rho\,\beta_n(1-\beta_n)\,\bbP(V'=h-1)^2
+\big[(n-k)^2-2\rho\big]\alpha_n(1-\alpha_n)\,\bbP(V''=h-1)^2.
\end{align*}

Note that $\mu_{V'}=\mu_{n,k+1}+O(1)$ and
$\sigma_{V'}=\sigma_{n,k+1}(1+o(1))$ (and similarly for $V''$). Therefore, 
\[
\frac{(h-1)-\mu_{V'}}{\sigma_{V'}}=x_n+o(1),
\qquad
\frac{(h-1)-\mu_{V''}}{\sigma_{V''}}=x_n+o(1),
\]
and the conditions of Lemma~\ref{lem:binom_approx} apply. Therefore,
\[
\bbP(V'=h-1)=\frac{1+o(1)}{\sqrt{2\pi}\,\sigma_{n,k+1}}e^{-x_n^2/2},
\qquad
\bbP(V''=h-1)=\frac{1+o(1)}{\sqrt{2\pi}\,\sigma_{n,k+1}}e^{-x_n^2/2},
\]
so that
\[
\bbP(V'=h-1)^2=\bbP(V''=h-1)^2
=\frac{1+o(1)}{2\pi\,\sigma_{n,k+1}^2}e^{-x_n^2}.
\]
Thus,
\[
\sigma_h^2
\le \lambda_h
+\frac{e^{-x_n^2}}{2\pi\,\sigma_{n,k+1}^2}(1+o(1))
\Big\{2\rho\,\beta_n(1-\beta_n)
+\big[(n-k)^2-2\rho\big]\alpha_n(1-\alpha_n)\Big\}.
\]
The bracketed term can be rewritten as $\sigma_{n,k+1}^2 R$, where
\[
R
:=
\frac{2\rho\,b+\big[(n-k)^2-2\rho\big]a}{\sigma_{n,k+1}^2},
\qquad
a=\alpha_n(1-\alpha_n),\ \ b=\beta_n(1-\beta_n).
\]

Let $a=\alpha_n(1-\alpha_n)$, $b=\beta_n(1-\beta_n)$, and $\Delta=b-a$.
Since $\sigma_{n,k+1}^2=(n-1-d)a+db$, we may write
\[
R
:=
\frac{2\rho\,b+\big[(n-k)^2-2\rho\big]a}{\sigma_{n,k+1}^2}
=
\frac{2\rho\,\Delta+(n-k)^2a}{d\,\Delta+(n-1)a}.
\]

Each tail vertex has degree $d$ and at most $k$ neighbors in
$\{1,\ldots,k\}$, hence at least $\max\{0,d-k\}$ neighbors inside $T$.
Therefore,
\[
\frac{(n-k)\max\{0,d-k\}}{2}\le \rho \le \frac{(n-k)\min\{d,n-k-1\}}{2}.
\]
If $\Delta\ge0$, then $R$ is increasing in $\rho$, and the upper bound yields
\[
R\le \frac{(n-k)d\Delta+(n-k)(n-1)a}{d\Delta+(n-1)a}=n-k.
\]
If $\Delta<0$, then $R$ is decreasing in $\rho$, and the lower bound yields
\[
R\le \frac{(n-k)d\Delta+(n-k)(n-1)a}{d\Delta+(n-1)a}=n-k.
\]
Therefore, $R\le n-k$ for all $\Delta\in\mathbb R$ and using it yields
\[
\sigma_h^2
\le \lambda_h + \frac{n-k}{2\pi}e^{-x_n^2}(1+o(1)).
\]
It remains to compare the second term with $\lambda_h$.
By \eqref{eqn:mu_asymp},
\[
\lambda_h
\sim
\frac{n-k}{\sqrt{2\pi}}\frac{1}{x_n}e^{-x_n^2/2},
\]
and hence
\[
\frac{(n-k)e^{-x_n^2}}{\lambda_h}
\sim
\sqrt{2\pi}\,x_n\,e^{-x_n^2/2}
\;\longrightarrow\;0,
\]
since $x_n\to\infty$. Therefore,
\[
\frac{n-k}{2\pi}e^{-x_n^2}(1+o(1)) = o(\lambda_h),
\]
and we conclude that
\begin{equation}\label{eqn:sigma_bound}
\sigma_h^2 \le (1+o(1))\,\lambda_h \le c\,\lambda_h,
\end{equation}
for some constant $c>0$ and all sufficiently large $n$.

We now invert the exceedance bounds to obtain sharp asymptotics for the upper order statistics. Choose
\begin{align*}
h_f
&:= \mu_{n,k+1} + x_f \sigma_{n,k+1} \\
&= \mu_{n,k+1}
+ \bigl(2\sigma_{n,k+1}^2\log(n-k)\bigr)^{1/2}
- \Bigl(\frac{\sigma_{n,k+1}^2}{2\log(n-k)}\Bigr)^{1/2}
\Bigl\{\tfrac12\log\!\bigl(2\log(n-k)\bigr)+f+\log\bigl((2\pi)^{1/2}\bigr)\Bigr\},
\end{align*}
where $f=f(n)=o\!\bigl((\log n)^{1/2}\bigr)$ and
$x_f := (h_f-\mu_{n,k+1})/\sigma_{n,k+1}$.

By \eqref{eqn:mu_asymp},
\[
\lambda_{h_f}
:= \bbE(X_{h_f})
\sim \frac{n-k}{\sqrt{2\pi}}\,\frac{1}{x_f}e^{-x_f^2/2}
= e^{f}\,(1+o(1)).
\]

We first derive an upper bound on $\tilde d^{[>k]}_{(m)}$ for fixed
$m\in\{1,\ldots,n-k\}$.
Let $f_1=-C(n)\to-\infty$ arbitrarily slowly, so that
$\lambda_{h_{f_1}}=e^{f_1}(1+o(1))\to0$.
Apply Lemma~\ref{lem:quantile_tail} with $h'=0$ and $h''=h_{f_1}$.
Since $\lambda'=\bbE(X_0)=n-k$ and $\lambda''=\lambda_{h_{f_1}}<m$
for all sufficiently large $n$, the lemma yields
\begin{align*}
\bbP\!\left(\tilde d^{[>k]}_{(m)} \ge h_{f_1}\right)
&\le \min\!\left\{\frac{\lambda_{h_{f_1}}}{m},
\frac{\sigma^2_{h_{f_1}}}{(m-\lambda_{h_{f_1}})^2}\right\} \\
&\le c \min\!\left\{\frac{\lambda_{h_{f_1}}}{m},
\frac{\lambda_{h_{f_1}}}{(m-\lambda_{h_{f_1}})^2}\right\}
= o(1),
\end{align*}
where the second inequality uses \eqref{eqn:sigma_bound}.
Thus, with probability $1-o(1)$,
\[
\tilde d^{[>k]}_{(m)} < h_{f_1}.
\]
For the lower bound, choose $f_2=C(n)\to\infty$ arbitrarily slowly.
Then $\lambda_{h_{f_2}}=e^{f_2}(1+o(1))\to\infty$.
Applying Lemma~\ref{lem:quantile_tail} with $h'=h_{f_2}$, we obtain
\[
\bbP\!\left(\tilde d^{[>k]}_{(m)} < h_{f_2}\right)
\le
\frac{\sigma_{h_{f_2}}^2}{(\lambda_{h_{f_2}}-m+1)^2}
= o(1),
\]
where the last step follows from $\sigma_{h}^2 \le c\lambda_h$ and
$\lambda_{h_{f_2}}\to\infty$.
Hence,
\[
\tilde d^{[>k]}_{(m)} \ge h_{f_2}
\quad\text{with probability }1-o(1).
\]
Combining the upper and lower bounds, we obtain
\begin{align*}
\Bigl|
\tilde d^{[>k]}_{(m)}
&- \mu_{n,k+1}
- \bigl(2\sigma_{n,k+1}^2\log(n-k)\bigr)^{1/2}
+ \Bigl(\frac{\sigma_{n,k+1}^2}{8\log(n-k)}\Bigr)^{1/2}
\log\!\bigl(2\log(n-k)\bigr)
+ \log\bigl((2\pi)^{1/2}\bigr)
\Bigr| \\
&\le 2C(n)\Bigl(\frac{\sigma_{n,k+1}^2}{2\log(n-k)}\Bigr)^{1/2},
\end{align*}
with probability $1-o(1)$, for any $C(n)=o\!\bigl((\log n)^{1/2}\bigr)$.

\paragraph{Extension to general tail degree sequences.}
We now extend the upper-envelope bound to the general ordering
\[
d_1 \ge \cdots \ge d_k > d_{k+1} \ge \cdots \ge d_n,
\qquad
T := \{k+1,\ldots,n\}. 
\]

By Lemma~\ref{lem:stochastic_dominance}, for every $i\in T$ and all thresholds $h$,
\[
\bbP(\tilde d_i \ge h)
\le
\bbP(\tilde d_{k+1} \ge h).
\]
Hence the exceedance mean satisfies
\begin{equation}\label{eq:lambda_star_gen}
\lambda_h=\bbE(X_h)\le (n-k)\,\bbP(\tilde d_{k+1}\ge h)=:\lambda_h^{\star}.
\end{equation}

Next, conditioning on $A_{ij}$ and applying the same difference-of-squares
argument as in the homogeneous-tail case gives
\begin{align}\label{eq:var_general_gen}
\sigma_h^2
&\le \lambda_h
+2\beta_n(1-\beta_n)
\sum_{\substack{i<j\\ i,j\in T\\ A_{ij}=1}}
\bbP\!\big(V'(d_i)=h-1\big)\,
\bbP\!\big(V'(d_j)=h-1\big)\notag\\
&\quad
+2\alpha_n(1-\alpha_n)
\sum_{\substack{i<j\\ i,j\in T\\ A_{ij}=0}}
\bbP\!\big(V''(d_i)=h-1\big)\,
\bbP\!\big(V''(d_j)=h-1\big),
\end{align}
with $V'(d)$ and $V''(d)$ as defined earlier.

By the point-mass monotonicity in Lemma~\ref{lem:stochastic_dominance}, for all
$i\in T$ and all $h\ge \mu_{n,k+1}+1$,
\[
\bbP\!\big(V'(d_i)=h-1\big)\le \bbP\!\big(V'(d_{k+1})=h-1\big),
\qquad
\bbP\!\big(V''(d_i)=h-1\big)\le \bbP\!\big(V''(d_{k+1})=h-1\big).
\]
Substituting these bounds into \eqref{eq:var_general_gen} reduces the right-hand
side to the homogeneous-tail variance bound with $d$ replaced by $d_{k+1}$.
Therefore, by the estimate proved above in the homogeneous-tail case, there exists
$c>0$ such that
\begin{equation}\label{eq:var_star_gen}
\sigma_h^2 \le c\,\lambda_h^{\star}
\end{equation}
for all sufficiently large $n$ and all thresholds in the same regime.

Finally, applying Lemma~\ref{lem:quantile_tail} using
\eqref{eq:lambda_star_gen} and \eqref{eq:var_star_gen} yields the same
upper-envelope bound as before:
\[
\max_{i\ge k+1}\tilde d_i
\le
\mu_{n,k+1}
+\Big(\sqrt{2\log(n-k)}-\epsilon_1(n-k)+\epsilon_2(n)\Big)\sigma_{n,k+1}
\]
with probability $1-o(1)$.

\smallskip
\noindent
\emph{Lower envelope.}
The lower-envelope bound is achieved (and hence asymptotically sharp) under the
homogeneous-tail configuration $d_{k+1}=\cdots=d_n$, completing the proof.
\end{proof}

\begin{lem}[Monotone stochasticity in $d$]\label{lem:stochastic_dominance}
Let $n \in \mathbb N$ and let $\alpha_n,\beta_n \in [0,1]$ satisfy $1-\beta_n \ge \alpha_n$.
For each $d \in \{0,1,\dots,n-1\}$, let
\[
X(d) \;\sim\; \mathrm{Bin}(n-1-d,\alpha_n) + \mathrm{Bin}(d,1-\beta_n),
\]
where the two binomial variables are independent. Then for any $d_1 \ge d_2$, $X(d_1)$ stochastically dominates $X(d_2)$ in the usual stochastic order; that is, for all $h \in \mathbb R$,
\[
\bbP\!\left( X(d_1) \ge h \right) \;\ge\; \bbP\!\left( X(d_2) \ge h \right).
\]

Moreover, let $h^\star\in\{0,1,\dots,n-1\}$ satisfy
\[
h^\star \ge \left\lceil (n-d_1-1)\alpha_n + (d_1-1)(1-\beta_n) \right\rceil + 1.
\]
Then $\bbP( X(d_1)=h^\star ) \ge \bbP( X(d_2)=h^\star )$.
\end{lem}

\begin{proof}

\medskip
\noindent\textbf{Proof of the first claim.}
We argue by coupling.

Construct all random variables on a common probability space via i.i.d.\ uniform random variables. Let $U_1,\dots,U_n \stackrel{\text{i.i.d.}}{\sim} \mathrm{Unif}(0,1)$. For each $d\in\{0,1,\dots,n-1\}$, define
\[
X(d)\;=\;\sum_{i=1}^{n-1-d} \mathbf{1}\{U_i\le \alpha_n\}\;+\;\sum_{i=n-d}^{n-1} \mathbf{1}\{U_i\le 1-\beta_n\}.
\]
By construction, $X(d) \stackrel{d}{=} \mathrm{Bin}(n-1-d,\alpha_n)
+ \mathrm{Bin}(d,1-\beta_n)$, where the two binomial components are independent. 

Fix $d_1 \ge d_2$ and compare $X(d_1)$ and $X(d_2)$ term by term:
\begin{itemize}
\item For $i \le n-1-d_1$, both sums use the threshold $\alpha_n$, so the summands coincide.
\item For $n-d_1 \le i \le n-1-d_2$, $X(d_1)$ uses the threshold $1-\beta_n$
      while $X(d_2)$ uses $\alpha_n$.
\item For $i \ge n-d_2$, both sums use the threshold $1-\beta_n$, so the summands again coincide.
\end{itemize}
Since $1-\beta_n \ge \alpha_n$, we have for all $u \in [0,1]$,
\[
\mathbf{1}\{u \le 1-\beta_n\} \ge \mathbf{1}\{u \le \alpha_n\}.
\]
Applying this inequality with $u = U_i$ in each summand yields
\[
X(d_1) \ge X(d_2) \qquad \text{almost surely}.
\]
Therefore, for every $t \in \mathbb R$,
\[
\mathbb P\!\left( X(d_1) \ge t \right)
\ge
\mathbb P\!\left( X(d_2) \ge t \right),
\]
which shows that $X(d_1)$ stochastically dominates $X(d_2)$ in the usual stochastic order.

\medskip
\noindent\textbf{Proof of the second claim.}

Write $p_d(h) = \bbP(X(d)=h)$. 
Let $R_k$ be the sum of the $n-2$ Bernoulli variables that are common to $X(k)$ and $X(k+1)$, i.e.,
\[
R_k \sim \mathrm{Bin}(n-k-2,\alpha_n) + \mathrm{Bin}(k,1-\beta_n),
\]
and let $r_h^{(k)}=\bbP(R_k=h)$ (with the convention $r^{(k)}_{-1}=0$).
Let $B_{\alpha_n}\sim \mathrm{Bern}(\alpha_n)$ and $B_{\beta_n}\sim \mathrm{Bern}(1-\beta_n)$ be independent of $R_k$.
Then
\[
X(k)=R_k + B_{\alpha_n},\qquad X(k+1)=R_k + B_{\beta_n}.
\]
Consequently, for any $h$,
\[
p_k(h)= (1-\alpha_n) r_h^{(k)} + \alpha_n r_{h-1}^{(k)},\qquad
p_{k+1}(h)= \beta_n r_h^{(k)} + (1-\beta_n) r_{h-1}^{(k)}.
\]
Subtracting yields
\[
p_{k+1}(h)-p_k(h)
=\big((1-\beta_n)-\alpha_n\big)\big(r_{h-1}^{(k)}-r_h^{(k)}\big).
\]
Let $m_k$ be a mode of $R_k$. Since the pmf of $R_k$ is unimodal, $r_h^{(k)}$ is nonincreasing for $h\ge m_k$, so $r_{h-1}^{(k)} \ge r_h^{(k)}$ for all $h\ge m_k+1$.
Since $(1-\beta_n)-\alpha_n \ge 0$, it follows that
\[
p_{k+1}(h) - p_k(h)
=\big((1-\beta_n)-\alpha_n\big)\big(r_{h-1}^{(k)}-r_h^{(k)}\big)\ge 0,
\quad \forall\, h \ge m_k+1,
\]
i.e.,
\[
p_{k+1}(h) \ge p_k(h) \quad \text{for all } h \ge m_k+1.
\]
The mean of $R_k$ is
\[
\mu_{R,k} := \bbE R_k = (n-k-2)\alpha_n + k(1-\beta_n).
\]
By Darroch's bound, any mode $m_k$ of $R_k$ satisfies
\[
m_k \le \lceil \mu_{R,k} \rceil.
\]
Since $\mu_{R,k}$ is nondecreasing in $k$, we have
\[
\max_{d_2 \le k \le d_1-1} m_k
\le \max_{d_2 \le k \le d_1-1} \lceil \mu_{R,k} \rceil
= \lceil \mu_{R,d_1-1} \rceil.
\]
Hence, if $h^\star \ge \lceil \mu_{R,d_1-1} \rceil + 1$, then
$h^\star \ge m_k + 1$ for all $k=d_2,\dots,d_1-1$.

Therefore, applying $p_{k+1}(h^\star) \ge p_k(h^\star)$ stepwise for
$k=d_2,\dots,d_1-1$ yields
\[
p_{d_1}(h^\star) \ge p_{d_2}(h^\star),
\]
which proves the result.

\end{proof}

\begin{lem}[Local CLT and moderate deviations for a two--binomial sum]
\label{lem:binom_approx}
Let
\[
X \sim \mathrm{Bin}(n-d,\alpha_n) + \mathrm{Bin}(d,1-\beta_n),
\]
where the two binomial variables are independent. Define
\[
\mu_n := \bbE(X) = (n-d)\alpha_n + d(1-\beta_n),
\qquad
\sigma_n^2 := \Var(X)
= (n-d)\alpha_n(1-\alpha_n) + d(1-\beta_n)\beta_n,
\]
and assume that $\sigma_n^2 \to \infty$ as $n\to\infty$.

Let $h=h_n\in\mathbb Z$ and define
\[
x_n := \frac{h-\mu_n}{\sigma_n}.
\]
Assume that $|x_n| = o(\sigma_n^{1/3})$ and $x_n \rightarrow \infty$ as $n\rightarrow \infty$.

Then
\[
\bbP(X=h)
=
\frac{1}{\sqrt{2\pi}\,\sigma_n}
\exp\!\left(-\frac{x_n^2}{2}\right)\,(1+o(1)),
\]
and
\[
\bbP(X\ge h)
\sim
\frac{1}{\sqrt{2\pi}}\frac{1}{x_n}\exp\!\left(-\frac{x_n^2}{2}\right),
\]
where $\Phi$ denotes the standard normal distribution function and $\sim$ means that the percentage difference between the two sides tends to 0 as $n \rightarrow \infty$.
\end{lem}
% new and simplified proof 
\begin{proof}
We treat the local probability and the upper tail separately.

Let $X_1\sim\mathrm{Bin}(n-d,\alpha_n)$ and $X_2\sim\mathrm{Bin}(d,1-\beta_n)$ be independent.
Write $\mu_i=\bbE X_i$, $\sigma_i^2=\Var(X_i)$ for $i=1,2$, so that
$\mu_n=\mu_1+\mu_2$ and $\sigma_n^2=\sigma_1^2+\sigma_2^2$.
\paragraph{Local probability.}
We distinguish two regimes depending on the relative sizes of
$\sigma_1^2$ and $\sigma_2^2$.

\medskip
\noindent\emph{Balanced case.}
Assume that $\sigma_1/\sigma_2$ is bounded away from $0$ and $\infty$.
Let $h$ satisfy $|h-\mu_n|=o(\sigma_n^{4/3})$.

Let $x_n := \frac{h-\mu_n}{\sigma_n}$ and define a truncation window
\[
W_n := \Big\{s\in\mathbb Z:\ 
\max(0,\lceil\mu_1-2x_n\sigma_n\rceil)
\le s \le
\min(n-d,\lfloor\mu_1+2x_n\sigma_n\rfloor)
\Big\},
\]
where 
\(
x_n = o(\sigma_1^{1/3}) = o(\sigma_n^{1/3})\) and $x_n\to\infty$ by assumption. 
Decompose
\[
\mathbb P(X=h)
=
\sum_{s\in W_n}\mathbb P(X_1=s)\mathbb P(X_2=h-s)
+
\sum_{s\notin W_n}\mathbb P(X_1=s)\mathbb P(X_2=h-s).
\]
For notation simplicity, denote  
\[
p_1(s) = \bbP(X_1 = s), \qquad p_2(h-s) = \bbP(X_2= h-s).
\]

\smallskip
\noindent
\textbf{Main contribution.}
For $s\in W_n$ we have $|s-\mu_1|\le 2x_n\sigma_n=o(\sigma_1^{4/3})$. Moreover,
since $\sigma_1\asymp\sigma_2\asymp\sigma_n$ in the balanced case,
\[
|(h-s)-\mu_2|
\le |h-\mu_n|+|s-\mu_1|
=
o(\sigma_n^{4/3})+o(\sigma_1^{4/3})
=
o(\sigma_2^{4/3}).
\]
Since we are in the balanced case, we have
\[
\sigma_1^2 \asymp \sigma_2^2 \asymp \sigma_n^2 \to \infty.
\]
Moreover, for all $s\in W_n$,
\[
|s-\mu_1| \le 2x_n\sigma_n = o(\sigma_1^{4/3}),
\qquad
|(h-s)-\mu_2|
\le |h-\mu_n|+|s-\mu_1|
= o(\sigma_2^{4/3}).
\]
Therefore the conditions of Lemma~\ref{lem:lclt_bin} are satisfied for both
$X_1\sim\mathrm{Bin}(n-d,\alpha_n)$ and
$X_2\sim\mathrm{Bin}(d,1-\beta_n)$, and the local CLT applies uniformly over
$s\in W_n$. Consequently,
\[
p_1(s)
=
\frac{1}{\sigma_1}\varphi\!\left(\frac{s-\mu_1}{\sigma_1}\right)\{1+r_{1,n}(s)\},
\qquad
p_2(h-s)
=
\frac{1}{\sigma_2}\varphi\!\left(\frac{h-s-\mu_2}{\sigma_2}\right)\{1+r_{2,n}(h-s)\},
\]
where $\varphi(t)=(2\pi)^{-1/2}e^{-t^2/2}$ and
\[
\sup_{s\in W_n}|r_{1,n}(s)|\to0,
\qquad
\sup_{s\in W_n}|r_{2,n}(h-s)|\to0,
\]
uniformly over admissible $h$. Hence for $s\in W_n$,
\[
p_1(s) p_2(h-s)
=
f_n(s)\,(1+o(1)),
\qquad
f_n(s):=
\frac{1}{\sigma_1\sigma_2}
\varphi\!\left(\frac{s-\mu_1}{\sigma_1}\right)
\varphi\!\left(\frac{h-s-\mu_2}{\sigma_2}\right),
\]
where the $o(1)$ term is uniform in $s\in W_n$ (and in $h$). Summing over $s\in W_n$
yields
\begin{equation}\label{eq:main_sum_gauss_proxy}
\sum_{s\in W_n}p_1(s) p_2(h-s)
=
(1+o(1))\sum_{s\in W_n} f_n(s).
\end{equation}

\smallskip
\noindent
\textbf{Tail contribution outside \(W_n\).}
Recall that \(p_2(t)=\bbP(X_2=t)\). By a change of variables,
\[
\sum_{s\in W_n} p_2(h-s)
=
\sum_{s\in W_n}\bbP(X_2=h-s)
=
\bbP\big(X_2\in h-W_n\big),
\]
and similarly,
\[
\sum_{s\notin W_n} p_2(h-s)
=
\bbP\big(X_2\notin h-W_n\big).
\]
Since \(W_n=\{s\in\bbZ:\ |s-\mu_1|\le 2x_n\sigma_n\}\), we have
\[
h-W_n=\{t\in\bbZ:\ |(h-t)-\mu_1|\le 2x_n\sigma_n\}.
\]
Using \(\mu_n=\mu_1+\mu_2\) and \(h-\mu_n=x_n\sigma_n\), for any integer \(t\),
\[
(h-t)-\mu_1 = (h-\mu_n) - (t-\mu_2) = x_n\sigma_n - (t-\mu_2).
\]
Therefore,
\[
t\in h-W_n
\iff
|x_n\sigma_n-(t-\mu_2)|\le 2x_n\sigma_n
\iff
t-\mu_2\in[-x_n\sigma_n,\,3x_n\sigma_n].
\]
Consequently,
\[
\sum_{s\in W_n} p_2(h-s)
=
\bbP\big(X_2-\mu_2\in[-x_n\sigma_n,\,3x_n\sigma_n]\big),
\]
and
\[
\sum_{s\notin W_n} p_2(h-s)
=
\bbP\big(X_2-\mu_2\notin[-x_n\sigma_n,\,3x_n\sigma_n]\big)
\le
\bbP\big(|X_2-\mu_2|\ge x_n\sigma_n\big).
\]
In the balanced case \(\sigma_2\asymp\sigma_n\), hence \(x_n\sigma_n\asymp x_n\sigma_2\).
By Bernstein (or Hoeffding) inequality for binomial variables, there exists \(c>0\) such that
\[
\bbP\big(|X_2-\mu_2|\ge x_n\sigma_n\big)\le 2\exp(-c x_n^2)=o(1),
\]
since \(x_n\to\infty\). Thus
\[
\sum_{s\in W_n} p_2(h-s)=1-o(1),
\qquad
\sum_{s\notin W_n} p_2(h-s)=o(1).
\]

To control contribution from the set $W_n$, 
$$
\sum_{s\in W_n} p_1(s)p_2(h-s) \geq \left(\min_{s\in W_n}p_1(s)\right) \left(\sum_{s\in W_n}  p_2(h-s) \right)  \geq (1-o(1))\left(\min_{s\in W_n}p_1(s)\right). 
$$
To control contribution outside the set $W_n$, 
$$
\sum_{s\notin W_n} p_1(s)p_2(h-s) \leq  \left(\max_{s\notin W_n}p_1( s) \right)  \left(\sum_{s\notin W_n} p_2(h-s) \right)\leq o(1) \left(\max_{s\notin W_n}p_1( s) \right).  
$$
Let $N:=n-d$ and $p:=\alpha_n$, so that $X_1\sim\mathrm{Bin}(N,p)$. For $0\le s\le N-1$,
\[
\frac{p_1(s+1)}{p_1(s)}=\frac{N-s}{s+1}\cdot\frac{p}{1-p}.
\]
Let $m=\lfloor (N+1)p\rfloor$ be a mode. Then $p_1(s)$ is nondecreasing for $s\le m$
and nonincreasing for $s\ge m$. Since $|m-\mu_1|\le 1$ and $2x_n\sigma_n\to\infty$,
we have $m\in W_n=\{a_n,\dots,b_n\}$ for all large $n$. Hence
\[
\max_{s\notin W_n}p_1(s)
=\max\{p_1(a_n-1),p_1(b_n+1)\}
\le \min\{p_1(a_n),p_1(b_n)\}
=\min_{s\in W_n}p_1(s).
\]
As a result, 
$$
\frac{\sum_{s\notin W_n} p_1(s)p_2(h-s)}{\sum_{s\in W_n} p_1(s)p_2(h-s)}
\leq \frac{o(1) \left(\max_{s\notin W_n}p_1( s) \right)}{ (1-o(1)) \big(\min_{s\in W_n}p_1(s)\big)} = o(1). 
$$

Consequently,
\begin{equation}
\mathbb P(X=h) \sim \sum_{s\in W_n}p_1(s) p_2(h-s)
=
(1+o(1))\sum_{s\in W_n} f_n(s).
\end{equation}

\smallskip
\noindent\textbf{Discretization on $W_n$: sum versus integral.}
Write the integer window as
\[
W_n=\{a_n,a_n+1,\dots,b_n\},
\qquad
a_n:=\big\lceil \mu_1-2x_n\sigma_n\big\rceil,\quad
b_n:=\big\lfloor \mu_1+2x_n\sigma_n\big\rfloor,
\]
and identify $W_n$ with the real interval $[a_n,b_n]$ when it appears as the
domain of integration.

Partition $[a_n,b_n]$ into unit intervals $[s,s+1]$ for $s=a_n,\dots,b_n-1$.
Then
\[
\int_{W_n} f_n(u)\,du
=
\sum_{s=a_n}^{b_n-1}\int_{s}^{s+1} f_n(u)\,du,
\]
so
\[
\sum_{s\in W_n} f_n(s)-\int_{W_n} f_n(u)\,du
=
\sum_{s=a_n}^{b_n-1}\Big(f_n(s)-\int_s^{s+1}f_n(u)\,du\Big)+f_n(b_n).
\]
For each integer $s$, by the fundamental theorem of calculus,
\[
f_n(s)-f_n(u)= -\int_s^{u} f_n'(t)\,dt,
\qquad u\in[s,s+1],
\]
and hence
\[
\Big|f_n(s)-\int_s^{s+1}f_n(u)\,du\Big|
=
\Big|\int_s^{s+1}\bigl(f_n(s)-f_n(u)\bigr)\,du\Big|
\le
\int_s^{s+1}\int_s^{u}|f_n'(t)|\,dt\,du
\le
\int_s^{s+1}|f_n'(t)|\,dt.
\]
Summing over $s=a_n,\dots,b_n-1$ yields the deterministic bound
\begin{equation}\label{eq:disc_Wn_basic}
\Bigg|\sum_{s\in W_n} f_n(s)-\int_{W_n} f_n(u)\,du\Bigg|
\le
f_n(a_n)+f_n(b_n)+\int_{W_n}|f_n'(u)|\,du.
\end{equation}

Next write
\[
A(u):=\frac{u-\mu_1}{\sigma_1},
\qquad
B(u):=\frac{h-u-\mu_2}{\sigma_2},
\]
so that
\[
f_n(u)=\frac{1}{\sigma_1\sigma_2}\varphi(A(u))\varphi(B(u)),
\qquad
f_n'(u)=f_n(u)\Big(-\frac{A(u)}{\sigma_1}+\frac{B(u)}{\sigma_2}\Big),
\]
and therefore
\[
|f_n'(u)|
\le
\frac{1}{\sigma_1}f_n(u)|A(u)|
+
\frac{1}{\sigma_2}f_n(u)|B(u)|.
\]
Integrating over $W_n$ and extending to $\mathbb R$ gives
\begin{equation}\label{eq:disc_Wn_fnprime}
\int_{W_n}|f_n'(u)|\,du
\le
\frac{1}{\sigma_1}\int_{\mathbb R}f_n(u)|A(u)|\,du
+
\frac{1}{\sigma_2}\int_{\mathbb R}f_n(u)|B(u)|\,du.
\end{equation}
The measure $f_n(u)\,du$ is proportional to a Gaussian density in $u$ with variance
$\tau_n^2=(\sigma_1^{-2}+\sigma_2^{-2})^{-1}=\sigma_1^2\sigma_2^2/\sigma_n^2$ and mean
$m_n=\mu_1+\frac{\sigma_1^2}{\sigma_n^2}(h-\mu_n)$.
Consequently, there exists an absolute constant $C>0$ such that
\[
\int_{\mathbb R}f_n(u)|A(u)|\,du
+
\int_{\mathbb R}f_n(u)|B(u)|\,du
\le
C(1+|x_n|)\int_{\mathbb R} f_n(u)\,du.
\]
In the balanced case $\sigma_1\asymp\sigma_2\asymp\sigma_n$, combining this with
\eqref{eq:disc_Wn_fnprime} yields
\[
\int_{W_n}|f_n'(u)|\,du
\le
C\,\frac{1+|x_n|}{\sigma_n}\int_{\mathbb R} f_n(u)\,du
=
o\!\left(\int_{\mathbb R} f_n(u)\,du\right),
\]
since $(1+|x_n|)/\sigma_n=o(1)$.

Finally, for $u\in\{a_n,b_n\}$ we have $|u-\mu_1|\ge 2x_n\sigma_n-O(1)$, hence
$|A(u)|\ge c x_n$ for some constant $c>0$ and therefore
\[
f_n(a_n)+f_n(b_n)
\le
\frac{C}{\sigma_1\sigma_2}e^{-c x_n^2}
=
o\!\left(\int_{\mathbb R} f_n(u)\,du\right),
\]
because $\int_{\mathbb R} f_n(u)\,du=\sigma_n^{-1}\varphi(x_n)\asymp \sigma_n^{-1}e^{-x_n^2/2}$
and $c>1/2$ can be chosen in the balanced case.
Plugging these bounds into \eqref{eq:disc_Wn_basic} gives
\[
\sum_{s\in W_n} f_n(s)
=
\int_{W_n} f_n(u)\,du
+
o\!\left(\int_{\mathbb R} f_n(u)\,du\right).
\]

Using the similar idea for $\sum{s\in W_n}$, we can show that 
$$
\frac{\int_{W_n^c} f_n(s) ds}{\int_{W_n} f_n(s) ds} = o(1). 
$$
As a result, we have
$$
\mathbb P(X=h)
\sim \sum_{s\in W_n}p_1(s) p_2(h-s)
\sim \sum_{s\in W_n} f_n(s)
\sim \int_{W_n} f_n(u)\,du
\sim \int_{\mathbb R} f_n(u)\,du.
$$

\medskip
\noindent\emph{Unbalanced case.}
Assume without loss of generality that $\sigma_1^2/\sigma_n^2\to0$,
so $\sigma_n\sim\sigma_2$ and $\sigma_2^2\to\infty$.

Using the convolution identity,
\[
\mathbb P(X=h)
=
\mathbb E\big[\mathbb P(X_2=h-X_1)\mid X_1\big].
\]
Write $T:=X_1-\mu_1$ and $m:=h-\mu_1$, so $h-X_1=m-T$ and $m-\mu_2=h-\mu_n$.

Choose $\eta_n\downarrow0$ such that $\eta_n\sigma_2/\sigma_1\to\infty$
(e.g.\ $\eta_n=(\sigma_1/\sigma_2)^{1/2}$), and define
$\mathcal E_n=\{|T|\le\eta_n\sigma_2\}$.
Then $\mathbb P(\mathcal E_n^c)=o(1)$ by Chebyshev’s inequality.

On $\mathcal E_n$,
\[
|(m-T)-\mu_2|
\le |h-\mu_n|+|T|
=
o(\sigma_2^{4/3}),
\]
so Lemma~\ref{lem:lclt_bin} applies uniformly and yields
\[
\mathbb P(X_2=m-T)
=
\frac{1}{\sqrt{2\pi}\sigma_2}
\exp\!\left(-\frac{(m-T-\mu_2)^2}{2\sigma_2^2}\right)(1+o(1)).
\]
Moreover, write $x_2 := (h-\mu_n)/\sigma_2$ and $U:=T/\sigma_2$. On $\mathcal E_n$,
Lemma~\ref{lem:lclt_bin} gives
\[
\mathbb P(X_2=m-T)
=
\frac{1}{\sigma_2}\,\varphi(x_2-U)\,(1+o(1)),
\]
uniformly, where $\varphi(t)=(2\pi)^{-1/2}e^{-t^2/2}$.
Since $|\varphi(\cdot)|\le (2\pi)^{-1/2}$, taking expectations and using
$\mathbb P(\mathcal E_n^c)=o(1)$ yields
\[
\mathbb P(X=h)
=
(1+o(1))\,\frac{1}{\sigma_2}\,\mathbb E\big[\varphi(x_2-U)\big]
+ o\!\left(\frac{1}{\sigma_2}\right).
\]

We next show that $\mathbb E[\varphi(x_2-U)]=\varphi(x_2)\,(1+o(1))$ uniformly for
$|x_2|=o(\sigma_2^{1/3})$. Using Taylor's theorem,
\[
\varphi(x_2-U)=\varphi(x_2)-U\varphi'(x_2)+\frac{U^2}{2}\varphi''(\xi),
\]
for some $\xi$ between $x_2$ and $x_2-U$. Taking expectations and using $\mathbb E[U]=0$,
\[
\mathbb E[\varphi(x_2-U)]-\varphi(x_2)
=
\frac{1}{2}\,\mathbb E\big[U^2\,\varphi''(\xi)\big].
\]
Since $\varphi''(t)=(t^2-1)\varphi(t)$ and $\sup_{|v|\le |U|} \varphi(x_2-v)\le \varphi(x_2-|U|)$,
we have the crude bound $|\varphi''(\xi)|\le C(1+x_2^2+U^2)\,\varphi(x_2-U)$ for an absolute constant $C$.
On $\mathcal E_n$, $|U|\le \eta_n\to0$, so $\varphi(x_2-U) = \varphi(x_2)\,(1+o(1))$ uniformly,
and therefore
\[
\left|\mathbb E[\varphi(x_2-U)]-\varphi(x_2)\right|
\le C\,\mathbb E[U^2]\,(1+x_2^2+o(1))\,\varphi(x_2).
\]
But $\mathbb E[U^2]=\Var(T)/\sigma_2^2=\sigma_1^2/\sigma_2^2=o(1)$, and $x_2^2=o(\sigma_2^{2/3})$,
so the right-hand side is $o(1)\varphi(x_2)$. Hence
\[
\mathbb E[\varphi(x_2-U)] = \varphi(x_2)\,(1+o(1)).
\]
Combining the above displays gives
\[
\mathbb P(X=h)
=
\frac{1}{\sigma_2}\,\varphi\!\left(\frac{h-\mu_n}{\sigma_2}\right)\,(1+o(1))
=
\frac{1}{\sqrt{2\pi}\sigma_2}\exp\!\left(-\frac{(h-\mu_n)^2}{2\sigma_2^2}\right)\,(1+o(1)).
\]
Finally, since $\sigma_n\sim\sigma_2$, the last expression is also
\[
\frac{1}{\sqrt{2\pi}\sigma_n}\exp\!\left(-\frac{(h-\mu_n)^2}{2\sigma_n^2}\right)\,(1+o(1)).
\]

\paragraph{Upper tail.}
Write $X=\sum_{i=1}^{n} Y_i$ as a sum of independent Bernoulli random variables,
with $0\le Y_i\le 1$, $\bbE X=\mu_n$, and $\Var(X)=\sigma_n^2\to\infty$.
For sums of independent uniformly bounded random variables, a moderate deviation
normal approximation holds: uniformly for $x=o(\sigma_n^{1/3})$,
\[
\bbP\!\left(\frac{X-\mu_n}{\sigma_n}\ge x\right)
=
\bigl(1-\Phi(x)\bigr)\,(1+o(1)).
\]
This follows, for example, from  \cite{petrov2012sums}. 
In particular, if $x\to\infty$ while still $x=o(\sigma_n^{1/3})$, then by Mills' ratio,
\[
1-\Phi(x)
\sim
\frac{1}{\sqrt{2\pi}}\frac{1}{x}e^{-x^2/2},
\]
and hence
\[
\mathbb P(X\ge \mu_n+x\sigma_n)
\sim
\frac{1}{\sqrt{2\pi}}\frac{1}{x}e^{-x^2/2}.
\]
\end{proof}

\begin{lem}[Binomial local CLT on a moderate window]\label{lem:lclt_bin}
Let $Y\sim\mathrm{Bin}(m,p)$ with $p=p_m\in(0,1)$ possibly depending on $m$.
Let $\mu=mp$ and $\sigma^2=mp(1-p)$, and assume $\sigma^2\to\infty$ as $m\to\infty$.
Then uniformly over integers $k$ satisfying
\[
|k-\mu|=o(\sigma^{4/3}),
\]
we have
\[
\mathbb P(Y=k)
=
\frac{1}{\sqrt{2\pi}\,\sigma}\exp\!\left(-\frac{(k-\mu)^2}{2\sigma^2}\right)\,(1+o(1)).
\]
\end{lem}

\begin{proof}
Fix an integer $k$ and write $q:=k/m\in[0,1]$. The binomial pmf can be written as
\[
\mathbb P(Y=k)=\binom{m}{k}p^k(1-p)^{m-k}.
\]
We first show that on the stated window, $k\to\infty$ and $m-k\to\infty$ uniformly.
Indeed, $\sigma^2=mp(1-p)\to\infty$ implies both $mp\to\infty$ and $m(1-p)\to\infty$
(since $p(1-p)\le \min\{p,1-p\}$). Moreover,
\[
\frac{\sigma^{4/3}}{\mu}=\frac{(mp(1-p))^{2/3}}{mp}=\frac{(1-p)^{2/3}}{(mp)^{1/3}}\to0,
\qquad
\frac{\sigma^{4/3}}{m-\mu}=\frac{(mp(1-p))^{2/3}}{m(1-p)}=\frac{p^{2/3}}{(m(1-p))^{1/3}}\to0,
\]
so $|k-\mu|=o(\sigma^{4/3})$ implies $k\sim \mu\to\infty$ and $m-k\sim m-\mu\to\infty$
uniformly over the window. In particular, $q\in(0,1)$ and $q\to p$ uniformly.

\smallskip
\noindent\textbf{Step 1: Stirling with uniform relative error.}
By Stirling's formula with remainder,
\[
n! = \sqrt{2\pi n}\left(\frac{n}{e}\right)^n\left(1+O\left(\frac{1}{n}\right)\right),
\]
uniformly for $n\to\infty$. Applying this to $m!,k!, (m-k)!$ and using that
$k\wedge(m-k)\to\infty$ uniformly on the window gives
\begin{align*}
\binom{m}{k}
&=
\frac{m!}{k!(m-k)!}
=
\frac{1}{\sqrt{2\pi m q(1-q)}}
\cdot
\frac{m^m}{k^k(m-k)^{m-k}}
\cdot
\big(1+o(1)\big),
\end{align*}
uniformly over $|k-\mu|=o(\sigma^{4/3})$. Therefore,
\begin{align*}
\mathbb P(Y=k)
&=
\frac{1}{\sqrt{2\pi m q(1-q)}}
\left(\frac{m^m}{k^k(m-k)^{m-k}}\right)
p^k(1-p)^{m-k}\,
(1+o(1)).
\end{align*}
Taking logs of the non-prefactor part and using $k=mq$ yields the standard identity
\[
\log\left[
\left(\frac{m^m}{k^k(m-k)^{m-k}}\right)p^k(1-p)^{m-k}
\right]
=
-\,m\Big[q\log\!\frac{q}{p}+(1-q)\log\!\frac{1-q}{1-p}\Big]
= -m\,D(q\|p),
\]
where $D(q\|p)=\mathrm{KL}(\mathrm{Bern}(q)\,\|\,\mathrm{Bern}(p))$.
Hence
\begin{equation}\label{eq:bin_pmf_KL_form}
\mathbb P(Y=k)
=
\frac{1}{\sqrt{2\pi m q(1-q)}}\,
\exp\!\big(-mD(q\|p)\big)\,
(1+o(1)),
\end{equation}
uniformly on the window.

\smallskip
\noindent\textbf{Step 2: Expand the KL term in the moderate window.}
Let $\Delta:=k-\mu$ so that $q-p=\Delta/m$. Note that
\[
\frac{|q-p|}{p(1-p)}
=\frac{|\Delta|}{mp(1-p)}
=\frac{|\Delta|}{\sigma^2}
=o(\sigma^{4/3}/\sigma^2)
=o(\sigma^{-2/3})\to0,
\]
uniformly over $|k-\mu|=o(\sigma^{4/3})$. Hence for any fixed $c\in(0,1)$,
for all sufficiently large $m$ we have uniformly on the window
$|q-p|\le c\,p(1-p)$.
Applying Corollary~\ref{cor:KL_quadratic} with the substitution $(p,q)\mapsto(q,p)$,
we obtain
\[
D(q\|p)
=
\frac{(q-p)^2}{2\,p(1-p)}
+O\!\left(\frac{|q-p|^3}{p^2(1-p)^2}\right),
\]
uniformly on the window, where the $O(\cdot)$ constant depends only on $c$.
Multiplying by $m$ gives
\[
mD(q\|p)
=
\frac{\Delta^2}{2\sigma^2}
+O\!\left(\frac{|\Delta|^3}{\sigma^4}\right).
\]
Since $|\Delta|=o(\sigma^{4/3})$, we have $|\Delta|^3/\sigma^4=o(1)$, and therefore
\begin{equation}\label{eq:KL_expansion_window}
mD(q\|p)=\frac{(k-\mu)^2}{2\sigma^2}+o(1),
\end{equation}
uniformly on the window.

\smallskip
\noindent\textbf{Step 3: Simplify the prefactor.}
We claim that $mq(1-q)\sim mp(1-p)=\sigma^2$ uniformly on the window.
Indeed,
\[
|q(1-q)-p(1-p)|
= |(q-p)(1-q-p)|
\le 2|q-p|,
\]
for $q$ close to $p$. Also,
\[
\frac{|q-p|}{p(1-p)}
=
\frac{|\Delta|/m}{p(1-p)}
=
\frac{|\Delta|}{\sigma^2}
=
o(\sigma^{4/3}/\sigma^2)
=
o(\sigma^{-2/3})
\to 0,
\]
so $|q(1-q)-p(1-p)|=o(p(1-p))$ and hence $q(1-q)=p(1-p)(1+o(1))$.
Therefore
\begin{equation}\label{eq:prefactor}
\frac{1}{\sqrt{2\pi m q(1-q)}}
=
\frac{1}{\sqrt{2\pi}\,\sigma}\,(1+o(1)),
\end{equation}
uniformly on the window.

\smallskip
\noindent\textbf{Step 4: Combine.}
Plugging \eqref{eq:KL_expansion_window} and \eqref{eq:prefactor} into
\eqref{eq:bin_pmf_KL_form} gives
\[
\mathbb P(Y=k)
=
\frac{1}{\sqrt{2\pi}\,\sigma}\,
\exp\!\left(-\frac{(k-\mu)^2}{2\sigma^2}\right)\,
(1+o(1)),
\]
uniformly over integers $k$ satisfying $|k-\mu|=o(\sigma^{4/3})$.
\end{proof}

The following lemma provides the foundation for the proof of Lemma ~\ref{lem:binom_approx}. 
\begin{lem}\label{lem:KL_Taylor}
Let $q\in(0,1)$ and define $D(p\|q)=\mathrm{KL}(\mathrm{Bern}(p)\,\|\,\mathrm{Bern}(q))$.
Fix $c\in(0,1)$. Uniformly for $p\in(0,1)$ satisfying $|p-q|\le c\,q(1-q)$,
\[
D(p\|q)
=
\frac{(p-q)^2}{2\,q(1-q)}
+ \frac{2q-1}{6\,q^2(1-q)^2}(p-q)^3
+ O\!\left(\frac{|p-q|^4}{q^3(1-q)^3}\right),
\]
where the $O(\cdot)$ constant depends only on $c$.
\end{lem}

The expansion follows from a straightforward Taylor expansion of $D(p\|q)$ with respect to $p$ around $p=q$.
\begin{rmk}
Lemma~\ref{lem:KL_Taylor} remains valid in the sparse regime $p,q=o(1)$.
However, when $q\to0$, the quadratic term provides a leading approximation
only in a relative neighborhood of $q$, i.e.\ when $|p-q|/q\to0$.
This regime is made precise in Corollary~\ref{cor:KL_quadratic}.
\end{rmk}

\begin{coro}\label{cor:KL_quadratic}
Fix $c\in(0,1)$. Uniformly for $q\in(0,1)$ and $p\in(0,1)$ such that
$|p-q|\le c\,q(1-q)$,
\[
D(p\|q)
=
\frac{(p-q)^2}{2\,q(1-q)}
\left(1+O\!\left(\frac{|p-q|}{q(1-q)}\right)\right),
\]
equivalently,
\[
D(p\|q)
=
\frac{(p-q)^2}{2\,q(1-q)}
+O\!\left(\frac{|p-q|^3}{q^2(1-q)^2}\right),
\]
where the $O(\cdot)$ constant depends only on $c$.
In particular, if $\frac{|p-q|}{q(1-q)}\to0$, then
\[
D(p\|q)=\frac{(p-q)^2}{2\,q(1-q)}\,(1+o(1)).
\]
\end{coro}

\end{document}